\pdfoutput=1
\RequirePackage{ifpdf}
\ifpdf 
\documentclass[pdftex]{sigma}
\else
\documentclass{sigma}
\fi

\numberwithin{equation}{section}

\newtheorem{Theorem}{Theorem}[section]

\newtheorem{Conjecture}[Theorem]{Conjecture}
\newtheorem{Lemma}[Theorem]{Lemma}
\newtheorem{Proposition}[Theorem]{Proposition}
 { \theoremstyle{definition}
\newtheorem{Definition}[Theorem]{Definition}

\newtheorem{Remark}[Theorem]{Remark} }

\usepackage{tikz} 
\usepackage[new]{old-arrows} 

\newcommand{\fonc}[5]{ 
 \begin{array}{crll}#1\colon & #2 & \rightarrow & #3 \\ %
 &#4 &\mapsto & #5 %
 \end{array}}

\newcommand{\exposantGauche}[2]{{\vphantom{#2}}^{#1}#2} 

\DeclareMathOperator{\End}{End}
\DeclareMathOperator{\SLF}{SLF}
\DeclareMathOperator{\vect}{vect}

\begin{document}

\allowdisplaybreaks

\newcommand{\arXivNumber}{1805.00924}

\renewcommand{\PaperNumber}{077}

\FirstPageHeading

\ShortArticleName{Modular Group Representations in Combinatorial Quantization}

\ArticleName{Modular Group Representations in Combinatorial\\ Quantization with Non-Semisimple Hopf Algebras}

\Author{Matthieu FAITG}

\AuthorNameForHeading{M.~Faitg}

\Address{IMAG, Univ Montpellier, CNRS, Montpellier, France}
\Email{\href{mailto:matthieu.faitg@gmail.com}{matthieu.faitg@gmail.com}}

\ArticleDates{Received February 02, 2019, in final form September 24, 2019; Published online October 03, 2019}

\Abstract{Let $\Sigma_{g,n}$ be a compact oriented surface of genus $g$ with $n$ open disks removed. The algebra $\mathcal{L}_{g,n}(H)$ was introduced by Alekseev--Grosse--Schomerus and Buffenoir--Roche and is a combinatorial quantization of the moduli space of flat connections on $\Sigma_{g,n}$. Here we focus on the two building blocks $\mathcal{L}_{0,1}(H)$ and $\mathcal{L}_{1,0}(H)$ under the assumption that the gauge Hopf algebra $H$ is finite-dimensional, factorizable and ribbon, but not necessarily semisimple. We construct a projective representation of $\mathrm{SL}_2(\mathbb{Z})$, the mapping class group of the torus, based on $\mathcal{L}_{1,0}(H)$ and we study it explicitly for $H = \overline{U}_q(\mathfrak{sl}(2))$. We also show that it is equivalent to the representation constructed by Lyubashenko and Majid.}


\Classification{16T05; 81R05}

\section{Introduction}

Let $\Sigma_{g,n}$ be a compact oriented surface of genus $g$ with $n$ open disks removed and let $G$ be a connected, simply-connected Lie group. Fock and Rosly defined a Poisson structure on the space of flat $G$-connections on $\Sigma_{g,n} {\setminus} D$ (where $D$ is an open disk) in a combinatorial way, by using a description of the surface as a ribbon graph, see \cite{FR, FockRosly}. This construction is another formulation of the Atiyah--Bott--Goldman Poisson structure on the character variety \cite{AB, goldman}. The algebra $\mathcal{L}_{g,n}$ is an associative non-commutative algebra which is a combinatorial quantization of the algebra of functions on the space of flat $G$-connections, by deformation of the Fock--Rosly Poisson structure.

These algebras were introduced and studied by Alekseev--Grosse--Schomerus \cite{alekseev, AGS, AGS2, AS} and Buffenoir--Roche \cite{BR, BR2}. They replaced the Lie group $G$ by the quantum group $U_q(\mathfrak{g})$, with $\mathfrak{g} = \mathrm{Lie}(G)$, and defined non-commutative relations between functions on connections (also called gauge fields) \textit{via} matrix equations involving the $R$-matrix of $U_q(\mathfrak{g})$. More precisely, in the Fock--Rosly graph description, a flat connection is described by the collection of its holonomies along the edges of the graph. For instance, with the canonical graph shown in Fig.~\ref{figureSurfaces}, a flat $G$-connection on $\Sigma_{1,0} {\setminus} D$ is represented by $(x_b, x_a) \in G \times G$. The gauge group $G$ acts on connections by the adjoint action and dually on gauge fields (on the right). The algebra of gauge fields is generated by matrix coefficients $\overset{V}{A}{^i_j}$, $\overset{V}{B}{^i_j}$ defined by $\overset{V}{A}{^i_j}(x_b, x_a) = e^i(x_a \cdot e_j)$ and $\overset{V}{B}{^i_j}(x_b, x_a) = e^i(x_b \cdot e_j)$, where $V$ is a $G$-module with basis $(e_i)$ and dual basis $(e^i)$. Then the algebra $\mathcal{L}_{1,0}$ is generated by the coefficients of the matrices $\overset{I}{A}$, $\overset{I}{B}$, where $I$ is now a representation of $U_q(\mathfrak{g})$, modulo certain matrix relations involving the $R$-matrix, see Definition~\ref{definitionL10} (also see Remark~\ref{remarqueProductGaugeFields} for an explicit formula). A~connection is now $x_b \otimes x_a \in U_q(\mathfrak{g})^{\otimes 2}$. The gauge algebra $U_q(\mathfrak{g})$ acts on gauge fields on the right by the adjoint action. These relations give a~non-commutative $U_q(\mathfrak{g})$-equivariant multiplication between gauge fields.

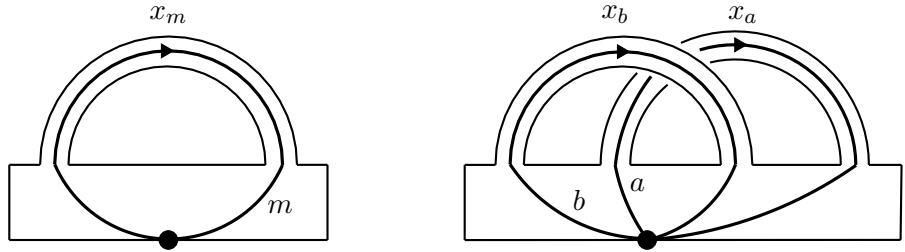
\begin{figure}[h]\centering
\begin{tikzpicture} 
\draw [shift={(7.116832178429767,6)},line width=0.8pt] plot[domain=0:3.141592653589793,variable=\t]({1*1.7075037526311627*cos(\t r)+0*1.7075037526311627*sin(\t r)},{0*1.7075037526311627*cos(\t r)+1*1.7075037526311627*sin(\t r)});
\draw [shift={(7.095787711044769,6)},line width=0.8pt] plot[domain=0:3.141592653589793,variable=\t]({1*1.309575195744216*cos(\t r)+0*1.309575195744216*sin(\t r)},{0*1.309575195744216*cos(\t r)+1*1.309575195744216*sin(\t r)});
\draw [line width=0.8pt] (5,6)-- (5,5);
\draw [line width=0.8pt] (5,5)-- (9.23,5);
\draw [line width=0.8pt] (5,6)-- (5.409328425798604,6);
\draw [line width=0.8pt] (8.82433593106093,6)-- (9.23,6);
\draw [line width=0.8pt] (5.786212515300553,6)-- (8.405362906788985,6);
\draw [line width=0.8pt] (9.23,6)-- (9.23,5);
\draw [shift={(7.119918427226924,6.666445939168958)},line width=1.2pt] plot[domain=3.551793132488248:4.709437541739214,variable=\t]({1*1.6536205027507724*cos(\t r)+0*1.6536205027507724*sin(\t r)},{0*1.6536205027507724*cos(\t r)+1*1.6536205027507724*sin(\t r)});
\draw [shift={(7.11991842722693,6.666445939168957)},line width=1.2pt] plot[domain=3.551793132488248:4.709437541739214,variable=\t]({-0.9999825780704314*1.6536205027507724*cos(\t r)+0.005902843011073883*1.6536205027507724*sin(\t r)},{0.005902843011073883*1.6536205027507724*cos(\t r)+0.9999825780704314*1.6536205027507724*sin(\t r)});
\draw [shift={(7.117958895052926,6.0025234526962805)},line width=1.2pt] plot[domain=-0.002951438645475868:3.1386412149443177,variable=\t]({1*1.5144846844083053*cos(\t r)+0*1.5144846844083053*sin(\t r)},{0*1.5144846844083053*cos(\t r)+1*1.5144846844083053*sin(\t r)});
\draw [fill=black] (7.115,5) circle (3.5pt);
\draw [color=black] (7.115,8) node {$x_m$};
\draw[color=black] (8.6, 5.44051180844409) node {$m$};
\draw [fill=black,shift={(7.074992707952773,7.5163985354139395)},rotate=270] (0,0) ++(0 pt,3pt) -- ++(2.598076211353316pt,-4.5pt)--++(-5.196152422706632pt,0 pt) -- ++(2.598076211353316pt,4.5pt);
\end{tikzpicture} ~~~~~~~~~~~~
\begin{tikzpicture}
\draw [shift={(7.116832178429767,6)},line width=0.8pt] plot[domain=0:3.141592653589793,variable=\t]({1*1.7075037526311627*cos(\t r)+0*1.7075037526311627*sin(\t r)},{0*1.7075037526311627*cos(\t r)+1*1.7075037526311627*sin(\t r)});
\draw [shift={(7.095787711044769,6)},line width=0.8pt] plot[domain=0:3.141592653589793,variable=\t]({1*1.309575195744216*cos(\t r)+0*1.309575195744216*sin(\t r)},{0*1.309575195744216*cos(\t r)+1*1.309575195744216*sin(\t r)});\draw [shift={(8.596936063235535,6)},line width=0.8pt] plot[domain=0:1.8118338027760237,variable=\t]({1*1.4030639367644646*cos(\t r)+0*1.4030639367644646*sin(\t r)},{0*1.4030639367644646*cos(\t r)+1*1.4030639367644646*sin(\t r)});
\draw [shift={(8.596936063235535,6)},line width=0.8pt] plot[domain=0:1.9894438844649243,variable=\t]({1*1.776642756206547*cos(\t r)+0*1.776642756206547*sin(\t r)},{0*1.776642756206547*cos(\t r)+1*1.776642756206547*sin(\t r)});\draw [shift={(8.596936063235535,6)},line width=0.8pt] plot[domain=2.2704846801922587:3.141592653589793,variable=\t]({1*1.3953711574291767*cos(\t r)+0*1.3953711574291767*sin(\t r)},{0*1.3953711574291767*cos(\t r)+1*1.3953711574291767*sin(\t r)});
\draw [shift={(8.596936063235535,6)},line width=0.8pt] plot[domain=2.3877511379579315:3.141592653589793,variable=\t]({1*1.789237366642512*cos(\t r)+0*1.789237366642512*sin(\t r)},{0*1.789237366642512*cos(\t r)+1*1.789237366642512*sin(\t r)});\draw [line width=0.8pt] (5,6)-- (5,5);
\draw [line width=0.8pt] (5,5)-- (10.794708954246598,5.00203777401749);
\draw [line width=0.8pt] (10.80786927095924,6)-- (10.794708954246598,5.00203777401749);
\draw [line width=0.8pt] (5,6)-- (5.409328425798604,6);
\draw [line width=0.8pt] (5.786212515300553,6)-- (6.8080282517522,6);
\draw [line width=0.8pt] (7.200815462007147,6)-- (8.405362906788985,6);
\draw [line width=0.8pt] (8.82433593106093,6)-- (10,6);
\draw [line width=0.8pt] (10.373578819442082,6)-- (10.80786927095924,6);
\draw [shift={(7.449580695168613,7.237726613631109)},line width=1.2pt] plot[domain=3.729595257137362:4.701699287460754,variable=\t]({1*2.2187359355090455*cos(\t r)+0*2.2187359355090455*sin(\t r)},{0*2.2187359355090455*cos(\t r)+1*2.2187359355090455*sin(\t r)});
\draw [shift={(7.1026186594434115,6.006993354832272)},line width=1.2pt] plot[domain=0:3.141592653589793,variable=\t]({1*1.4991378524730479*cos(\t r)+0*1.4991378524730479*sin(\t r)},{0*1.4991378524730479*cos(\t r)+1*1.4991378524730479*sin(\t r)});
\draw [shift={(9.412966893148852,6.4539788135790666)},line width=1.2pt] plot[domain=3.327560085751379:3.7729570768128284,variable=\t]({1*2.455301608888626*cos(\t r)+0*2.455301608888626*sin(\t r)},{0*2.455301608888626*cos(\t r)+1*2.455301608888626*sin(\t r)});
\draw [shift={(9.33211204597329,5.747220848195546)},line width=1.2pt] plot[domain=2.484254690839589:3.0336235085501446,variable=\t]({1*2.3423271124786536*cos(\t r)+0*2.3423271124786536*sin(\t r)},{0*2.3423271124786536*cos(\t r)+1*2.3423271124786536*sin(\t r)});
\draw [shift={(7.055663711022826,6.6244133629535815)},line width=1.2pt] plot[domain=4.936459140732812:5.903246179167029,variable=\t]({1*1.6651881573061873*cos(\t r)+0*1.6651881573061873*sin(\t r)},{0*1.6651881573061873*cos(\t r)+1*1.6651881573061873*sin(\t r)});
\draw [shift={(8.59002279079864,5.988921534198891)},line width=1.2pt] plot[domain=0.006881023462768673:1.8643557340986399,variable=\t]({1*1.6100153249433073*cos(\t r)+0*1.6100153249433073*sin(\t r)},{0*1.6100153249433073*cos(\t r)+1*1.6100153249433073*sin(\t r)});
\draw [shift={(7.024070995031567,10.499990747896216)},line width=1.2pt] plot[domain=4.785288706191459:5.326972569147958,variable=\t]({1*5.513782387193336*cos(\t r)+0*5.513782387193336*sin(\t r)},{0*5.513782387193336*cos(\t r)+1*5.513782387193336*sin(\t r)});
\draw [fill=black] (7.425668292444792,5.00085301330238) circle (3.5pt);
\draw[color=black] (6.520955110598566,5.535917869663414) node {$b$};
\draw[color=black] (7.296935596721761,5.733440175220822) node {$a$};
\draw[color=black] (7,8) node {$x_b$};
\draw[color=black] (8.7,8) node {$x_a$};
\draw [fill=black,shift={(7.083362211564692,7.506007527524644)},rotate=270] (0,0) ++(0 pt,3pt) -- ++(2.598076211353316pt,-4.5pt)--++(-5.196152422706632pt,0 pt) -- ++(2.598076211353316pt,4.5pt);
\draw [fill=black,shift={(8.643159148919558,7.598059773753824)},rotate=270] (0,0) ++(0 pt,3pt) -- ++(2.598076211353316pt,-4.5pt)--++(-5.196152422706632pt,0 pt) -- ++(2.598076211353316pt,4.5pt);
\end{tikzpicture}
\caption{Surfaces $\Sigma_{0,1} {\setminus} D$ and $\Sigma_{1,0}{\setminus} D$, with canonical curves and discrete connections.}
\label{figureSurfaces}
\end{figure}

These quantized algebras of functions $\mathcal{L}_{g,n}$ and their generalizations appear in various works of mathematics and mathematical physics. Let us indicate papers which are related to this work. In \cite{BFKB}, they introduce multitangles (which are diagrams encoding transformations of the graph and of the discrete connections) and use them to define a comultiplication on the discrete connections (dual to the product in $\mathcal{L}_{g,n}$), the holonomy of a path in the graph and the Wilson loops. In \cite{MW}, the axiomatic formulation of lattice gauge theory with Hopf algebras is given and it generalizes the foundational works cited above. Their formalism does not require semi-simplicity of the gauge algebra (except for some properties about the holonomy); however, contrarily to the present work, they do not use matrix coefficients, they work with an arbitrary graph filling the surface and the purpose of their paper is different (they do not discuss representations of mapping class groups). In~\cite{BZBJ}, the algebras~$\mathcal{L}_{g,n}$ are recovered in a categorical setting based on factorization homology; the case of the punctured torus is studied in~\cite{BJ}, recovering a version of~$\mathcal{L}_{1,0}$. In \cite{AGPS}, they study $\mathcal{L}_{1,0}$ with the super Hopf algebra~$\mathfrak{gl}(1|1)$ as well as the associated~$\mathrm{SL}_2(\mathbb{Z})$ action and they conjecture that our result on the equivalence with the Lyubashenko--Majid representation (Theorem~\ref{EquivalenceLMandSLF}) remains true with super Hopf algebras. \cite{CMR} also discuss related subjects.

The definition of $\mathcal{L}_{g,n}$ being purely algebraic, we can replace the gauge algebra $U_q(\mathfrak{g})$ by any ribbon Hopf algebra~$H$, thus obtaining an algebra $\mathcal{L}_{g,n}(H)$. In the foundational papers on combinatorial quantization, it is always assumed that $H$ is either $U_q(\mathfrak{g})$ with $q$ generic or a~semisimple truncation of $U_q(\mathfrak{g})$ at a root of unity (defined in the context of weak quasi-Hopf algebras). The moduli algebra, introduced in \cite{AGS, AGS2, AS}, is only defined when $H$ is modular since its definition uses the particular properties of the $S$-matrix in this case. Moreover, in~\cite{AS}, the representation theory of $\mathcal{L}_{g,n}(H)$ is investigated and a projective representation of the mapping class group of~$\Sigma_{g,n}$ on the moduli algebra is constructed, under the assumption that~$H$ is a~modular Hopf algebra (and in particular semisimple).

In this paper we assume that $H$ is a finite-dimensional, factorizable, ribbon Hopf algebra which is not necessarily semisimple, the guiding example being the restricted quantum group $H=\overline{U}_q(\mathfrak{sl}(2))$. We consider $\mathcal{L}_{0,1}(H)$, $\mathcal{L}_{1,0}(H)$ and we generalize to this setting the projective representation of the mapping class group of \cite{AS}; in this way we obtain a projective representation of $\mathrm{SL}_2(\mathbb{Z})$. The algebras $\mathcal{L}_{0,1}(H)$ and $\mathcal{L}_{1,0}(H)$ deserve particular interest because they are the building blocks of the theory, thanks to the Alekseev isomorphism,
\[ \mathcal{L}_{g,n}(H) \cong \mathcal{L}_{1,0}(H)^{\otimes g} \otimes \mathcal{L}_{0,1}(H)^{\otimes n}. \]
This isomorphism was stated in \cite{alekseev} for $H = U_q(\mathfrak{g})$, $q$ generic, but it can be generalized to our assumptions on $H$ (see \cite[Proposition~3.5]{F2}).

After having recalled the definition of $\mathcal{L}_{0,1}(H)$ and $\mathcal{L}_{1,0}(H)$ and the $H$-action on them, we prove that $\mathcal{L}_{0,1}(H) \cong H$ (Theorem~\ref{IsoRSD}) and that $\mathcal{L}_{1,0}(H)$ is isomorphic to the Heisenberg double of the dual Hopf algebra $\mathcal{O}(H)$ (Theorem~\ref{isoL10Heisenberg}). In particular, it implies that $\mathcal{L}_{1,0}(H)$ is (isomorphic to) a matrix algebra. We then define in Section~\ref{sectionRepInvariants} a representation of the algebra of invariants, $\mathcal{L}^{\mathrm{inv}}_{1,0}(H)$, on the space of symmetric linear forms $\SLF(H)$ (Theorem~\ref{repInv}). We will use this representation of the invariant elements to define a projective representation of $\mathrm{SL}_2(\mathbb{Z})$ on $\SLF(H)$.

The main ingredient of the construction of the mapping class group representation is the natural action of the mapping class group $\mathrm{MCG}(\Sigma_{1,0} {\setminus} D)$ on $\mathcal{L}_{1,0}(H)$, obtained by considering the action of $\mathrm{MCG}(\Sigma_{1,0} {\setminus} D)$ on $\pi_1(\Sigma_{1,0} {\setminus} D)$ and by replacing the loops representing elements in $\pi_1(\Sigma_{1,0} {\setminus} D)$ by the corresponding matrices in $\mathcal{L}_{1,0}(H)$, up to some normalization \cite{AS}. It turns out that $\mathrm{MCG}(\Sigma_{1,0} {\setminus} D)$ acts by automorphisms on $\mathcal{L}_{1,0}(H)$, and thus, since it is a matrix algebra, we get elements which implements these automorphisms by conjugation. Representing these elements on $\SLF(H)$, we get a projective representation of $\mathrm{SL}_2(\mathbb{Z})$; this is explained in Section~\ref{sectionModulaire}, with detailed proofs under our assumptions on $H$.

We show (Theorem~\ref{EquivalenceLMandSLF}) that this projective representation is equivalent to the Lyubashenko--Majid representation \cite{LM}. This gives a natural and geometrical interpretation of the latter, which was constructed by categorical methods.

Section~\ref{exempleUq} is devoted to the example of $\overline{U}_q = \overline{U}_q(\mathfrak{sl}(2))$ as gauge algebra. All the preliminary facts about $\overline{U}_q$ and the GTA basis, which is a suitable basis of $\SLF\big(\overline{U}_q\big)$, are available in \cite{F}. In Theorem~\ref{actionSL2ZArike} we give the explicit formulas for the action of $\mathrm{SL}_2(\mathbb{Z})$ on the GTA basis of~$\SLF\big(\overline{U}_q\big)$. The multiplication formulas in this basis (see \cite[Section~5]{F}, \cite{GT}), are the crucial tool to obtain the result. The structure of $\SLF\big(\overline{U}_q\big)$ under the action of $\mathrm{SL}_2(\mathbb{Z})$ is determined. Thanks to the equivalence with the Lyubashenko--Majid representation and the work of \cite{FGST}, the representation obtained in Section~\ref{SL2ZUq} is equivalent to the one studied in \cite{FGST}, which comes from logarithmic conformal field theory. Finally, in Section~\ref{sectionConjecture}, we formulate a conjecture about the structure of~$\SLF\big(\overline{U}_q\big)$ as a $\mathcal{L}^{\mathrm{inv}}_{1,0}\big(\overline{U}_q\big)$-module.

To sum up, the main results of this paper are:
\begin{itemize}\itemsep=0pt
\item[--] the construction of a projective representation of $\text{SL}_2(\mathbb{Z})$ (the mapping class group of the torus) on the space of symmetric linear forms on $H$ (Theorem~\ref{mainResult}),
\item[--] the equivalence of this representation to the one found by Lyubashenko and Majid in \cite{LM} (Theorem~\ref{EquivalenceLMandSLF}),
\item[--] in the case of $H = \overline{U}_q(\mathfrak{sl}(2))$, an explicit realization of this representation and its structure (Theorem~\ref{actionSL2ZArike} and~Theorem~\ref{thDecRep}).
\end{itemize}

In \cite{F2}, we generalize Theorems~\ref{mainResult} and~\ref{EquivalenceLMandSLF} to higher genus. The projective representation of the mapping class group is shown to be equivalent to that constructed by Lyubashenko using the coend of a ribbon category~\cite{lyu}. The advantage of combinatorial quantization is that the construction of the projective representation of the mapping class group is very explicit and that the particular features of the algebra $\mathcal{L}_{g,n}(H)$ are helpful to compute the formulas of the representation for a given~$H$ (see Section~\ref{SL2ZUq}). In particular, for $H = \overline{U}_q(\mathfrak{sl}(2))$, one can hope that the correspondence with logarithmic conformal field theory mentioned aboved still exists in higher genus. Let us also mention that the subalgebra of invariant elements $\mathcal{L}_{g,n}^{\mathrm{inv}}(H)$ is known to be related to skein theory~\cite{BFKB}; in fact, combinatorial quantization provides new representations of skein algebras (work in progress).

\textbf{Notations.} \label{Notations} If $A$ is an algebra, $V$ is a finite-dimensional $A$-module and $x \in A$, we denote by $\overset{V}{x} \in \End_{\mathbb{C}}(V)$ the representation of $x$ on the module $V$. More generally, if $X \in A^{\otimes n}$ and if $V_1, \dots , V_n$ are $A$-modules, we denote by $\overset{V_1 \dots V_n}{X}$ the representation of $X$ on $V_1 \otimes \dots \otimes V_n$. As in~\cite{CR}, we will use the abbreviation PIM for principal indecomposable module. Here we consider only finite-dimensional representations.

Let $M \in \mathrm{Mat}_m(\mathbb{C}) \otimes A=\mathrm{Mat}_m(A)$, namely $M$ is a matrix with coefficients in $A$. It can be written as $M = \sum_{i,j} E^i_j \otimes M^i_j$, where $E^i_j$ is the matrix with $1$ at the intersection of the $i$-th row and the $j$-th column and~$0$ elsewhere. More generally, every $L \in \mathrm{Mat}_{m_1}(\mathbb{C}) \otimes \dots \otimes \mathrm{Mat}_{m_l}(\mathbb{C}) \otimes A$ can be written as
\[ L = \sum_{i_1,j_1, \dots, i_l, j_l} E^{i_1}_{j_1} \otimes \dots \otimes E^{i_l}_{j_l} \otimes L^{i_1 \dots i_l}_{j_1 \dots j_l} \]
and the elements $L^{i_1 \dots i_l}_{j_1 \dots j_l} \in A$ are called the coefficients of~$L$.
If $f\colon A \to B$ is a morphism of algebras, then we define $f(L) \in \mathrm{Mat}_{m_1}(\mathbb{C}) \otimes \dots \otimes \mathrm{Mat}_{m_l}(\mathbb{C}) \otimes B$ by
\[ f(L) = \sum_{i_1,j_1, \dots, i_l, j_l} E^{i_1}_{j_1} \otimes \dots \otimes E^{i_l}_{j_l} \otimes f\bigl(L^{i_1 \dots i_l}_{j_1 \dots j_l}\bigr) \]
or equivalently $f(L)^{i_1 \dots i_l}_{j_1 \dots j_l} = f\bigl(L^{i_1 \dots i_l}_{j_1 \dots j_l}\bigr)$.

Let $M \in \mathrm{Mat}_m(\mathbb{C}) \otimes A$, $N \in \mathrm{Mat}_n(\mathbb{C}) \otimes A$. We embed $M$, $N$ in $\mathrm{Mat}_m(\mathbb{C}) \otimes \mathrm{Mat}_n(\mathbb{C}) \otimes A$ by
\[ M_1 = \sum_{i,j}E^i_j \otimes \mathbb{I}_n \otimes M^i_j, \qquad N_2 = \sum_{i,j} \mathbb{I}_m \otimes E^i_j \otimes N^i_j, \]
where $\mathbb{I}_k$ is the identity matrix of size $k$. In other words, $M_1 = M \otimes \mathbb{I}_n$, $N_2 = \mathbb{I}_m \otimes N$, $\otimes$~being the Kronecker product. The coefficients of $M_1$ and $N_2$ are respectively $(M_1)^{ik}_{jl} = M^i_j \delta^k_l$, $(N_2)^{ik}_{jl} = \delta^i_j N^k_l$ where $\delta^a_b$ is the Kronecker delta. Observe that $M_1N_2$ (resp. $N_2M_1$) contains all the possible products of coefficients of $M$ (resp.\ of~$N$) by coefficients of $N$ (resp.\ of~$M$): $(M_1N_2)^{ik}_{j\ell} = M^i_jN^k_{\ell}$ (resp. $(N_2M_1)^{ik}_{j\ell} = N^k_{\ell}M^i_j$). In particular, $M_1N_2 = N_2M_1$ if and only if the coefficients of $M$ commute with those of $N$. Now let $Q \in \mathrm{Mat}_{m}(\mathbb{C}) \otimes \mathrm{Mat}_{n}(\mathbb{C}) \otimes A$; then we denote $Q_{12} = Q$ and
\[
Q_{21} = \sum_{i,j,k,l} E^k_l \otimes E^i_j \otimes Q^{ik}_{jl} \in \mathrm{Mat}_n(\mathbb{C}) \otimes \mathrm{Mat}_m(\mathbb{C}) \otimes A.
\]
so that $(Q_{12})^{ac}_{bd} = Q^{ac}_{bd}$, $(Q_{21})^{ac}_{bd} = Q^{ca}_{db}$. Such notations are obviously generalized to bigger tensors.

In order to simplify notations, we will use implicit summations. First, we use Einstein's notation for the computations involving indices: when an index variable appears twice, one time in upper position and one time in lower position, it implicitly means summation over all the values of the index. For instance if $M \in \mathrm{Mat}_m(\mathbb{C})\otimes A$ and $Q \in \mathrm{Mat}_{m}(\mathbb{C}) \otimes \mathrm{Mat}_{n}(\mathbb{C}) \otimes A$, then
$(M_1Q_{12})^{ac}_{bd} = M^a_i Q^{ic}_{bd}$; note that the repeated subscript $1$ corresponds to matrix multiplication in the first space.
Second, we use Sweedler's notation (see \cite[Notation~III.1.6]{kassel}) without summation sign for the coproducts, that is we write
\[ \Delta(x) = x' \otimes x'', \qquad (\Delta \otimes \mathrm{id}) \circ \Delta(x) = (\mathrm{id} \otimes \Delta)\circ \Delta(x) = x' \otimes x'' \otimes x''', \qquad \text{and so on.} \]
Finally, we write the components of the $R$-matrix as $R = a_i \otimes b_i$ with implicit summation on $i$, and we define $R' = b_i \otimes a_i$.

For $q \in \mathbb{C}{\setminus}\{-1,0,1\}$, we define the $q$-integer $[n]$ (with $n \in \mathbb{Z})$ by
\[ [n] = \frac{q^n - q^{-n}}{q-q^{-1}}. \]
We will denote $\hat q = q-q^{-1}$ to shorten formulas. Observe that if $q$ is a $2p$-root of unity, then $[p]=0$ and $[p-n] = [n]$.

As usual $\delta_{s,t}$ is the Kronecker symbol and $\mathbb{I}_n$ is the identity matrix of size $n$.

\section{Some basic facts}\label{SomeBasicFacts}
We refer to \cite[Chapters~IV and VIII]{CR} for background material about representation theory.

\subsection{Dual of a finite-dimensional algebra}\label{dualAlg}
Let $A$ be a finite-dimensional $\mathbb{C}$-algebra,
 $V$ be a finite-dimensional $A$-module and $\overset{V}{\rho}\colon A \to \mathrm{End}_{\mathbb{C}}(V)$ be the representation associated to $V$. In other words, $\overset{V}{\rho}$ is an element of $\mathrm{End}_{\mathbb{C}}(V) \otimes A^*$ and if we choose a basis of $V$, $\overset{V}{\rho}$ becomes an element, denoted $\overset{V}{T}$, of $\mathrm{Mat}_{n}(\mathbb{C}) \otimes A^* = \mathrm{Mat}_{n}(A^*)$. In other words $\overset{V}{T}$ is a matrix whose coefficients are linear forms on~$A$. The linear form $\overset{V}{T}{^i_j}$ is called a matrix coefficient of~$V$. We will always consider the representation $\overset{V}{x} \in \mathrm{End}_{\mathbb{C}}(V)$ of~$x$ on~$V$ as the matrix $\overset{V}{T}(x) \in \mathrm{Mat}_{\dim(V)}(\mathbb{C})$: $\overset{V}{x}{^i_j} = \overset{V}{T}{^i_j}(x)$. Note that if a $A$-module $V$ is a~sub\-module or a~quotient of a $A$-module~$W$, then the matrix coefficients of~$W$ contain those of~$V$ since $\overset{V}{T}$ is a~sub\-matrix of $\overset{W}{T}$.

Since $A^*$ is finite-dimensional, it is generated as a vector space by the matrix coefficients of the regular representation $_AA$. Indeed, let $\{ x_1, \dots, x_n \}$ be a basis of $A$ with $x_1 = 1$ and let $\big\{ x^1, \dots, x^n \big\} \subset A^*$ be the dual basis; it is readily seen that $\overset{_AA}{T}{^i_1}(x_j) = \delta_{i,j}$ and thus $\overset{_AA}{T}{^i_1} = x^i$. This implies that in the sequel one might restrict everywhere to the regular representation; however, when studying examples it is very relevant to consider smaller representations in order to greatly reduce the number of generators for the algebras that we will consider in the sequel (like the fundamental representation $\mathcal{X}^+(2)$ for $A = \overline{U}_q(\mathfrak{sl}(2))$, see Section~\ref{exempleUq}).

Recall that the PIMs of $A$ are the indecomposable projective $A$-modules. They are isomorphic to the direct summands of the regular representation $_AA$. In particular, the matrix coefficients of $_AA$ are exactly those of the PIMs and it follows that the matrix coefficients of the PIMs span~$A^*$. Note however that the matrix coefficients of the PIMs (or equivalently of~$_AA$) do not form a basis of~$A^*$ in general. Indeed, even if we fix a family $(P_{\alpha})$ representing each isomorphism class of PIMs, it is possible for $P_{\alpha}$ and $P_{\beta}$ to have a composition factor $S$ in common. In this case, both $\overset{P_{\alpha}}{T}$ and $\overset{P_{\beta}}{T}$ contain $\overset{S}{T}$ as submatrix. This is what happens for $H = \overline{U}_q(\mathfrak{sl}(2))$, see, e.g., \cite[Section~3]{F}. In the semisimple case this phenomenon does not occur.

\subsection{Braided Hopf algebras, factorizability, ribbon element}\label{braidedHopfAlgebras}
Let $H$ be a braided Hopf algebra with universal $R$-matrix $R = a_i \otimes b_i$ (see, e.g., \cite[Chapter~VIII]{kassel}). Recall that
\begin{gather}
(\Delta \otimes \mathrm{id})(R) = R_{13}R_{23}, \qquad (\mathrm{id} \otimes \Delta)(R) = R_{13}R_{12}, \label{deltaR}\\
(S \otimes \mathrm{id})(R) = \big(\mathrm{id} \otimes S^{-1}\big)(R) = R^{-1}, \qquad (S \otimes S)(R) = R, \label{propSR} \\
R_{12}R_{13}R_{23} = R_{23}R_{13}R_{12} \label{YangBaxter}
\end{gather}
with $R_{12} = a_i \otimes b_i \otimes 1$, $R_{13} = a_i \otimes 1 \otimes b_i$, $R_{23} = 1 \otimes a_i \otimes b_i \in H^{\otimes 3}$. The relation \eqref{YangBaxter} is called the (quantum) Yang--Baxter equation.

Consider
\[ \fonc{\Psi}{H^*}{H,}{\beta}{(\beta \otimes \mathrm{id})(RR'),} \]
where $R' = b_i \otimes a_i$. $\Psi$ is called the Reshetikhin--Semenov-Tian-Shansky map \cite{RS}, or the Drinfeld map~\cite{drinfeld}. We say that $H$ is factorizable if $\Psi$ is an isomorphism of vector spaces. By the remarks above, we can restrict $\beta$ to be a matrix coefficient of some~$I$.

Define $R^{(+)} = R$, $R^{(-)} = (R')^{-1}$, and let
\begin{gather}\label{L}
\overset{I}{L}{^{(\pm)}} = \big(\overset{I}{T} \otimes \mathrm{id}\big)\big(R^{(\pm)}\big) = \bigl(\overset{I}{a_i^{(\pm)}}\bigr) b_i^{(\pm)} \in \mathrm{Mat}_{\dim(I)}(H)
\end{gather}
with $R^{(\pm)} = a_i^{(\pm)} \otimes b_i^{(\pm)}$ (note that $\bigl(\overset{I}{a_i^{(\pm)}}\bigr) b_i^{(\pm)}$ is the matrix obtained by multiplying each coefficient (which is a scalar) of the matrix $\bigl(\overset{I}{a_i^{(\pm)}}\bigr)$ by the element $b_i \in H$). Recall that $R^{(-)}$ is also a universal $R$-matrix and in particular it satisfies the properties \eqref{deltaR}--\eqref{YangBaxter} above. We use the letters $I,J, \dots$ for modules over Hopf algebras. If $H$ is factorizable, the coefficients of the matrices $\overset{I}{L} \!^{(\pm)}$ generate $H$ as an algebra. These matrices satisfy nice relations which are consequences of \eqref{deltaR} and \eqref{YangBaxter}:
\begin{gather}
\overset{\!\!\!\!\!I\otimes J}{L^{(\epsilon)}_{12}} = \overset{I}{L}{^{(\epsilon)}_1} \overset{J}{L}{^{(\epsilon)}_2},\nonumber\\
\overset{IJ}{R}{^{(\epsilon)}_{12}} \overset{I}{L}{^{(\epsilon)}_1} \overset{J}{L}{^{(\sigma)}_2} = \overset{J}{L}{^{(\sigma)}_2} \overset{I}{L}{^{(\epsilon)}_1} \overset{IJ}{R}{^{(\epsilon)}_{12}} \qquad \forall\, \epsilon, \sigma \in \{\pm\},\nonumber\\
\overset{IJ}{R}{^{(\epsilon)}_{12}} \overset{I}{L}{^{(\sigma)}_1} \overset{J}{L}{^{(\sigma)}_2} = \overset{J}{L}{^{(\sigma)}_2} \overset{I}{L}{^{(\sigma)}_1} \overset{IJ}{R}{^{(\epsilon)}_{12}} \qquad\forall\, \epsilon, \sigma \in \{\pm\},\nonumber\\
\Delta(\overset{I}{L}{^{(\epsilon)} \!^a_b}) = \overset{I}{L}{^{(\epsilon)} \!^i_b} \otimes \overset{I}{L}{^{(\epsilon)} \!^a_i}.\label{propertiesL}
\end{gather}
For instance, here is a proof of the first equality with $\epsilon=+$:
\[ \overset{\!\!\!\!\!I\otimes J}{L^{(+)}_{12}} = \bigl(\overset{I \otimes J}{a_i}\bigr)_{\!12} b_i = \bigl(\overset{I}{a_i} \otimes \overset{J}{a_j} \bigr)_{\!12} b_i b_j = \bigl(\overset{I}{a_i}\bigr)_{\!1} \bigl(\overset{J}{a_j}\bigr)_{\!2} b_i b_j = \bigl(\overset{I}{a_i}\bigr)_{\!1} b_i \bigl(\overset{J}{a_j}\bigr)_{\!2} b_j = \overset{I}{L} \!^{(+)}_1\overset{J}{L} \!^{(+)}_2, \]
where we used \eqref{deltaR}; note that these are equalities between matrices in
\[
\mathrm{Mat}_{\dim(I)}(\mathbb{C}) \otimes \mathrm{Mat}_{\dim(J)}(\mathbb{C}) \otimes H,\]
 which imply equalities among the coefficients. If the representations $I$ and $J$ are fixed and arbitrary, we will simply write these relations as
\[ L^{(\epsilon)}_{12} = L^{(\epsilon)}_{1}L^{(\epsilon)}_{2}, \qquad R^{(\epsilon)}_{12} L^{(\epsilon)}_1 L^{(\sigma)}_2 = L^{(\sigma)}_2 L^{(\epsilon)}_1 R^{(\epsilon)}_{12}, \qquad
R^{(\epsilon)}_{12} L^{(\sigma)}_1 L \!^{(\sigma)}_2 = L^{(\sigma)}_2 L^{(\sigma)}_1 R^{(\epsilon)}_{12}, \]
the subscript $1$ (resp.~$2$) corresponding implicitly to evaluation in the representation $I$ (resp.~$J$).

Recall that the {\em Drinfeld element} $u$ and its inverse are
\begin{gather}\label{elementDrinfeld}
u = S(b_i)a_i = b_iS^{-1}(a_i) \qquad \text{and} \qquad u^{-1} = S^{-2}(b_i)a_i = S^{-1}(b_i)S(a_i) = b_iS^2(a_i).
\end{gather}
We assume that $H$ contains a ribbon element $v$. It satisfies
\begin{gather}\label{ribbon}
v \text{ is central and invertible}, \qquad \Delta(v) = (R'R)^{-1}v \otimes v, \nonumber\\
S(v) = v, \qquad \varepsilon(v)=1, \qquad v^2 = uS(u).
\end{gather}
The two last equalities can be deduced easily from the others. A ribbon element is in general not unique. A ribbon Hopf algebra $(H,R,v)$ is a braided Hopf algebra $(H,R)$ together with a~ribbon element $v$.

We say that $g \in H$ is a {\em pivotal element} if
\begin{gather}\label{pivot}
\Delta(g) = g \otimes g \qquad \text{ and } \qquad \forall\, x \in H, \qquad S^2(x) = gxg^{-1}.
\end{gather}
Note that $g$ is invertible because it is grouplike, and hence $S(g)=g^{-1}$. A pivotal element is in general not unique. But in a ribbon Hopf algebra $(H,R,v)$ there is a canonical choice
\begin{gather}\label{pivotCan}
g = uv^{-1}.
\end{gather}
We will always take this canonical pivotal element $g$ in the sequel.

\subsection[Dual Hopf algebra $\mathcal{O}(H)$]{Dual Hopf algebra $\boldsymbol{\mathcal{O}(H)}$}
The canonical Hopf algebra structure on $H^*$ is defined by
\begin{gather*} (\psi \cdot \varphi)(x) = (\psi \otimes \varphi)(\Delta(x)), \qquad \eta(1)=\varepsilon, \\
 \Delta(\psi)(x \otimes y) = \psi(xy), \qquad \varepsilon(\psi) = \psi(1), \qquad S(\psi) = \psi \circ S \end{gather*}
with $\psi, \varphi \in H^*$ and $x,y \in H$. When it is endowed with this structure, $H^*$ is called dual Hopf algebra, and is denoted by $\mathcal{O}(H)$ in the sequel. In terms of matrix coefficients, we have
\begin{gather}
\overset{I \otimes J}{T}\!\!\!_{12} = \overset{I}{T_1}\overset{J}{T_2}, \qquad \eta(1) = \overset{\mathbb{C}}{T}, \qquad \Delta(\overset{I}{T^{ a}_{ b}}) = \overset{I}{T^{ a}_{ i}} \otimes \overset{I}{T^{ i}_{ b}}, \nonumber\\ \varepsilon(\overset{I}{T}) = \mathbb{I}_{\dim(I)}, \qquad S(\overset{I}{T}) = \overset{I}{T}{^{-1}},\label{dualHopf}
\end{gather}
where $\mathbb{C}$ is the trivial representation. For instance, here are proofs of the first and last equalities (with $h \in H$):
\begin{gather*}
 \bigl(\overset{I \otimes J}{T}\!_{\!\!12}\bigr)(h) = \bigl(\overset{I \otimes J}{h}\bigr)_{12} = \bigl( \overset{I}{h'} \otimes \overset{J}{h''} \bigr)_{12} = \bigl(\overset{I}{h'}\bigr)_1 \bigl(\overset{J}{h''}\bigr)_2 = \overset{I}{T_1}(h') \overset{J}{T_2}(h'') = \bigl( \overset{I}{T_1} \overset{J}{T_2} \bigr)(h),\\
 \bigl(S(\overset{I}{T}) \overset{I}{T}\bigr)(h) = S(\overset{I}{T})(h') \overset{I}{T}(h'') = \overset{I}{T}\bigl(S(h')\bigr) \overset{I}{T}(h'') = \overset{I}{T}\bigl( S(h')h'' \bigr) = \varepsilon(h) \overset{I}{T}(1) = \varepsilon(h) \mathbb{I}_{\dim(I)}.
\end{gather*}
By definition of the action on the dual $I^*$, it holds
\begin{gather}\label{antipodeT}
S\big(\overset{I}{T}\big) = {} ^t \overset{\:I^*}{T},
\end{gather}
where $^t$ is the transpose. Also recall the well-known exchange relation
\begin{gather}\label{FRT}
\overset{IJ}{R}_{12} \overset{I}{T_1} \overset{J}{T_2} = \overset{J}{T_2} \overset{I}{T_1} \overset{IJ}{R}_{12}.
\end{gather}
As before, if the representations $I$ and $J$ are fixed and arbitrary, we will simply write $T_{12} = T_1 T_2$ and $R_{12} T_1 T_2 = T_2 T_1 R_{12}$.

\section[The loop algebra $\mathcal{L}_{0,1}(H)$]{The loop algebra $\boldsymbol{\mathcal{L}_{0,1}(H)}$}\label{sectionLoopL01}
We assume that $H$ is a finite-dimensional factorizable Hopf algebra. The ribbon assumption is not needed in this section.

\subsection[Definition of $\mathcal{L}_{0,1}(H)$ and $H$-module-algebra structure]{Definition of $\boldsymbol{\mathcal{L}_{0,1}(H)}$ and $\boldsymbol{H}$-module-algebra structure}
Let $\mathrm{T}(H^*)$ be the tensor algebra of $H^*$, which by definition is linearly spanned by all the formal products $\varphi_1 \cdots \varphi_n$ (with $n \geq 0$ and $\varphi_i \in H^*$) modulo the obvious multilinear relations. There is a canonical injection $j \colon H^* \to \mathrm{T}(H^*)$ and we let $\overset{I}{M} \in \mathrm{Mat}_{\dim(I)}\bigl( \mathrm{T}(H^*) \bigr)$ be the matrix defined by $\overset{I}{M} = j\bigl(\overset{I}{T}\bigr)$, that is $\overset{I}{M}{^a_b} = j\bigl(\overset{I}{T}{^a_b}\bigr)$.

\begin{Definition}\label{defL01} The loop algebra $\mathcal{L}_{0,1}(H)$ is the quotient of $\mathrm{T}(H^*)$ by the following fusion relations
\[ \overset{I \otimes J}{M}\!_{12} = \overset{I}{M}_1\overset{IJ}{(R')}_{12}\overset{J}{M}_2\overset{IJ}{\big(R'^{-1}\big)}_{12} \]
for all finite-dimensional $H$-modules~$I$, $J$.
\end{Definition}

The right hand-side of the fusion relation in Definition \ref{defL01} is the one of \cite[Definition~1]{BNR}; the one of \cite[Definition~12]{AGS2} and \cite[equation~(3.11)]{AS} is different, due to different choices of the action of~$H$ on~$\mathcal{L}_{0,1}(H)$ and to particular normalization of Clebsch--Gordan operators. Moreover in the papers \cite{AGS2, AS, BNR} the matrix $\overset{I \otimes J}{M}$ did not appeared, instead it was always decomposed as a sum of the matrices of the irreducible direct summands of $I \otimes J$ thanks to Clebsch--Gordan operators, which is relevant in the semisimple case only; in~\cite{S}, $I$ and $J$ are restricted to the regular representation and the matrix $\overset{I \otimes J}{M}$ is thus denoted by $\Delta_a(M)$. In the semisimple setting, the algebras resulting from each of these definitions are isomorphic.

Note that the fusion relation is a relation between matrices in $\mathrm{Mat}_{\dim(I)}(\mathbb{C}) \otimes \mathrm{Mat}_{\dim(J)}(\mathbb{C}) \otimes \mathcal{L}_{0,1}(H)$ (for all finite-dimensional $I$,~$J$) which implies relations among elements of $\mathcal{L}_{0,1}(H)$ (the coefficients of these matrices). Explicitly, in terms of matrix coefficients it is written as
\[ \forall\, I,J,a,b,c,d, \qquad \overset{I \otimes J}{M}{^{ac}_{bd}} = \overset{I}{M}{^a_i} \overset{IJ}{(R')}{^{ic}_{jk}} \overset{J}{M}{^k_l} \overset{IJ}{(R'^{-1})}{^{jl}_{bd}}, \]
see the definition of the subscripts $1$ and $2$ in the Notations at p.~\pageref{Notations}. If the two representations~$I$ and~$J$ are fixed and arbitrary, we will simply write
\begin{gather}\label{fusionL01}
M_{12} = M_1 R_{21} M_2 R_{21}^{-1}
\end{gather}
the subscript $1$ (resp.~$2$) corresponding implicitly to evaluation in the representation~$I$ (resp.~$J$). Finally, note that if $f\colon I \to J$ is a morphism, it holds
\begin{gather}\label{naturalite01}
f\overset{I}{M} = \overset{J}{M}f,
\end{gather}
where we identify $f$ with its matrix. Indeed, $f\overset{I}{M} = f \bar j\bigl(\overset{I}{T}\bigr) = \bar j\bigl(f\overset{I}{T}\bigr) = \bar j\bigl(\overset{J}{T}f\bigr) = \bar j\bigl(\overset{J}{T}\bigr) f = \overset{J}{M}f $ where $\bar j$ is the linear map $H^* \to \mathrm{T}(H^*) \to \mathcal{L}_{0,1}(H)$.
\begin{Remark}
One can check that $\overset{(I \otimes J)\otimes K}{M} = \overset{I \otimes (J \otimes K)}{M}$ holds thanks to the Yang--Baxter equation.
\end{Remark}

We have an useful analogue of relation \eqref{FRT}.

\begin{Proposition}\label{EqRef}The following exchange relations hold in $\mathcal{L}_{0,1}(H)$:
\[ \overset{IJ}{R}_{12}\overset{I}{M}_1\overset{IJ}{(R')}_{12}\overset{J}{M}_2 = \overset{J}{M}_2\overset{IJ}{R}_{12}\overset{I}{M}_1\overset{IJ}{(R')}_{12}. \]
\end{Proposition}
\noindent This relation is called the {\em reflection equation}. It can be written in a shortened way if the representations $I$ and $J$ are fixed and arbitrary:
\begin{gather}\label{reflection}
R_{12}M_1R_{21}M_2 = M_2R_{12}M_1R_{21}.
\end{gather}
\begin{proof}
Let $P_{I,J}\colon I \otimes J \to J \otimes I$ be the flip map $P_{IJ}(v \otimes w) = w \otimes v$. We have the isomorphism $P_{IJ}\overset{IJ}{R}\colon I \otimes J \to J \otimes I$, hence $P_{IJ}\overset{IJ}{R} \overset{I \otimes J}{M} = \overset{J \otimes I}{M}P_{IJ}\overset{IJ}{R}$ thanks to \eqref{naturalite01}. This gives
\begin{align*}
(P_{IJ})_{12}\overset{IJ}{R}_{12}\overset{I}{M}_1\overset{IJ}{(R')}_{12}\overset{J}{M}_2\overset{IJ}{(R'^{-1})}_{12}& = \overset{J}{M}_1\overset{JI}{(R')}_{12}\overset{I}{M}_2\overset{JI}{(R'^{-1})}_{12}(P_{IJ})_{12}\overset{IJ}{R}_{12} \\
&=(P_{IJ})_{12} \overset{J}{M}_2\overset{JI}{(R')}_{21}\overset{I}{M}_1\overset{JI}{(R'^{-1})}_{21}\overset{IJ}{R}_{12} = (P_{IJ})_{12} \overset{J}{M}_2\overset{IJ}{R}_{12}\overset{I}{M}_1
\end{align*}
as desired.
\end{proof}

Consider the following right action $\cdot$ of $H$ on $\mathcal{L}_{0,1}(H)$, which is the analogue of the right action of the gauge group on the gauge fields
\begin{gather}\label{actionHsurL01}
\overset{I}{M} \cdot h = \overset{I}{h'} \overset{I}{M} \overset{I}{S(h'')}
\end{gather}
or more explicitly $\overset{I}{M}{^a_b} \cdot h = (\overset{I}{h'}){^a_i} \overset{I}{M}{^i_j} \overset{I}{S(h'')}{^j_b}$. As in \cite{BR}, one can equivalently work with the corresponding left coaction $\Omega \colon \mathcal{L}_{0,1}(H) \to \mathcal{O}(H) \otimes \mathcal{L}_{0,1}(H)$ defined by
\[ \Omega(\overset{I}{M}{^a_b}) = \overset{I}{T}{^a_i}S\big(\overset{I}{T}{^j_b}\big) \otimes \overset{I}{M}{^i_j}, \]
so that we recover $\cdot$ by evaluation: $x \cdot h = (\langle ?, h \rangle \otimes \mathrm{id}) \circ \Omega(x)$. If we view $\mathcal{O}(H)$ and $\mathcal{L}_{0,1}(H)$ as subalgebras of $\mathcal{O}(H) \otimes \mathcal{L}_{0,1}(H)$ in the canonical way, then $\Omega$ is simply written as $\Omega(\overset{I}{M}) = \overset{I}{T}\overset{I}{M}S(\overset{I}{T})$.

\begin{Proposition}\label{L01moduleAlgebra}
The right action $\cdot$ is a $H$-module-algebra structure on $\mathcal{L}_{0,1}(H)$. Equivalently, $\Omega$ is a left $\mathcal{O}(H)$-comodule-algebra structure on $\mathcal{L}_{0,1}(H)$.
\end{Proposition}
\begin{proof}One must show for instance that $\Omega$ is a morphism of algebras, as in \cite{BR}. For the sake of completeness we display the computation, where we use the shortened notation explained before:
 \begin{align*}
\Omega(M)_{12} &= T_{12} M_{12} S(T_{12}) &&\quad \text{(definition)}\\
&= T_1 T_2 M_1 R_{21} M_2 R^{-1}_{21} S(T)_2 S(T)_1 &&\quad \text{(equations~\eqref{dualHopf} and \eqref{fusionL01})}\\
&= T_1 M_1 T_2 R_{21} M_2 R^{-1}_{21} S(T)_2 S(T)_1 &&\quad \text{(commuting coefficients)}\\
&= T_1 M_1 T_2 R_{21} M_2 S(T)_1 S(T)_2 R_{21}^{-1} &&\quad \text{(equation~\eqref{FRT})}\\
&= T_1 M_1 T_2 R_{21} S(T)_1 M_2 S(T)_2 R_{21}^{-1} &&\quad \text{(commuting coefficients)}\\
&= T_1 M_1 S(T)_1 R_{21} T_2 M_2 S(T)_2 R_{21}^{-1} &&\quad \text{(equation~\eqref{FRT})}\\
&= \Omega(M)_1 R_{21} \Omega(M)_2 R_{21}^{-1} &&\quad \text{(definition)}.\tag*{\qed}
\end{align*}\renewcommand{\qed}{}
\end{proof}

We say that an element $x \in \mathcal{L}_{0,1}(H)$ is invariant if for all $h \in H$, $x \cdot h = \varepsilon(h)x$ (or equivalently, $\Omega(x) = \varepsilon \otimes x$). For instance, for every $\Phi \in \End_H(I)$, the element $\mathrm{tr}(\overset{I}{g}\Phi\overset{I}{M})$ is invariant:
\[ \mathrm{tr}\big(\overset{I}{g}\Phi\overset{I}{M}\big) \cdot h = \mathrm{tr}\big(\overset{I}{g}\Phi \overset{I}{h'} \overset{I}{M} \overset{I}{S(h'')}\big) = \mathrm{tr}\big(\overset{I}{g}\Phi \overset{I}{S^{-1}(h'')}\overset{I}{h'} \overset{I}{M}\big) = \varepsilon(h)\mathrm{tr}\big(\overset{I}{g}\Phi\overset{I}{M}\big). \]
We denote by $\mathcal{L}_{0,1}^{\text{inv}}(H)$ the subalgebra of invariants elements of $\mathcal{L}_{0,1}(H)$.

\subsection[Isomorphism $\mathcal{L}_{0,1}(H) \cong H$]{Isomorphism $\boldsymbol{\mathcal{L}_{0,1}(H) \cong H}$}\label{isoL01H}
Recall the right adjoint action of $H$ on itself defined by $a \cdot h = S(h')ah''$ with $a,h \in H$, whose invariant elements are the central elements of $H$.
\begin{Proposition}\label{reshetikhin} If we endow $H$ with the right adjoint action, the following map is a morphism of $($right$)$ $H$-module-algebras
\[ \fonc{\Psi_{0,1}}{\mathcal{L}_{0,1}(H)}{H,}{\overset{I}{M}}{\big(\overset{I}{T} \otimes \text{\rm id}\big)(RR') = \overset{I}{L}{^{(+)}}\overset{I}{L}{^{(-)-1}}.} \]
Hence, $\Psi_{0,1}$ brings invariant elements to central elements.
\end{Proposition}
\noindent We will call $\Psi_{0,1}$ the {\em Reshetikhin--Semenov-Tian-Shansky--Drinfeld morphism} (RSD morphism for short). The difference with the morphism $\Psi$ of Section~\ref{braidedHopfAlgebras} is that the source spaces are different.
\begin{proof}
Thanks to the relations of \eqref{propertiesL}, we check that $\Psi_{0,1}$ preserves the relation of Definition~\ref{defL01}:
\begin{align*}
\Psi_{0,1}(M)_1R_{21}\Psi_{0,1}(M)_2R^{-1}_{21} &= L^{(+)}_1L^{(-)-1}_1 R_{21} L^{(+)}_2L^{(-)-1}_2R_{21}^{-1}\\
& = L^{(+)}_1 L^{(+)}_2 R_{21} L^{(-)-1}_1 L^{(-)-1}_2R_{21}^{-1}= L^{(+)}_1 L^{(+)}_2 L^{(-)-1}_2 L^{(-)-1}_1\\
& = L^{(+)}_{12}L^{(-)-1}_{12} = \Psi_{0,1}(M)_{12}.
\end{align*}
For the $H$-linearity
\begin{align*}
\Psi_{0,1}\big(\overset{I}{h'}\overset{I}{M}\overset{I}{S(h'')}\big) &= \big(\overset{I}{T} \otimes \mathrm{id}\big) (h' \otimes 1 RR' S(h'') \otimes 1 )\\
&= \big(\overset{I}{T} \otimes \mathrm{id}\big) (h' \otimes 1 RR' S(h)''' \otimes S(h)''h'''' )\\
&= \big(\overset{I}{T} \otimes \mathrm{id}\big) (h'S(h)''' \otimes S(h)'' RR' 1 \otimes h'''' )\\
&= \big(\overset{I}{T} \otimes \mathrm{id}\big) (1 \otimes S(h') RR' 1 \otimes h'' ) = S(h') \Psi_{0,1}(\overset{I}{M})h''.
\end{align*}
We used the basic properties of $S$ and the fact that $\Delta^{\mathrm{op}}R = R\Delta$, with $\Delta^{\mathrm{op}}(h) = h'' \otimes h'$.
\end{proof}

Write $\mathrm{T}(H^*) = \bigoplus_{ n \in \mathbb{N}} \mathrm{T}_n(H^*)$, where $\mathrm{T}_n(H^*)$ is the subspace generated by all the products $\psi_1 \cdots \psi_n$, with $\psi_i \in H^*$ for each $i$.

\begin{Lemma}\label{lemmeDegre}
Each element of $\mathrm{T}(H^*)$ is equivalent modulo the fusion relation of $\mathcal{L}_{0,1}(H)$ to an element of $\mathrm{T}_1(H^*)$. It follows that $\dim (\mathcal{L}_{0,1}(H) ) \leq \dim (H^* )$.
\end{Lemma}
\begin{proof}
It suffices to show that the product of two elements of $T_1(H^*)$ is equivalent to a linear combination of elements of $\mathrm{T}_1(H^*)$, and the result follows by induction. We can restrict to matrix coefficients since they linearly span $H^*$. If we write $R = a_i \otimes b_i$, then $R' = b_i \otimes a_i$ and the fusion relation is rewritten as
\[ \overset{I \otimes J}{M}\!\!_{12} (\overset{IJ}{R'})_{12} = \overset{I}{M}\!_1 (\overset{IJ}{R'})_{12} \overset{J}{M}\!_2 = \overset{I}{M}\!_1 \big(\overset{I}{b_i}\big)_1 \big(\overset{J}{a_i}\big)_2 \overset{J}{M}\!_2 = \big(\overset{J}{a_i}\big)_2 \overset{I}{M}\!_1 \overset{J}{M}\!_2 \big(\overset{I}{b_i}\big)_1. \]
Using $S^{-1}(a_j)a_i \otimes b_ib_j = 1 \otimes 1$, we get
\begin{gather}\label{eqFusionInverse}
\overset{I}{M}\!_1 \overset{J}{M}\!_2 = \overset{J}{S^{-1}(a_i)}_2\overset{I \otimes J}{M}\!\!_{12} \big(\overset{IJ}{R'}\big)_{12} \big(\overset{I}{b_i}\big)_1
\end{gather}
and this give the result since $\overset{I}{M}_1 \overset{J}{M}_2$ contains all the possible products between the coefficients of $\overset{I}{M}$ and those of $\overset{J}{M}$.
\end{proof}

\begin{Theorem}\label{IsoRSD}Recall that we assume that $H$ is a finite-dimensional factorizable Hopf algebra. Then the RSD morphism $\Psi_{0,1}$ gives an isomorphism of $H$-module-algebras $\mathcal{L}_{0,1}(H) \cong H$. It follows that $\mathcal{L}_{0,1}^{\mathrm{inv}}(H) \cong \mathcal{Z}(H)$.
\end{Theorem}
\begin{proof}Since $H$ is factorizable, $\Psi_{0,1}$ is surjective. Hence $\dim(\mathcal{L}_{0,1}(H)) \geq \dim(H)$. But by Lemma~\ref{lemmeDegre}, $\dim(\mathcal{L}_{0,1}(H)) \leq \dim(H^*) = \dim(H)$. Thus $\dim(\mathcal{L}_{0,1}(H)) = \dim(H)$.
\end{proof}

Let us point out obvious consequences. First, by comparing the dimensions, we see that the canonical map $H^* \hookrightarrow \mathrm{T}(H^*) \twoheadrightarrow \mathcal{L}_{0,1}(H)$ is an isomorphism of vector spaces. Second, this shows that the matrices $\overset{I}{M}$ are invertible since $RR'$ is invertible. More importantly, this theorem allows us to identify $\mathcal{L}_{0,1}(H)$ with~$H$ via $\overset{I}{M} = \overset{I}{L} \!^{(+)}\overset{I}{L} \!^{(-)-1}$, where the matrices $L^{(\pm)}$ are defined in $\eqref{L}$. We will always work with this identification in the sequel.

\begin{Remark}Thanks to the isomorphism of vector spaces $\overset{I}{T} \to \overset{I}{M}$ and the relation~\eqref{eqFusionInverse}, we can see $\mathcal{L}_{0,1}(H)$ as $H^*$ endowed with a product $\ast$ defined by
\begin{gather}\label{productGaugeFields01}
\varphi \ast \psi = \varphi\!\left( ? b_j b_i \right) \psi\big(S^{-1}(a_i) ? a_j\big), \qquad x_m \mapsto \varphi ( x_m' b_j b_i ) \psi\big(S^{-1}(a_i) x_m'' a_j\big).
\end{gather}
This is the product of the gauge fields $\varphi, \psi$ and its evaluation on the connection which assigns $x_m \in H$ to the loop~$m$, see Fig.~\ref{figureSurfaces}. The action of the gauge algebra $H$ on a gauge field $\varphi$ is $\varphi \cdot h = \varphi (h' ? S(h'') )$ and $\ast$ is $H$-equivariant.
\end{Remark}

We denote by $\SLF(H)$ the space of symmetric linear forms on $H$:
\[ \SLF(H) = \{ \psi \in H^* \,|\, \forall\, x,y \in H, \, \psi(xy) = \psi(yx) \}. \]
$\SLF(H)$ is obviously a subalgebra of $\mathcal{O}(H)$. Consider the following variant of the map $\Psi$ of Section~\ref{braidedHopfAlgebras}, which will be useful in what follows
\begin{gather}\label{morphismeDrinfeld}
\fonc{\mathcal{D}}{\mathcal{O}(H)}{H,}{\psi}{\left(\psi \otimes \text{id}\right)\left( (g \otimes 1) RR'\right),}
\end{gather}
where $g$ is the pivotal element~\eqref{pivotCan}. Since $H$ is factorizable, $\mathcal{D}$ is an isomorphism of vector spaces. A computation similar to that of the proof of Proposition~\ref{reshetikhin} shows that $\mathcal{D}$ brings symmetric linear forms to central elements. Moreover, it is not difficult to show that it induces an isomorphism of algebras $\SLF(H) \cong \mathcal{Z}(H) = \mathcal{L}^{\mathrm{inv}}_{0,1}(H)$.

Let us fix a notation. Every $\psi \in H^*$ can be written as $\psi = \sum_{i,j,I}\lambda_{ij}^I\overset{I}{T} \!^i_j$ with $\lambda_{ij}^I \in \mathbb{C}$. In order to avoid the indices, define for each $I$ a matrix $\Lambda_I \in \mathrm{Mat}_{\dim(I)}(\mathbb{C})$ by $(\Lambda_I)^i_j = \lambda^I_{ji}$. Then~$\psi$ can be expressed as $\psi = \sum_I \mathrm{tr}\big(\Lambda_I\overset{I}{T}\big)$. We record these observations as a lemma.

\begin{Lemma}\label{writingInvariants}Every $x \in \mathcal{L}_{0,1}(H)$ can be expressed as
\[ x = \sum_I \mathrm{tr}\big(\Lambda_I\overset{I}{g}\overset{I}{M}\big) \]
such that $\mathcal{D}^{-1}(x) = \sum_I \mathrm{tr}\big(\Lambda_I\overset{I}{T}\big)$. Moreover, if $x \in \mathcal{L}^{\mathrm{inv}}_{0,1}(H)$, then $\mathcal{D}^{-1}(x) \in \SLF(H)$.
\end{Lemma}

\begin{Remark}\label{remarkWriting}Let us stress that, due to non-semi-simplicity, this way of writing elements of $\mathcal{L}_{0,1}(H)$ and of $\SLF(H)$ is in general not unique, see the comments in Section~\ref{dualAlg}.
\end{Remark}

\section[The handle algebra $\mathcal{L}_{1,0}(H)$]{The handle algebra $\boldsymbol{\mathcal{L}_{1,0}(H)}$}\label{sectionHandleL10}
We assume that $H$ is a finite-dimensional factorizable ribbon Hopf algebra. Note however that the ribbon assumption is not needed in Sections~\ref{defL10} and~\ref{sectionIsoL10Heisenberg}.

\subsection[Definition of $\mathcal{L}_{1,0}(H)$ and $H$-module-algebra structure]{Definition of $\boldsymbol{\mathcal{L}_{1,0}(H)}$ and $\boldsymbol{H}$-module-algebra structure}\label{defL10}
Consider the free product $\mathcal{L}_{0,1}(H) \ast \mathcal{L}_{0,1}(H)$, and let $j_1$ (resp.~$j_2$) be the canonical injection in the first (resp.\ second) copy of $\mathcal{L}_{0,1}(H)$. We define $\overset{I}{B} = j_1(\overset{I}{M})$ and $\overset{I}{A} = j_2(\overset{I}{M})$, that is $\overset{I}{B}{^a_b} = j_1(\overset{I}{M}{^a_b})$, $\overset{I}{A}{^a_b} = j_2(\overset{I}{M}{^a_b})$.

\begin{Definition}\label{definitionL10}The handle algebra $\mathcal{L}_{1,0}(H)$ is the quotient of $\mathcal{L}_{0,1}(H) \ast \mathcal{L}_{0,1}(H)$ by the following exchange relations:
\[ \overset{IJ}{R}_{12}\overset{I}{B}_1\overset{IJ}{(R')}_{12}\overset{J}{A}_2 = \overset{J}{A}_2\overset{IJ}{R}_{12}\overset{I}{B}_1\overset{IJ}{R}{^{-1}_{12}} \]
for all finite-dimensional $H$-modules $I$, $J$.
\end{Definition}
\noindent The exchange relation above is the same as in \cite[Definition~1]{BNR} except that $A$ and $B$ are switched; the one of \cite[Definition~12]{AGS2} and \cite[equation~(3.14)]{AS} is different, due to a different choice of the action of $H$ on $\mathcal{L}_{1,0}(H)$. In the semisimple setting, the algebras resulting from each of these definitions are isomorphic. This is a relation between matrices in $\mathrm{Mat}_{\dim(I)}(\mathbb{C}) \otimes \mathrm{Mat}_{\dim(J)}(\mathbb{C}) \otimes \mathcal{L}_{1,0}(H)$ (for all finite-dimensional $I$,~$J$) which implies relations among elements of $\mathcal{L}_{1,0}(H)$, namely
\[ \forall\, I,J,a,b,c,d, \qquad \overset{IJ}{R}{^{ac}_{ij}} \overset{I}{B}{^i_k } (\overset{IJ}{R'})^{kj}_{bl} \overset{J}{A}{^l_d} = \overset{J}{A}{^c_i} \overset{IJ}{R}{^{ai}_{jk}} \overset{I}{B}{^j_l} \big(\overset{IJ}{R}{^{-1}}\big)^{lk}_{bd}. \]
\noindent Like the other relations before, we can write the $\mathcal{L}_{1,0}(H)$-exchange relation more simply as:
\begin{gather}\label{echangeL10}
R_{12} B_1 R_{21} A_2 = A_2 R_{12} B_1 R_{12}^{-1}.
\end{gather}
It immediately follows from \eqref{naturalite01} that if $f\colon I \to J$ is a morphism, it holds
\[ f \overset{I}{B} = \overset{J}{B} f, \qquad f \overset{I}{A} = \overset{J}{A} f, \]
where we identify $f$ with its matrix.

Similarly to $\mathcal{L}_{0,1}(H)$, consider the following right action of $H$ on $\mathcal{L}_{1,0}(H)$, which is the analogue of the action of the gauge group on the gauge fields:
\begin{gather}\label{actionHsurL10}
\overset{I}{B} \cdot h = \overset{I}{h'} \overset{I}{B} \overset{I}{S(h'')}, \qquad \overset{I}{A} \cdot h = \overset{I}{h'} \overset{I}{A} \overset{I}{S(h'')}.
\end{gather}
As above, it is equivalent to work with the corresponding left coaction $\Omega\colon \mathcal{L}_{1,0}(H) \to \mathcal{O}(H) \otimes \mathcal{L}_{1,0}(H)$ defined by
\[ \Omega(\overset{I}{B}) = \overset{I}{T}\overset{I}{B}S(\overset{I}{T}), \qquad \Omega(\overset{I}{A}) = \overset{I}{T}\overset{I}{A}S(\overset{I}{T}). \]

\begin{Proposition}\label{L10moduleAlgebra}
The right action $\cdot$ is a $H$-module-algebra structure on $\mathcal{L}_{1,0}(H)$. Equivalently, $\Omega$ is a left $\mathcal{O}(H)$-comodule-algebra structure on $\mathcal{L}_{1,0}(H)$.
\end{Proposition}
\begin{proof}One must show that $\Omega$ is an algebra morphism, as in~\cite{BR}. This amounts to check that $\Omega$ is compatible with the exchange relation, which is similar to the proof of Proposition~\ref{L01moduleAlgebra}.
\end{proof}

We denote by $\mathcal{L}_{1,0}^{\text{inv}}(H)$ the subalgebra of invariant elements of $\mathcal{L}_{1,0}(H)$. For instance, the elements{\samepage
\begin{gather}\label{FamilleInvL10}
\mathrm{tr}_{12}\Big(\overset{I \otimes J}{g}\!\!\!_{12}\Phi\overset{I}{A}_1\overset{IJ}{(R')}_{12}\overset{J}{B}_2\overset{IJ}{R}_{12}\Big)
\end{gather}
with $\Phi \in \End_H(I \otimes J)$ and $\mathrm{tr}_{12} = \mathrm{tr} \otimes \mathrm{tr}$, are invariant.}

Finally, we describe a wide family of maps $\mathcal{L}_{0,1}(H) \to \mathcal{L}_{1,0}(H)$. For $w \in \mathcal{L}^{\mathrm{inv}}_{0,1}(H) = \mathcal{Z}(H)$ and $m_1, n_1, \dots, m_k, n_k \in \mathbb{Z}$, define
\[ \fonc{j_{wB^{m_1}A^{n_1} \dots B^{m_k}A^{n_k}}}{\mathcal{L}_{0,1}(H)}{\mathcal{L}_{1,0}(H),}{\overset{I}{M}}{\overset{I}{w}\overset{I}{B}{^{m_1}}\overset{I}{A}{^{n_1}}\cdots \overset{I}{B}{^{m_k}}\overset{I}{A}{^{n_k}}.} \]
It is clear that these maps are morphisms of $H$-modules, but not of algebras in general. Hence the restriction satisfies $j_{wB^{m_1}A^{n_1} \dots B^{m_k}A^{n_k}} \colon \mathcal{L}^{\mathrm{inv}}_{0,1}(H) \to \mathcal{L}^{\mathrm{inv}}_{1,0}(H)$. This gives a particular type of invariants in $\mathcal{L}_{1,0}(H)$. We will more shortly write
\begin{gather}\label{coinvParticuliers}
x_{wB^{m_1}A^{n_1} \dots B^{m_k}A^{n_k}} = j_{wB^{m_1}A^{n_1} \dots B^{m_k}A^{n_k}}(x).
\end{gather}
We also use this notation for $x \in H$, thanks to the identification $H = \mathcal{L}_{0,1}(H)$.

\begin{Remark}Recall from Remark~\ref{remarkWriting} that the matrix coefficients do not form a basis of $\mathcal{L}_{0,1}(H)$. They just linearly span this space. However, the maps $j_{wA^{m_1}B^{n_1} \dots A^{m_k}B^{n_k}}$ are well-defined. Indeed, first observe that
\begin{gather*}
j_B \colon \ \mathcal{L}_{0,1}(H) \overset{j_1}{\longhookrightarrow} \mathcal{L}_{0,1}(H) \ast \mathcal{L}_{0,1}(H) \overset{\pi}{\longrightarrow\!\!\!\!\!\rightarrow} \mathcal{L}_{1,0}(H),\\
j_A \colon \ \mathcal{L}_{0,1}(H) \overset{j_2}{\longhookrightarrow} \mathcal{L}_{0,1}(H) \ast \mathcal{L}_{0,1}(H) \overset{\pi}{\longrightarrow\!\!\!\!\!\rightarrow} \mathcal{L}_{1,0}(H)
\end{gather*}
are well-defined. Let us show for instance that the map $j_{A^{-1}B^{-1}A}$ is well-defined. Assume that $\lambda^a_b\overset{I}{T}{^b_a} = 0$. Applying the coproduct in $\mathcal{O}(H)$ twice and tensoring with $\mathrm{id}_H$, we get
\[ \lambda^a_b \overset{I}{T}{^b_k} \otimes \mathrm{id}_H \otimes \overset{I}{T}{^k_l} \otimes \mathrm{id}_H \otimes \overset{I}{T}{^l_a} \otimes \mathrm{id}_H = 0. \]
We evaluate this on $\left(RR'\right)^{-1} \otimes \left(RR'\right)^{-1} \otimes RR'$:
\[ \lambda^a_b \big(\overset{I}{M}{^{-1}}\big){^b_k} \otimes \big(\overset{I}{M}{^{-1}}\big){^k_l} \otimes \overset{I}{M}{^l_a} = 0. \]
Finally, we apply the map $j_A \otimes j_B \otimes j_A$ and multiplication in $\mathcal{L}_{1,0}(H)$:
\[ \lambda^a_b \big(\overset{I}{A}{^{-1}}\overset{I}{B}{^{-1}}\overset{I}{A}\big){^b_a} = 0 \]
as desired. A similar proof can be used to show that all the other maps defined by means of matrix coefficients (like $\Psi_{1,0}$ or $\alpha$, $\beta$ below etc..) are well-defined.
\end{Remark}

\subsection[Isomorphism $\mathcal{L}_{1,0}(H) \cong \mathcal{H}(\mathcal{O}(H))$]{Isomorphism $\boldsymbol{\mathcal{L}_{1,0}(H) \cong \mathcal{H}(\mathcal{O}(H))}$}\label{sectionIsoL10Heisenberg}

Let us begin by recalling the following definition (see for instance \cite[Example~4.1.10]{Mon}).
\begin{Definition}\label{defHeisenberg}
Let $H$ be a Hopf algebra. The {Heisenberg double} of $\mathcal{O}(H)$, $\mathcal{H}(\mathcal{O}(H))$, is the vector space $\mathcal{O}(H) \otimes H$ endowed with the algebra structure defined by the following multiplication rules:
\begin{itemize}\itemsep=0pt
\item The canonical injections $H, \mathcal{O}(H) \to \mathcal{O}(H) \otimes H$ are algebra morphisms. Thus we identify~$\mathcal{O}(H)$ (resp.~$H$) with $\mathcal{O}(H) \otimes 1 \subset \mathcal{H}(\mathcal{O}(H))$ (resp.\ with $1 \otimes H \subset \mathcal{H}(\mathcal{O}(H))$).
\item Under this identification, we have the exchange relation
\[ \forall\, \psi \in \mathcal{O}(H), \qquad \forall\, h \in H, \qquad h\psi = \psi(? h')h'' = \psi''(h')\psi'h''. \]
\end{itemize}
\end{Definition}

There is a representation of $\mathcal{H}(\mathcal{O}(H))$ on $\mathcal{O}(H)$ defined by
\begin{gather}\label{repHeisenberg}
h \triangleright \psi = \psi(?h) = \psi''(h)\psi', \qquad \varphi \triangleright \psi = \varphi\psi \qquad h \in H, \qquad \psi, \varphi \in \mathcal{O}(H).
\end{gather}
This representation is faithful (see \cite[Lemma~9.4.2]{Mon}). Hence, if $H$ is finite-dimensional, it follows that $\mathcal{H}(\mathcal{O}(H))$ is a matrix algebra
\begin{gather}\label{HeisenbergMat}
\mathcal{H}(\mathcal{O}(H)) \cong \End_{\mathbb{C}}(H^*).
\end{gather}

Under our assumptions on $H$, a natural set of generators for $\mathcal{H}(\mathcal{O}(H))$ consists of the matrix coefficients of $\overset{I}{T}$ and of $\overset{I}{L} \!^{(\pm)}$.
\begin{Lemma}\label{echangeHeisenberg}
With these generators, the exchange relation of $\mathcal{H}(\mathcal{O}(H))$ is
\[ \overset{I}{L} \!^{(\pm)}_1 \overset{J}{T_2} = \overset{J}{T_2} \overset{I}{L} \!^{(\pm)}_1 \overset{IJ}{R} \!^{(\pm)}_{12} \]
and the representation $\triangleright$ is
\[ \overset{I}{L} \!^{(\pm)}_1 \triangleright \overset{J}{T_2} = \overset{J}{T_2}\overset{IJ}{R} \!^{(\pm)}_{12}, \qquad \overset{I}{T_1}\triangleright \overset{J}{T_2} = \overset{I \otimes J}{T}\!\!\!_{12}.\]
\end{Lemma}
\begin{proof}
For the exchange relation, we apply the defining relation of $\mathcal{H}(\mathcal{O}(H))$ together with the definition \eqref{L} of $\overset{I}{L}{^{(\pm)}}$ and~\eqref{deltaR}:
\begin{align*}
\overset{I}{L}{^{(+)}_1} \overset{J}{T_2} &= \big(\overset{I}{a_i}\big)_1 b_i \overset{J}{T_2} = \bigl(\overset{I}{a_i}\bigr)_1 \overset{J}{T}(? b'_i)_2 b_i'' = \bigl(\overset{I}{a_i}\bigr)_1 \overset{J}{T_2} \bigl(\overset{J}{b_i'}\bigr)_2 b_i'' = \bigl( \overset{I}{a_i} \overset{I}{a_j} \bigr)_1 \overset{J}{T_2} \bigl(\overset{J}{b_j}\bigr)_2 b_i\\
& = \overset{J}{T_2} \bigl( \overset{I}{a_i}\bigr)_1 b_i \bigl(\overset{I}{a_j} \bigr)_1 \bigl(\overset{J}{b_j}\bigr)_2 = \overset{J}{T_2} \overset{I}{L}{^{(+)}_1} \overset{IJ}{R}_{12}.
\end{align*}
We used that $\overset{J}{T}(?x) = \overset{J}{T} \overset{J}{x}$ (which just follows from the fact that $\overset{J}{T}$ is a morphism from $H$ to $\mathrm{Mat}_{\dim(J)}(\mathbb{C})$). The proof for $\overset{I}{L}{^{(-)}}$ is exactly the same since $R^{(-)}$ is also a universal $R$-matrix. For the second formula we just apply the definition of $\triangleright$:
\[ \overset{I}{L}{^{(+)}_1} \triangleright \overset{J}{T_2} = \bigl(\overset{I}{a_i}\bigr)_1 b_i \triangleright \overset{J}{T_2} = \bigl(\overset{I}{a_i}\bigr)_1 \overset{J}{T}(? b_i)_2 = \bigl(\overset{I}{a_i}\bigr)_1 \overset{J}{T_2} \bigl(\overset{J}{b_i}\bigr)_2 = \overset{J}{T_2} \bigl(\overset{I}{a_i}\bigr)_1 \bigl(\overset{J}{b_i}\bigr)_2 = \overset{J}{T_2} \overset{IJ}{R}_{12}. \]
The proof for $\overset{I}{L}{^{(-)}}$ is exactly the same. The last formula follows from \eqref{dualHopf}.
\end{proof}

Thanks to \cite{alekseev}, we know that there is a morphism from $\mathcal{L}_{1,0}(H)$ to $\mathcal{H}(\mathcal{O}(H))$:

\begin{Proposition}The following map is a morphism of algebras:
\[ \begin{array}{@{}crll}
\Psi_{1,0}\colon & \mathcal{L}_{1,0}(H) & \rightarrow & \mathcal{H}(\mathcal{O}(H)), \\
 & \overset{I}{B} &\mapsto & \overset{I}{L} \!^{(+)}\overset{I}{T}\overset{I}{L} \!^{(-)-1},\\
 & \overset{I}{A} &\mapsto & \overset{I}{L} \!^{(+)}\overset{I}{L} \!^{(-)-1}.
\end{array} \]
\end{Proposition}
\begin{proof}
One must check that the fusion and exchange relations are compatible with $\Psi_{1,0}$. Observe that the restriction of $\Psi_{1,0}$ to the first copy of $\mathcal{L}_{0,1}(H) \subset \mathcal{L}_{1,0}(H)$ is just the RSD morphism~$\Psi_{0,1}$, thus $\Psi_{1,0}$ is compatible with the fusion relation over~$A$. For the fusion relation over~$B$, we have
\begin{align*}
\Psi_{1,0}(B)_{12} &= L^{(+)}_{12} T_{12} L^{(-)-1}_{12} &&\quad \text{(definition)}\\
&= L^{(+)}_1 L^{(+)}_2 T_1 T_2 L^{(-)-1}_2 L^{(-)-1}_1 &&\quad \text{(equations~\eqref{propertiesL} and \eqref{dualHopf})}\\
&= L^{(+)}_1 T_1 L^{(+)}_2 R_{21} T_2 L^{(-)-1}_2 L^{(-)-1}_1 &&\quad \text{(Lemma~\ref{echangeHeisenberg})}\\
&= L^{(+)}_1 T_1 L^{(+)}_2 R_{21} T_2 R_{21} L^{(-)-1}_1 L^{(-)-1}_2 R_{21}^{-1} &&\quad \text{(equation~\eqref{propertiesL})}\\
&= L^{(+)}_1 T_1 L^{(+)}_2 R_{21} L^{(-)-1}_1 T_2 L^{(-)-1}_2 R_{21}^{-1} &&\quad \text{(Lemma~\ref{echangeHeisenberg})}\\
&= L^{(+)}_1 T_1 L^{(-)-1}_1 R_{21} L^{(+)}_2 T_2 L^{(-)-1}_2 R_{21}^{-1} &&\quad \text{(equation~\eqref{propertiesL})}\\
&= \Psi_{1,0}(B)_1 R_{21} \Psi_{1,0}(B)_2 R_{21}^{-1} &&\quad \text{(definition)}.
\end{align*}
The same kind of computation allows one to show that $\Psi_{1,0}$ is compatible with the $\mathcal{L}_{1,0}$-exchange relation.
\end{proof}

We will now show that $\Psi_{1,0}$ is an isomorphism under our assumptions on $H$.
\begin{Lemma}\label{lemmeDimL10}Every element in $\mathcal{L}_{1,0}(H)$ can be written as $\sum_i (x_i)_B (y_i)_A$ with $x_i, y_i \in \mathcal{L}_{0,1}(H)$. It follows that $\dim (\mathcal{L}_{1,0}(H) ) \leq \dim (\mathcal{L}_{0,1}(H) )^2 = \dim(H)^2$.
\end{Lemma}
\begin{proof}This is the same proof as in Lemma~\ref{lemmeDegre}. It suffices to show that an element like $y_A x_B$ can be expressed as $\sum_i (x_i)_B (y_i)_A$. The exchange relation can be rewritten as
\begin{gather}\label{echangeL10Inverse}
\overset{I}{A}_1\overset{J}{B}_2 = \overset{I}{S^{-1}(a_i)}_2\overset{IJ}{(R')}_{12}\overset{J}{B}_2\overset{IJ}{R}_{12}\overset{I}{A}_1\overset{IJ}{(R')}_{12}\overset{J}{(b_i)}_1,
\end{gather}
and the result follows since $\overset{I}{A}_1 \overset{J}{B}_2$ contains all the possible products between the coefficients of~$\overset{I}{A}$ and those of~$\overset{J}{B}$.
\end{proof}

\begin{Theorem}\label{isoL10Heisenberg}Recall that we assume that $H$ is a finite-dimensional factorizable Hopf algebra. $\Psi_{1,0}$ gives an isomorphism of algebras $\mathcal{L}_{1,0}(H) \cong \mathcal{H}(\mathcal{O}(H))$. It follows that $\mathcal{L}_{1,0}(H)$ is a matrix algebra: $\mathcal{L}_{1,0}(H) \cong \mathrm{Mat}_{\dim(H)}(\mathbb{C})$.
\end{Theorem}
\begin{proof}
Observe that $\Psi_{1,0} \circ j_A = i_H \circ \Psi_{0,1}$ where $i_H\colon H \to \mathcal{H}(\mathcal{O}(H))$ is the canonical inclusion. Since $\Psi_{0,1}$ is an isomorphism, there exist matrices $\overset{I}{A} \!^{(\pm)}$ such that
\[ \Psi_{1,0}\big(\overset{I}{A} \!^{(\pm)}\big) = \overset{I}{L} \!^{(\pm)} \in \mathrm{Mat}_{\dim(I)}(\mathcal{H}(\mathcal{O}(H))). \]
Moreover, we have
\[ \Psi_{1,0}\big(\overset{I}{A} \!^{(+)-1}\overset{I}{B}\overset{I}{A} \!^{(-)}\big) = \overset{I}{T} \in \mathrm{Mat}_{\dim(I)}(\mathcal{H}(\mathcal{O}(H))). \]
Thus \looseness=1 $\Psi_{1,0}$ is surjective, and hence $\dim (\mathcal{L}_{1,0}(H) ) \geq \dim (\mathcal{H}(\mathcal{O}(H)) ) = \dim(H)^2$. This together with Lemma~\ref{lemmeDimL10} gives $\dim (\mathcal{L}_{1,0}(H) ) = \dim (\mathcal{H}(\mathcal{O}(H)) )$. The last claim is a general fact, see~\eqref{HeisenbergMat}.
\end{proof}

\begin{Remark}\label{remarqueProductGaugeFields}Due to Theorem~\ref{isoL10Heisenberg}, there is an isomorphism of vector spaces $\mathcal{L}_{1,0}(H) \to H^* \otimes H^*$ given by $\overset{I}{B}{^i_j} \overset{J}{A}{^k_l} \mapsto \overset{I}{T}{^i_j} \otimes \overset{J}{T}{^k_l}$. This defines a product $\ast$ on $H^* \otimes H^*$; due to~\eqref{echangeL10Inverse}, we get that it satisfies
\[ (\varepsilon \otimes \psi) \ast (\varphi \otimes \varepsilon) = \varphi\big( S^{-1}(a_i) a_j ? b_k a_l \big) \otimes \psi\big( \vphantom{S^{-1}}b_j a_k ? b_l b_i \big). \]
Combining this with \eqref{productGaugeFields01}, we obtain the general formula
 \begin{gather*}
 (\varphi_1 \otimes \psi_1) \ast (\varphi_2 \otimes \psi_2) = \varphi_1(? b_m b_n) \varphi_2\big( S^{-1}(a_i) a_j S^{-1}(a_n) ? a_m b_k a_l \big)\\
 \hphantom{(\varphi_1 \otimes \psi_1) \ast (\varphi_2 \otimes \psi_2) =}{} \otimes \psi_1\big( \vphantom{S^{-1}}b_j a_k ? b_o b_p b_l b_i \big) \psi_2\big(S^{-1}(a_p) ? a_o\big), \\
 x_b \otimes x_a \mapsto \varphi_1(x_b' b_m b_n) \varphi_2\big( S^{-1}(a_i) a_j S^{-1}(a_n) x_b'' a_m b_k a_l \big) \\
 \hphantom{x_b \otimes x_a \mapsto}{}\times \psi_1\big( \vphantom{S^{-1}}b_j a_k x_a' b_o b_p b_l b_i \big) \psi_2\big(S^{-1}(a_p) x_a'' a_o\big).
\end{gather*}
This is the product of the gauge fields $\varphi_1 \otimes \psi_1$, $\varphi_2 \otimes \psi_2$ and its evaluation on the connection which assigns $x_b$ to the loop $b$ and $x_a$ to the loop $a$, see Fig.~\ref{figureSurfaces}. The action of the gauge algebra~$H$ on a gauge field $\varphi \otimes \psi$ is $(\varphi \otimes \psi) \cdot h = \varphi (h' ? S(h'') ) \otimes \psi (h''' ? S(h'''') )$ and $\ast$ is $H$-equivariant.
\end{Remark}

\subsection[Representation of $\mathcal{L}_{1,0}^{\mathrm{inv}}(H)$ on $\SLF(H)$]{Representation of $\boldsymbol{\mathcal{L}_{1,0}^{\mathrm{inv}}(H)}$ on $\SLF(H)$}\label{sectionRepInvariants}

Recall from \eqref{repHeisenberg} that there is a faithful representation $\triangleright$ of $\mathcal{H}(\mathcal{O}(H))$ on $\mathcal{O}(H)$. Using the isomorphism $\Psi_{1,0}$, we get a representation of $\mathcal{L}_{1,0}(H)$ on $\mathcal{O}(H)$, still denoted~$\triangleright$:
\[ \forall\, x \in \mathcal{L}_{1,0}(H), \qquad \forall\, \psi \in \mathcal{O}(H), \qquad x \triangleright \psi = \Psi_{1,0}(x) \triangleright \psi. \]
Thanks to Lemma~\ref{echangeHeisenberg}, it is easy to get
\begin{gather}
\overset{I}{A}_1 \triangleright \overset{J}{T}\!_2 = \overset{J}{T}\!_2 \overset{IJ}{(RR')}_{12}, \nonumber\\
\overset{I}{B}_1 \triangleright \overset{J}{T}\!_2 = \overset{I}{(a_i)}_1\overset{I \otimes J}{T}\!\!\!_{12} \overset{I \otimes J}{(b_i)}_{12}\overset{IJ}{(R')}_{12} = \overset{I}{(a_ia_j)}_1\overset{I \otimes J}{T}\!\!\!_{12} \overset{I}{(b_j)}_{1}\overset{J}{(b_i)}_{2}\overset{IJ}{(R')}_{12},\label{actionL10}
\end{gather}
where as usual $R = a_i \otimes b_i$ and the last equality is obtained using~\eqref{deltaR}.

In \cite[Theorem~5]{alekseev} (which is stated in the case of $H = U_q(\mathfrak{g})$, $q$ generic), there is a representation of $\mathcal{L}_{1,0}^{\text{inv}}(H)$ on a subspace of invariants in $H^*$. This can be generalized to our assumptions.

\begin{Theorem}\label{repInv}The restriction of $\triangleright$ to $\mathcal{L}_{1,0}^{\text{\em inv}}(H)$ leaves the subspace $\SLF(H) \subset H^*$ stable
\[ \forall\, x \in \mathcal{L}^{\text{\em inv}}_{1,0}(H), \qquad \forall\, \psi \in \SLF(H), \qquad x \triangleright \psi \in \SLF(H). \]
Hence, we have a representation of $\mathcal{L}_{1,0}^{\text{\rm inv}}(H)$ on $\SLF(H)$. We denote it $\rho_{\mathrm{SLF}}$.
\end{Theorem}
\begin{proof}For $h \in H$, define $\widetilde{h} \in \mathcal{H}(\mathcal{O}(H))$ by $\widetilde{h} \triangleright \varphi = \varphi\bigl( S^{-1}(h) ? \bigr)$ for all $\varphi \in H^*$ (since the representation $\triangleright$ is faithful, this entirely defines $\widetilde{h}$). It is easy to see that
\[
\forall\, g \in H, \qquad \forall\, \psi \in \mathcal{O}(H), \qquad \widetilde{g}\widetilde{h} = \widetilde{gh}, \qquad g\widetilde{h} = \widetilde{h}g, \qquad \widetilde{h}\psi = \psi\big(S^{-1}(h'')?\big)\widetilde{h'}.
\]
We define matrices
\begin{gather*}
\overset{I}{\widetilde{L}}{^{(\pm)}} = \bigl( \overset{I}{a_i^{(\pm)}} \bigr) \widetilde{b_i^{(\pm)}} \in \mathrm{Mat}_{\dim(I)}\bigl(\mathcal{H}(\mathcal{O}(H))\bigr)
\end{gather*}
with $R^{(\pm)} = a_i^{(\pm)} \otimes b_i^{(\pm)}$. By definition and \eqref{propSR}, they satisfy $\overset{I}{\widetilde{L}}{^{(\pm)}_1} \triangleright \overset{J}{T}_2 = \overset{IJ}{R}{^{(\pm)-1}_{12}}\overset{J}{T}_2$. It is not difficult to show the following commutation rules
\begin{gather}\label{LTilde}
\overset{I}{\widetilde{L}}{^{(\epsilon)}_1} \overset{J}{\widetilde{L}}{^{(\epsilon)}_2} = \overset{I\otimes J}{\widetilde{L}}\!\!{^{(\epsilon)}_{12}}, \qquad \overset{I}{\widetilde{L}}{^{(\epsilon)}_1}\overset{J}{L}{^{(\sigma)}_2} = \overset{J}{L}{^{(\sigma)}_2}\overset{I}{\widetilde{L}}{^{(\epsilon)}_1}, \qquad \overset{IJ}{R}{^{(\epsilon)}_{12}} \overset{I}{\widetilde{L}}{^{(\epsilon)}_1}\overset{J}{T}_2 = \overset{J}{T}_2\overset{I}{\widetilde{L}}{^{(\epsilon)}_1},\\
\overset{IJ}{R} \!^{(\epsilon)}_{12} \overset{I}{\widetilde{L}} \!^{(\epsilon)}_1 \overset{J}{\widetilde{L}} \!^{(\sigma)}_2 = \overset{J}{\widetilde{L}} \!^{(\sigma)}_2 \overset{I}{\widetilde{L}} \!^{(\epsilon)}_1 \overset{IJ}{R} \!^{(\epsilon)}_{12} \!\qquad \forall\, \epsilon, \sigma \in \{\pm\}, \qquad \overset{IJ}{R} \!^{(\epsilon)}_{12} \overset{I}{\widetilde{L}} \!^{(\sigma)}_1 \overset{J}{\widetilde{L}} \!^{(\sigma)}_2 = \overset{J}{\widetilde{L}} \!^{(\sigma)}_2 \overset{I}{\widetilde{L}} \!^{(\sigma)}_1 \overset{IJ}{R} \!^{(\epsilon)}_{12} \!\qquad \forall\, \epsilon, \sigma \in \{\pm\}.\nonumber
\end{gather}
For instance, here is a proof of the third equality with $\epsilon=+$:
\begin{align*}
\overset{IJ}{R}{^{(+)}_{12}} \overset{I}{\widetilde{L}}{^{(+)}_1} \overset{J}{T}_2 &= \overset{IJ}{R}{^{(+)}_{12}} \bigl( \overset{I}{a_i} \bigr)_1 \widetilde{b_i} \overset{J}{T_2} = \overset{IJ}{R}{^{(+)}_{12}} \bigl( \overset{I}{a_i} \bigr)_1 \overset{J}{T}\bigl( S^{-1}(b_i'') ? \bigr)_2 \widetilde{b_i'} = \overset{IJ}{R}{^{(+)}_{12}} \bigl( \overset{I}{a_i} \overset{I}{a_j} \bigr)_1 \bigl( \overset{J}{S^{-1}(b_i)} \bigr)_2 \overset{J}{T_2} \widetilde{b_j}\\
& = \overset{IJ}{R}{^{(+)}_{12}} \bigl( \overset{I}{a_i} \bigr)_1 \bigl( \overset{J}{S^{-1}(b_i)} \bigr)_2 \overset{J}{T_2} \bigl( \overset{I}{a_j} \bigr)\widetilde{b_j} = \overset{J}{T_2} \overset{I}{\widetilde{L}}{^{(+)}_1},
\end{align*}
since $a_i \otimes S^{-1}(b_i) = R^{(+)-1}$. Now, let
\[ \overset{I}{C}{^{(\pm)}} = \Psi_{1,0}^{-1}\big(\overset{I}{L}{^{(\pm)}} \overset{I}{\widetilde{L}}{^{(\pm)}}\big) \in \mathrm{Mat}_{\dim(I)}\!\left(\mathcal{L}_{1,0}(H)\right). \]
Thanks to~\eqref{propertiesL} and~\eqref{LTilde}, it is easy to see that $\overset{I \otimes J}{C}\!{^{(\pm)}_{12}} = \overset{I}{C}{^{(\pm)}_1} \overset{J}{C}{^{(\pm)}_2}$.
\begin{Lemma}\quad\begin{enumerate}\itemsep=0pt
\item[$1)$] An element $x \in \mathcal{L}_{1,0}(H)$ is invariant under the right action of $H$ if, and only if, $x\overset{I}{C}{^{(\pm)}}= \overset{I}{C}{^{(\pm)}}x$ for all $I$.
\item[$2)$] A linear form $\psi \in H^*$ is symmetric if, and only if, $\overset{I}{C}{^{(\pm)}} \triangleright \psi = \psi \mathbb{I}_{\dim(I)}$ for all~$I$.
\end{enumerate}
\end{Lemma}
\begin{proof}1) Let $U = A$ or $B$, then by \eqref{L}, \eqref{actionHsurL10} and \eqref{deltaR} we get
\begin{align*}
\overset{J}{U}_2 \cdot S^{-1}\bigl(\overset{I}{L}{^{(\pm)}_1}\bigr) &= \overset{J}{U}_2 \cdot S^{-1}\big(b_i^{(\pm)}\big) \bigl( \overset{I}{a_i^{(\pm)}} \bigr)_1 = \overset{J}{S^{-1}\big(\left.b_i^{(\pm)}\right.''\big)}_2 \overset{J}{U}_2 \bigl(\overset{J}{\left. b_i^{(\pm)}\right.'}\bigr)_2 \bigl(\overset{I}{a_i^{(\pm)}} \bigr)_1\\
& = \overset{J}{S^{-1}\bigl(b_i^{(\pm)}\bigr)}_2 \overset{J}{U}_2 \bigl(\overset{J}{b_j^{(\pm)}}\bigr)_2 \bigl(\overset{I}{a_i^{(\pm)}} \overset{I}{a_j^{(\pm)}} \bigr)_1
= \bigl(\overset{I}{a_i^{(\pm)}}\bigr)_1 \overset{J}{S^{-1}\bigl(b_i^{(\pm)}\bigr)}_2 \overset{J}{U}_2 \bigl(\overset{I}{a_j^{(\pm)}} \bigr)_1 \bigl(\overset{J}{b_j^{(\pm)}}\bigr)_2\\
& = \overset{IJ}{R}{^{(\pm)-1}_{12}} \overset{J}{U}_2 \overset{IJ}{R}{^{(\pm)}_{12}}
\end{align*}
with $R^{(\pm)} = a_i^{(\pm)} \otimes b_i^{(\pm)}$. Second, using \eqref{propertiesL} and \eqref{LTilde} we get
\[ \overset{I}{C}{^{(\pm)}_1} \overset{J}{U}_2 \overset{I}{C}{^{(\pm)-1}_1} = \overset{IJ}{R}{^{(\pm)-1}_{12}} \overset{J}{U}_2 \overset{IJ}{R}{^{(\pm)}_{12}}. \]
For instance (with the shortened notation)
\begin{align*}
\Psi_{1,0}\big(C^{(+)}_1 A_2 C^{(+)-1}_1\big) &= L^{(+)}_1 \widetilde{L}^{(+)}_1 L^{(+)}_2 L^{(-)-1}_2 \widetilde{L}^{(+)-1}_1 L^{(+)-1}_1 = L^{(+)}_1 L^{(+)}_2 L^{(-)-1}_2 L^{(+)-1}_1\\
&= R_{12}^{-1} L^{(+)}_2 L^{(+)}_1 R_{12} L^{(-)-1}_2 L^{(+)-1}_1 = R_{12}^{-1} L^{(+)}_2 L^{(-)-1}_2 R_{12}\\
& = \Psi_{1,0}\big(R_{12}^{-1} A_2 R_{12}\big)
\end{align*}
and we have equality since $\Psi_{1,0}$ is an isomorphism; the others cases are similar. It follows that $\overset{J}{U}_2 \cdot S^{-1}\bigl(\overset{I}{L}{^{(\pm)}_1}\bigr) = \overset{I}{C}{^{(\pm)}_1} \overset{J}{U}_2 \overset{I}{C}{^{(\pm)-1}_1}$ or in other words $\overset{J}{U}{^c_d} \cdot S^{-1}\bigl(\overset{I}{L}{^{(\pm)}}{^a_b}\bigr) = \overset{I}{C}{^{(\pm)}}{^a_i} \overset{J}{U}{^c_d} \overset{I}{C}{^{(\pm)-1}}{^i_b}$, which means that $\overset{J}{U}{^c_d}$ is invariant under the action of $S^{-1}\bigl(\overset{I}{L}{^{(\pm)}}{^a_b}\bigr)$ if, and only if, it commutes with~$\overset{I}{C}{^{(\pm)}}{^a_b}$. Since the elements $\overset{J}{U}{^c_d}$ (resp.~$S^{-1}\bigl(\overset{I}{L}{^{(\pm)}}{^a_b}\bigr)$) generate $\mathcal{L}_{1,0}(H)$ (resp.~$H$) as an algebra, we get that an element is invariant if, and only if, it commutes with all the coefficients of the matrices $\overset{I}{C}{^{(\pm)}}$, as desired.

2) Consider the left action $\diamond$ of $H$ on $H^*$ given by $h \diamond \psi = \psi\big(S^{-1}(h') ? h''\big)$. It is easy to see that $\psi$ is symmetric if, and only if, it is invariant under $\diamond$ (namely $h \diamond \psi = \varepsilon(h)\psi$ for all $h \in H$). The definitions and \eqref{deltaR} yields
\begin{align*} \overset{I}{C}{^{(\pm)}} \triangleright \psi & = \bigl( \overset{I}{a_i^{(\pm)}} \overset{I}{a_j^{(\pm)}} \bigr) b_i \widetilde{b_j} \triangleright \psi = \bigl( \overset{I}{a_i^{(\pm)}} \overset{I}{a_j^{(\pm)}} \bigr) \psi\bigl( S^{-1}(b_j) ? b_i \bigr)\\
& = \bigl( \overset{I}{a_i^{(\pm)}} \bigr) \psi\bigl( S^{-1}(b_i') ? b_i'' \bigr) = \overset{I}{L}{^{(\pm)}} \diamond \psi. \end{align*}
Since the coefficients of the matrices $\overset{I}{L}{^{(\pm)}}$ generate $H$ as an algebra, we get the result.
\end{proof}

\textbf{End of the proof of Theorem~\ref{repInv}.} Let $x \in \mathcal{L}_{1,0}^{\text{inv}}(H)$ and $\psi \in \SLF(H)$. We apply the previous lemma
\[ \overset{J}{C}{^{(\pm)}} \triangleright (x \triangleright \psi ) = \big(\overset{J}{C}{^{(\pm)}} x \big) \triangleright \psi = \big(x\overset{J}{C}{^{(\pm)}} \big) \triangleright \psi = x \triangleright \big(\overset{J}{C}{^{(\pm)}} \triangleright \psi\big) = (x \triangleright \psi) \mathbb{I}_{\dim(J)}.\]
Then $x \triangleright \psi \in \SLF(H)$, as desired.
\end{proof}

\begin{Remark}\label{remarqueC}Let $\overset{I}{C} = \overset{I}{v^2}\overset{I}{B}\overset{I}{A}{^{-1}}\overset{I}{B}{^{-1}}\overset{I}{A}$. It can be shown that $\overset{I}{C}$ satisfies the fusion relation of~$\mathcal{L}_{0,1}(H)$, that $\overset{I}{C} = \overset{I}{C}{^{(+)}}\overset{I}{C}{^{(-)-1}}$, that $x \in \mathcal{L}_{1,0}(H)$ is invariant if, and only if, $x \overset{I}{C} = \overset{I}{C}x$ and that $\psi \in H^*$ is symmetric if, and only if, $\overset{I}{C} \triangleright \psi = \psi \mathbb{I}_{\dim(I)}$. The details in our general setting are given in \cite{F2} for arbitrary genus. Observe that geometrically, $\overset{I}{C}$ is the boundary of the sur\-face~$\Sigma_{1,0} {\setminus} D$, see Fig.~\ref{figureSurfaces}.
\end{Remark}

We now need to determine explicit formulas for the representation of particular types of invariants that will appear in the proof of the modular identities in Section~\ref{sectionModulaire}. If $\psi \in H^*$ and $a \in H$, we define
\[ \psi^a = \psi(a?), \]
where $\psi(a?)\colon x \mapsto \psi(ax)$. This defines a right representation of~$H$ on~$H^*$. Obviously, if $z \in \mathcal{Z}(H)$ and $\psi \in \SLF(H)$ then $\psi^z \in \SLF(H)$.

Recall that $z_A = j_A(z)$ (resp.~$z_B = j_B(z)$) is the image of $z \in \mathcal{L}_{0,1}(H)$ by the map $j_A(\overset{I}{M}) = \overset{I}{A}$ (resp.~$j_B(\overset{I}{M}) = \overset{I}{B}$). See~\eqref{coinvParticuliers} for the general definition.
\begin{Proposition}\label{actionAB}
Let $z \in \mathcal{L}_{0,1}^{\mathrm{inv}}(H) = \mathcal{Z}(H)$ and let $\psi \in \mathrm{SLF}(H)$. Then
\[ z_A \triangleright \psi = \psi^z \qquad \text{and} \qquad z_B \triangleright \psi = \big(\mathcal{D}^{-1}(z)\psi^v\big)^{v^{-1}}, \]
where $\mathcal{D}$ is the isomorphism defined in~\eqref{morphismeDrinfeld}.
\end{Proposition}
\begin{proof}The first relation is obvious. For the second formula, we write $z_B = \sum_I \mathrm{tr}\big(\Lambda_I\overset{I}{g}\overset{I}{B}\big)$ with $\mathcal{D}^{-1}(z) = \sum_I \mathrm{tr}\big(\Lambda_I\overset{I}{T}\big) \in \SLF(H)$ by Lemma~\ref{writingInvariants}. We also write $\psi = \sum_J \mathrm{tr}\big(\Theta_J\overset{J}{T}\big)$. Then, thanks to~\eqref{actionL10}:
 \begin{align*}
z_B \triangleright \psi &= \sum_{I,J} \mathrm{tr}_{12}\big((\Lambda_I)_1(\Theta_J)_2 \overset{I}{g}_1\overset{I}{B}_1 \triangleright \overset{J}{T}_2\big)\\
&=\sum_{I,J} \mathrm{tr}_{12}\big((\Lambda_I)_1(\Theta_J)_2 \overset{I}{g}_1\overset{I}{(a_ia_j)}_1\overset{I \otimes J}{T}\!\!\!_{12} \overset{I}{(b_j)}_{1}\overset{J}{(b_i)}_{2}\overset{IJ}{(R')}_{12}\big)\\
&= \mathcal{D}^{-1}(z) ( g a_i a_j ? b_j b_k ) \psi ( ?b_ia_k ) = \mathcal{D}^{-1}(z)\big( ? b_j b_k S^2(a_ia_j) g\big) \psi ( ?b_ia_k )
\end{align*}
with $\mathrm{tr}_{12} = \mathrm{tr} \otimes \mathrm{tr}$, $R = a_i \otimes b_i$. Thanks to the Yang--Baxter equation, we have
\[ b_j b_k \otimes a_ia_j \otimes b_ia_k = R_{23} R_{21} R_{31} = R_{31} R_{21} R_{23} = b_ib_j \otimes a_ja_k \otimes a_ib_k. \]
It follows that
 \begin{align*}
z_B \triangleright \psi &= \mathcal{D}^{-1}(z)\big( ? b_i b_j S^2(a_j a_k) g\big) \psi ( ? a_ib_k )\\
&= \mathcal{D}^{-1}(z)\big(? v^{-1} b_i a_k\big) \psi ( ? a_ib_k ) = \mathcal{D}^{-1}(z)\big(? \big(v^{-1}\big)'\big) \psi\big( ? v \big(v^{-1}\big)'' \big),
\end{align*}
where we used \eqref{elementDrinfeld}, \eqref{pivotCan} and \eqref{ribbon}. Hence for $x \in H$:
\begin{gather*} (z_B \triangleright \psi ) (x) = \mathcal{D}^{-1}(z)\big(\big(v^{-1}\big)'x'\big) \psi\big(v\big(v^{-1}\big)''x''\big) = \big(\mathcal{D}^{-1}(z) \psi^v\big)\big(v^{-1}x\big) = \big(\mathcal{D}^{-1}(z) \psi^v\big)^{v^{-1}} (x) \end{gather*}
as desired.
\end{proof}

\begin{Lemma}\label{lemmeBMoinsUn}
Let $z \in \mathcal{L}_{0,1}^{\mathrm{inv}}(H) = \mathcal{Z}(H)$ and let $\psi \in \SLF(H)$. Then
\[ z_{B^{-1}} \triangleright \psi = \big( S\big(\mathcal{D}^{-1}(z)\big) \psi^v \big)^{v^{-1}}. \]
It follows that if $S(\psi) = \psi$ for all $\psi \in \SLF(H)$, then $\rho_{\mathrm{SLF}}(z_{B^{-1}}) = \rho_{\mathrm{SLF}}(z_B)$.
\end{Lemma}
\begin{proof}This proof is quite similar to that of the previous proposition. Due to the fact that $\Psi_{1,0}\big(\overset{I}{B} \!^{-1}\big) = \overset{I}{L} \!^{(-)} S(\overset{I}{T}) \overset{I}{L} \!^{(+)-1} $ together with Lemma~\ref{echangeHeisenberg} and formulas \eqref{antipodeT}, \eqref{propSR} and \eqref{deltaR}, it is not too difficult to show that
\[ \overset{I}{B} \!^{-1}_1 \triangleright \overset{J}{T}\!_2 = \exposantGauche{t \otimes \mathrm{id}\!\!}{\big( \overset{I^*}{(a_i)}_1 \overset{I^* \otimes J}{T}\!\!\!\!_{12} \overset{I^*}{\big(a_jS^{-2}(b_jb_k)\big)}_1\overset{J}{(a_kb_i)}_2 \big)}, \]
where $^{t \otimes \mathrm{id}}$ means transpose on the first tensorand. Write $z_{B^{-1}} = \sum_I \mathrm{tr}\big(\Lambda_I\overset{I}{g}\overset{I}{B} \!^{-1}\big)$ with $\mathcal{D}^{-1}(z) = \sum_I \mathrm{tr}\big(\Lambda_I\overset{I}{T}\big) \in \SLF(H)$ by Lemma~\ref{writingInvariants}, and $\psi = \sum_J \mathrm{tr}\big(\Theta_J\overset{J}{T}\big)$. Thanks to~\eqref{antipodeT}, observe that
\[ S\big(\mathcal{D}^{-1}(z)\big) = \sum_I \mathrm{tr}\big( \Lambda_I S(\overset{I}{T})\big) = \sum_I \mathrm{tr}\big( \exposantGauche{t}{\Lambda_I} \overset{I^*}{T}\big). \]
Using the fact that $S(g) = g^{-1}$ and~\eqref{antipodeT}, we thus get
 \begin{align*}
z_{B^{-1}} \triangleright \psi &= \sum_{I,J} \mathrm{tr}_{12}\big( (\Lambda_I\overset{I}{g})_1(\Theta_J)_2 \exposantGauche{t \otimes \mathrm{id}\!\!}{\big( \overset{I^*}{(a_i)}_1 \overset{I^* \otimes J}{T}\!\!\!\!_{12} \overset{I^*}{(a_jS^{-2}(b_jb_k))}_1\overset{J}{(a_kb_i)}_2 \big)} \big)\\
&= \sum_{I,J} \mathrm{tr}_{12}\big( (\exposantGauche{t}{\Lambda_I})_1(\Theta_J)_2 \overset{I^*}{(a_i)}_1 \overset{I^* \otimes J}{T}\!\!\!\!_{12} \overset{I^*}{(a_jS^{-2}(b_jb_k)g^{-1})}_1\overset{J}{(a_kb_i)}_2 \big)\\
&= S\big(\mathcal{D}^{-1}(z)\big)\big( a_i ? a_jS^{-2}(b_jb_k)g^{-1} \big) \psi (? a_kb_i )\\
&= S\big(\mathcal{D}^{-1}(z)\big)\big( ? a_jS^{-2}(b_jb_k)g^{-1}a_i \big) \psi (? a_kb_i ) = S\big(\mathcal{D}^{-1}(z)\big)\big( ?\big(v^{-1}\big)'\big) \psi\big(? v \big(v^{-1}\big)''\big).
\end{align*}
For the last equality we used \eqref{elementDrinfeld}, \eqref{pivotCan} and \eqref{ribbon}. Hence we get as in the previous proof $ z_{B^{-1}} \triangleright \psi = \big(S\big(\mathcal{D}^{-1}(z)\big)\psi^v\big)^{v^{-1}}$.
\end{proof}

\section[Projective representation of $\mathrm{SL}_2(\mathbb{Z})$]{Projective representation of $\boldsymbol{\mathrm{SL}_2(\mathbb{Z})}$}\label{sectionModulaire}
As previously, $H$ is a finite-dimensional factorizable ribbon Hopf algebra.

\subsection{Mapping class group of the torus}
Let $\Sigma_{g,n}$ be a compact oriented surface of genus $g$ with $n$ open disks removed. Recall that the mapping class group $\mathrm{MCG}(\Sigma_{g,n})$ is the group of all isotopy classes of orientation-preserving homeomorphisms which fix the boundary pointwise, see~\cite{FM}.

Let us put the base point of $\pi_1(\Sigma_{g,n}{\setminus} D)$ on the boundary circle $c$. Since $c$ is pointwise fixed, we can consider the action of $\mathrm{MCG}(\Sigma_{g,n}{\setminus} D)$ on $\pi_1(\Sigma_{g,n}{\setminus} D)$, obviously defined by{\samepage
\[ \forall\, [f] \in\mathrm{MCG}(\Sigma_{g,n}), \qquad \forall\, [\gamma] \in \pi_1(\Sigma_{g,n} {\setminus} D), \qquad [f]\cdot[\gamma] = f_*([\gamma]) = [f \circ \gamma]. \]
Until now, we identify $f$ with its isotopy class $[f]$ and $\gamma$ with its homotopy class $[\gamma]$.}

Here we focus on the torus $\Sigma_{1,0} = S^1 \times S^1$. Consider $\Sigma_{1,0} {\setminus} D$, where $D$ is an embedded open disk. The surface $\Sigma_{1,0} {\setminus} D$ is represented as a ribbon graph in Fig.~\ref{figureSurfaces} together with the canonical curves $a$ and $b$. The groups $\mathrm{MCG}(\Sigma_{1,0} {\setminus} D)$ and $\mathrm{MCG}(\Sigma_{1,0})$ are generated by the Dehn twists~$\tau_a$,~$\tau_b$ about the free homotopy class of the curves~$a$ and~$b$. It is well-known (see~\cite{FM}) that
\[ \mathrm{MCG}(\Sigma_{1,0}) = \mathrm{SL}_2(\mathbb{Z}) = \big\langle \tau_a, \tau_b \,|\, \tau_a\tau_b\tau_a = \tau_b\tau_a\tau_b, \, (\tau_a\tau_b)^6=1 \big\rangle. \]
This presentation is not the usual one of $\mathrm{SL}_2(\mathbb{Z})$, which is
\[ \mathrm{SL}_2(\mathbb{Z}) = \big\langle s, t \,|\, (st)^3=s^2, \, s^4=1 \big\rangle. \]
The link between the two presentations is $s = \tau_a^{-1}\tau_b^{-1}\tau_a^{-1}$, $t = \tau_a$.

The action of the Dehn twists $\tau_a$ and $\tau_b$ on $\pi_1(\Sigma_{1,0}{\setminus} D)$ is given by
\begin{gather}\label{actionMCG}
(\tau_a)_*(a) = a, \qquad (\tau_a)_*(b) = ba \qquad \text{and} \qquad (\tau_b)_*(a) = b^{-1}a, \qquad (\tau_b)_*(b) = b.
\end{gather}
For instance, the action of $\tau_a$ on $b$ is depicted by
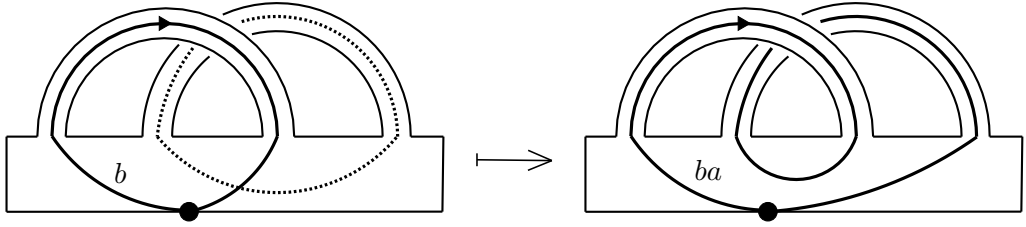
\begin{center}
\begin{tabular}{ccc}
\begin{tikzpicture} 
\draw [shift={(7.116832178429767,6)},line width=0.8pt] plot[domain=0:3.141592653589793,variable=\t]({1*1.7075037526311627*cos(\t r)+0*1.7075037526311627*sin(\t r)},{0*1.7075037526311627*cos(\t r)+1*1.7075037526311627*sin(\t r)});\draw [shift={(7.095787711044769,6)},line width=0.8pt] plot[domain=0:3.141592653589793,variable=\t]({1*1.309575195744216*cos(\t r)+0*1.309575195744216*sin(\t r)},{0*1.309575195744216*cos(\t r)+1*1.309575195744216*sin(\t r)});\draw [shift={(8.596936063235535,6)},line width=0.8pt] plot[domain=0:1.8118338027760237,variable=\t]({1*1.4030639367644646*cos(\t r)+0*1.4030639367644646*sin(\t r)},{0*1.4030639367644646*cos(\t r)+1*1.4030639367644646*sin(\t r)});\draw [shift={(8.596936063235535,6)},line width=0.8pt] plot[domain=0:1.9894438844649243,variable=\t]({1*1.776642756206547*cos(\t r)+0*1.776642756206547*sin(\t r)},{0*1.776642756206547*cos(\t r)+1*1.776642756206547*sin(\t r)});\draw [shift={(8.596936063235535,6)},line width=0.8pt] plot[domain=2.2704846801922587:3.141592653589793,variable=\t]({1*1.3953711574291767*cos(\t r)+0*1.3953711574291767*sin(\t r)},{0*1.3953711574291767*cos(\t r)+1*1.3953711574291767*sin(\t r)});\draw [shift={(8.596936063235535,6)},line width=0.8pt] plot[domain=2.3877511379579315:3.141592653589793,variable=\t]({1*1.789237366642512*cos(\t r)+0*1.789237366642512*sin(\t r)},{0*1.789237366642512*cos(\t r)+1*1.789237366642512*sin(\t r)});\draw [line width=0.8pt] (5,6)-- (5,5);\draw [line width=0.8pt] (5,5)-- (10.794708954246598,5.00203777401749);\draw [line width=0.8pt] (10.80786927095924,6)-- (10.794708954246598,5.00203777401749);\draw [line width=0.8pt] (5,6)-- (5.409328425798604,6);\draw [line width=0.8pt] (5.786212515300553,6)-- (6.8080282517522,6);\draw [line width=0.8pt] (7.200815462007147,6)-- (8.405362906788985,6);\draw [line width=0.8pt] (8.82433593106093,6)-- (10,6);\draw [line width=0.8pt] (10.373578819442082,6)-- (10.80786927095924,6);\draw [shift={(7.449580695168613,7.237726613631109)},line width=1.2pt] plot[domain=3.729595257137362:4.701699287460754,variable=\t]({1*2.2187359355090455*cos(\t r)+0*2.2187359355090455*sin(\t r)},{0*2.2187359355090455*cos(\t r)+1*2.2187359355090455*sin(\t r)});\draw [shift={(7.1026186594434115,6.006993354832272)},line width=1.2pt] plot[domain=0:3.141592653589793,variable=\t]({1*1.4991378524730479*cos(\t r)+0*1.4991378524730479*sin(\t r)},{0*1.4991378524730479*cos(\t r)+1*1.4991378524730479*sin(\t r)});\draw [shift={(9.33211204597329,5.747220848195546)},line width=1.2pt,dash pattern=on 1pt off 1pt] plot[domain=2.484254690839589:3.0336235085501446,variable=\t]({1*2.3423271124786536*cos(\t r)+0*2.3423271124786536*sin(\t r)},{0*2.3423271124786536*cos(\t r)+1*2.3423271124786536*sin(\t r)});\draw [shift={(7.055663711022826,6.6244133629535815)},line width=1.2pt] plot[domain=4.936459140732812:5.903246179167029,variable=\t]({1*1.6651881573061873*cos(\t r)+0*1.6651881573061873*sin(\t r)},{0*1.6651881573061873*cos(\t r)+1*1.6651881573061873*sin(\t r)});\draw [shift={(8.59002279079864,5.988921534198891)},line width=1.2pt,dash pattern=on 1pt off 1pt] plot[domain=0.006881023462768673:1.8643557340986399,variable=\t]({1*1.6100153249433073*cos(\t r)+0*1.6100153249433073*sin(\t r)},{0*1.6100153249433073*cos(\t r)+1*1.6100153249433073*sin(\t r)});\draw [shift={(8.59812891194914,7.333623827966784)},line width=1.2pt,dash pattern=on 1pt off 1pt] plot[domain=3.8382059511494124:5.588914646331431,variable=\t]({1*2.079092437930187*cos(\t r)+0*2.079092437930187*sin(\t r)},{0*2.079092437930187*cos(\t r)+1*2.079092437930187*sin(\t r)});

\draw [fill=black] (7.425668292444792,5.00085301330238) circle (3.5pt);\draw[color=black] (6.5129532506793275,5.519243771113981) node {$b$};
\draw [fill=black,shift={(7.083362211564692,7.506007527524644)},rotate=270] (0,0) ++(0 pt,3pt) -- ++(2.598076211353316pt,-4.5pt)--++(-5.196152422706632pt,0 pt) -- ++(2.598076211353316pt,4.5pt);

\end{tikzpicture}
&
\begin{tikzpicture}
\draw [line width=0.6pt] (4.998800483268871,4.8997262304289775)-- (4.998800483268871,4.728481035341609);\draw [line width=0.6pt] (4.998800483268871,4.814103632885294)-- (5.994612639366005,4.803256746369281);\draw [line width=0.6pt] (5.994612639366005,4.803256746369281)-- (5.679357331774752,4.939867379658825);\draw [line width=0.6pt] (5.676456420546011,4.6735455497886065)-- (5.994612639366005,4.803256746369281);\begin{scriptsize}\draw [color=white] (5.4146305482354515,4.0174688214917404) circle (0.5pt);\end{scriptsize}
\end{tikzpicture}
&
\begin{tikzpicture}
\draw [shift={(7.116832178429767,6)},line width=0.8pt] plot[domain=0:3.141592653589793,variable=\t]({1*1.7075037526311627*cos(\t r)+0*1.7075037526311627*sin(\t r)},{0*1.7075037526311627*cos(\t r)+1*1.7075037526311627*sin(\t r)});\draw [shift={(7.095787711044769,6)},line width=0.8pt] plot[domain=0:3.141592653589793,variable=\t]({1*1.309575195744216*cos(\t r)+0*1.309575195744216*sin(\t r)},{0*1.309575195744216*cos(\t r)+1*1.309575195744216*sin(\t r)});\draw [shift={(8.596936063235535,6)},line width=0.8pt] plot[domain=0:1.8118338027760237,variable=\t]({1*1.4030639367644646*cos(\t r)+0*1.4030639367644646*sin(\t r)},{0*1.4030639367644646*cos(\t r)+1*1.4030639367644646*sin(\t r)});\draw [shift={(8.596936063235535,6)},line width=0.8pt] plot[domain=0:1.9894438844649243,variable=\t]({1*1.776642756206547*cos(\t r)+0*1.776642756206547*sin(\t r)},{0*1.776642756206547*cos(\t r)+1*1.776642756206547*sin(\t r)});\draw [shift={(8.596936063235535,6)},line width=0.8pt] plot[domain=2.2704846801922587:3.141592653589793,variable=\t]({1*1.3953711574291767*cos(\t r)+0*1.3953711574291767*sin(\t r)},{0*1.3953711574291767*cos(\t r)+1*1.3953711574291767*sin(\t r)});\draw [shift={(8.596936063235535,6)},line width=0.8pt] plot[domain=2.3877511379579315:3.141592653589793,variable=\t]({1*1.789237366642512*cos(\t r)+0*1.789237366642512*sin(\t r)},{0*1.789237366642512*cos(\t r)+1*1.789237366642512*sin(\t r)});\draw [line width=0.8pt] (5,6)-- (5,5);\draw [line width=0.8pt] (5,5)-- (10.794708954246598,5.00203777401749);\draw [line width=0.8pt] (10.80786927095924,6)-- (10.794708954246598,5.00203777401749);\draw [line width=0.8pt] (5,6)-- (5.409328425798604,6);\draw [line width=0.8pt] (5.786212515300553,6)-- (6.8080282517522,6);\draw [line width=0.8pt] (7.200815462007147,6)-- (8.405362906788985,6);\draw [line width=0.8pt] (8.82433593106093,6)-- (10,6);\draw [line width=0.8pt] (10.373578819442082,6)-- (10.80786927095924,6);\draw [shift={(7.449580695168613,7.237726613631109)},line width=1.2pt] plot[domain=3.729595257137362:4.701699287460754,variable=\t]({1*2.2187359355090455*cos(\t r)+0*2.2187359355090455*sin(\t r)},{0*2.2187359355090455*cos(\t r)+1*2.2187359355090455*sin(\t r)});\draw [shift={(7.1026186594434115,6.006993354832272)},line width=1.2pt] plot[domain=0:3.141592653589793,variable=\t]({1*1.4991378524730479*cos(\t r)+0*1.4991378524730479*sin(\t r)},{0*1.4991378524730479*cos(\t r)+1*1.4991378524730479*sin(\t r)});\draw [shift={(9.33211204597329,5.747220848195546)},line width=1.2pt] plot[domain=2.484254690839589:3.0336235085501446,variable=\t]({1*2.3423271124786536*cos(\t r)+0*2.3423271124786536*sin(\t r)},{0*2.3423271124786536*cos(\t r)+1*2.3423271124786536*sin(\t r)});\draw [shift={(8.59002279079864,5.988921534198891)},line width=1.2pt] plot[domain=0.006881023462768673:1.8643557340986399,variable=\t]({1*1.6100153249433073*cos(\t r)+0*1.6100153249433073*sin(\t r)},{0*1.6100153249433073*cos(\t r)+1*1.6100153249433073*sin(\t r)});\draw [shift={(7.024070995031567,10.499990747896216)},line width=1.2pt] plot[domain=4.785288706191459:5.326972569147958,variable=\t]({1*5.513782387193336*cos(\t r)+0*5.513782387193336*sin(\t r)},{0*5.513782387193336*cos(\t r)+1*5.513782387193336*sin(\t r)});\draw [shift={(7.801761172798637,6.2729192667570235)},line width=1.2pt] plot[domain=3.471413503397752:5.962266589276761,variable=\t]({1*0.8438183366159638*cos(\t r)+0*0.8438183366159638*sin(\t r)},{0*0.8438183366159638*cos(\t r)+1*0.8438183366159638*sin(\t r)});

\draw [fill=black] (7.425668292444792,5.00085301330238) circle (3.5pt);
\draw[color=black] (6.5897929090383816,5.560096979379848) node {$\:ba$};
\draw [fill=black,shift={(7.083362211564692,7.506007527524644)},rotate=270] (0,0) ++(0 pt,3pt) -- ++(2.598076211353316pt,-4.5pt)--++(-5.196152422706632pt,0 pt) -- ++(2.598076211353316pt,4.5pt);
\end{tikzpicture}
\end{tabular}
\end{center}

\subsection[Automorphisms $\alpha$ and $\beta$]{Automorphisms $\boldsymbol{\alpha}$ and $\boldsymbol{\beta}$}
The fundamental idea, proposed in~\cite{AS} and~\cite{AS2}, is to mimic the action of the Dehn twists of $\Sigma_{g,n} {\setminus} D$ on $\pi_1(\Sigma_{g,n}{\setminus} D)$ at the level of the algebra $\mathcal{L}_{g,n}(H)$. Let us be more precise. We focus on the case~$(g,n) = (1,0)$.

In $\pi_1(\Sigma_{1,0}{\setminus} D)$ we have the two canonical curves $a$ and $b$, while in $\mathcal{L}_{1,0}(H)$ we have the matri\-ces~$\overset{I}{A}$ and~$\overset{I}{B}$. In view of~\eqref{actionMCG}, let us try to define two morphisms $\widetilde{\tau}_a, \widetilde{\tau}_b \colon \mathcal{L}_{1,0}(H) \to \mathcal{L}_{1,0}(H)$ by the same formulas
\begin{alignat*}{3}
& \widetilde{\tau}_a\big(\overset{I}{A}\big) = \overset{I}{A}, \qquad && \widetilde{\tau}_a\big(\overset{I}{B}\big) = \overset{I}{B}\overset{I}{A}, & \\
& \widetilde{\tau}_b\big(\overset{I}{A}\big) = \overset{I}{B} \!^{-1}\overset{I}{A}, \qquad && \widetilde{\tau}_b\big(\overset{I}{B}\big) = \overset{I}{B}.&
\end{alignat*}
Let us see the behavior of these mappings under the fusion and exchange relations. For the exchange relation, no problem arises
\begin{align*}
R_{12} \widetilde{\tau}_a(B)_1 R_{21} \widetilde{\tau}_a(A)_2 &= R_{12} B_1 A_1 R_{21} A_2 &&\quad \text{(definition)}\\
&= R_{12} B_1 R_{21} A_2 R_{12} A_1 R_{12}^{-1} &&\quad\text{(equation \eqref{reflection})}\\
&= A_2 R_{12} B_1 A_1 R_{12}^{-1} &&\quad \text{(equation\eqref{echangeL10})}\\
&=\widetilde{\tau}_a(A)_2 R_{12} \widetilde{\tau}_a(B)_1 R_{12}^{-1} &&\quad \text{(definition)},
\end{align*}
and a similar computation holds for $\widetilde{\tau}_b$. The fusion relation is almost satisfied
\begin{align*}
\widetilde{\tau}_a(B)_{12}&= B_{12} A_{12} &&\quad \text{(definition)}\\
&= \Delta(v)_{12} B_{12} v_1^{-1} v_2^{-1} R_{21} R_{12} A_{12} &&\quad \text{(trick based on \eqref{ribbon})}\\
&= \Delta(v)_{12} B_1 R_{21} B_2 R_{21}^{-1} v_1^{-1} v_2^{-1} R_{21} R_{12} A_1 R_{21} A_2 R_{21}^{-1} &&\quad \text{(equation~\eqref{fusionL01})}\\
&= \Delta(v)_{12} v_1^{-1} v_2^{-1} B_1 R_{21} B_2 R_{12} A_1 R_{21} A_2 R_{21}^{-1} &&\quad \text{($v$ is central)}\\
&= \Delta(v)_{12} v_1^{-1} v_2^{-1} B_1 A_1 R_{21} B_2 A_2 R_{21}^{-1} &&\quad \text{(equation~\eqref{echangeL10})}\\
&= \Delta(v)_{12} v_1^{-1} v_2^{-1} \widetilde{\tau}_a(B)_1 R_{21} \widetilde{\tau}_a(B)_2 R_{21}^{-1} &&\quad \text{(definition)},
\end{align*}
and we get similarly
\[ \widetilde{\tau}_b(A)_{12} = B^{-1}_{12} A_{12} = \Delta\big(v^{-1}\big)_{12} v_1 v_2 \widetilde{\tau}_b(A)_1 R_{21} \widetilde{\tau}_b(A)_2 R_{21}^{-1}. \]
From this we conclude that the elements $\overset{I}{v} \!^{-1}\overset{I}{B}\overset{I}{A}$ and $\overset{I}{v}\overset{I}{B} \!^{-1}\overset{I}{A}$ satisfy the relation \eqref{fusionL01}. Since $v$ is central, we see that the exchange relation still holds with these elements. We thus have found the morphisms which mimic $\tau_a$ and $\tau_b$. We denote them by $\alpha$ and $\beta$ respectively, and we have the following proposition (the maps $\alpha, \beta$ and the fact that they are automorphisms were already in \cite[Lemma~2]{AS} and \cite[equations~(4.1) and~(4.2)]{AS2}).

\begin{Proposition}\label{autoAlphaBeta}We have two automorphisms $\alpha$, $\beta$ of $\mathcal{L}_{1,0}(H)$ defined by
\begin{alignat*}{3}
& \alpha(\overset{I}{A}) = \overset{I}{A}, \qquad && \alpha(\overset{I}{B}) = \overset{I}{v} \!^{-1}\overset{I}{B}\overset{I}{A},& \\
& \beta(\overset{I}{A}) = \overset{I}{v} \overset{I}{B} \!^{-1}\overset{I}{A}, \qquad && \beta(\overset{I}{B}) = \overset{I}{B}.&
\end{alignat*}

Moreover, these automorphisms are inner: there exist $\widehat{\alpha}, \widehat{\beta} \in \mathcal{L}_{1,0}(H)$ unique up to scalar such that
\[ \forall\, x \in \mathcal{L}_{1,0}(H), \qquad \alpha(x) = \widehat{\alpha}x\widehat{\alpha}^{-1}, \qquad \beta(x) = \widehat{\beta}x\widehat{\beta}^{-1}. \]
\end{Proposition}
\begin{proof}
By Theorem~\ref{isoL10Heisenberg}, $\mathcal{L}_{1,0}(H)$ is a matrix algebra. Hence, by the Skolem--Noether theorem, every automorphism of $\mathcal{L}_{1,0}(H)$ is inner.
\end{proof}

A natural question is then to find explicitly the elements $\widehat{\alpha}$, $\widehat{\beta}$. The answer is amazingly simple (it has been provided in \cite{AS} and \cite{S} for the semisimple case). Recall the notation~\eqref{coinvParticuliers}.

\begin{Theorem}\label{valueHatAlphaBeta}Up to scalar, $ \widehat{\alpha} = v_A^{-1}$ and $\widehat{\beta} = v_B^{-1}$.
\end{Theorem}
\begin{proof}We must show that
\[
v^{-1}_A\overset{I}{A} = \overset{I}{A} v^{-1}_A, \qquad v^{-1}_A \overset{I}{B} = \overset{I}{v} \!^{-1} \overset{I}{B}\overset{I}{A} v^{-1}_A \qquad \text{and} \qquad v^{-1}_B \overset{I}{A} = \overset{I}{v} \overset{I}{B} \!^{-1}\overset{I}{A}v^{-1}_B, \qquad v^{-1}_B \overset{I}{B} = \overset{I}{B} v^{-1}_B.
\]
It is obvious that $v_A^{-1}$ (resp.~$v_B^{-1}$) commutes with the matrices $\overset{I}{A}$ (resp.~$\overset{I}{B}$) since it is central in $j_A(\mathcal{L}_{0,1}(H))$ (resp.\ in $j_B(\mathcal{L}_{0,1}(H))$). Let us show the other commutation relation for~$v_A^{-1}$. We use the isomorphism~$\Psi_{1,0}$. Observe that $\Psi_{1,0}(x_A) = x$ for all $x \in H$. Hence, using the exchange relation of Definition~\ref{defHeisenberg} and~\eqref{ribbon}, we have
\begin{align*}
\Psi_{1,0}\big(v_A^{-1}\overset{I}{B}\big) & = \overset{I}{L} \!^{(+)} v^{-1} \overset{I}{T} \overset{I}{L} \!^{(-)-1} = \overset{I}{L} \!^{(+)} \overset{I}{T}\big(? v'^{-1}\big) v''^{-1} \overset{I}{L} \!^{(-)-1} = \overset{I}{L} \!^{(+)} \overset{I}{T} \overset{I}{(v')}{^{-1}} v''^{-1} \overset{I}{L} \!^{(-)-1}\\
& = \overset{I}{L} \!^{(+)} \overset{I}{T} \overset{I}{v}{^{-1}} \overset{I}{b_i} \overset{I}{a_j} v^{-1} a_i b_j \overset{I}{L} \!^{(-)-1} = \overset{I}{v}{^{-1}} \overset{I}{L} \!^{(+)} \overset{I}{T} \overset{I}{b_i}a_i \overset{I}{a_j}b_j \overset{I}{L} \!^{(-)-1} v^{-1}\\
& = \overset{I}{v}{^{-1}} \overset{I}{L} \!^{(+)} \overset{I}{T} \overset{I}{L} \!^{(-)-1} \overset{I}{L} \!^{(+)} \overset{I}{L} \!^{(-)-1} v^{-1} = \Psi_{1,0}\big(\overset{I}{v} \!^{-1} \overset{I}{B}\overset{I}{A} v^{-1}_A\big)
\end{align*}
as desired. We now apply the morphism $\alpha^{-1} \circ \beta^{-1}$ to the equality $v_A^{-1}\overset{I}{B} = \overset{I}{v} \!^{-1} \overset{I}{B}\overset{I}{A} v_A^{-1}$:
\[ \alpha^{-1} \circ \beta^{-1}\big(v_A^{-1}\overset{I}{B}\big) = \overset{I}{v} v_B^{-1} \overset{I}{B}\overset{I}{A}{^{-1}} = \alpha^{-1} \circ \beta^{-1}\big(\overset{I}{v}^{-1}\overset{I}{B}\overset{I}{A}v_A^{-1}\big) = \overset{I}{B} \overset{I}{A}{^{-1}} \overset{I}{B} v_B^{-1}. \]
Using that $v_B$ and $\overset{I}{B}$ commute, we easily get the desired equality.
\end{proof}

\subsection[Projective representation of $\mathrm{SL}_2(\mathbb{Z})$ on $\SLF(H)$]{Projective representation of $\boldsymbol{\mathrm{SL}_2(\mathbb{Z})}$ on $\boldsymbol{\SLF(H)}$}

We now show that the elements $v_A^{-1}$, $v_B^{-1}$ give rise to a projective representation of $\mathrm{MCG}(\Sigma_{1,0})$ on $\SLF(H)$ via the following assignment
\[ \tau_a \mapsto \rho_{\mathrm{SLF}}\big(v_A^{-1}\big), \qquad \tau_b \mapsto \rho_{\mathrm{SLF}}\big(v_B^{-1}\big). \]
We must then check that
\[ \rho_{\mathrm{SLF}}\big( v_A^{-1} v_B^{-1} v_A^{-1} \big) \sim \rho_{\mathrm{SLF}}\big( v_B^{-1} v_A^{-1} v_B^{-1} \big), \qquad \rho_{\mathrm{SLF}}\big( v_A^{-1} v_B^{-1} \big)^6 \sim 1, \]
where $\sim$ means equality up to scalar. As we will see, it turns out that the braid relation holds in the algebra $\mathcal{L}_{1,0}(H)$ itself (the scalar being $1$), while the relation $\big(v_A^{-1} v_B^{-1}\big)^6 \sim 1$ only holds in the representation $\SLF(H)$. This is because the relation $(\tau_a)_*(\tau_b)_*(\tau_a)_* = (\tau_b)_*(\tau_a)_*(\tau_b)_*$ holds on $\pi_1(\Sigma_{1,0} {\setminus} D)$, while the relation $((\tau_a)_*(\tau_b)_*)^6=1$ only holds on $\pi_1(\Sigma_{1,0})$. As pointed out in Remark~\ref{remarqueC}, the matrices $\overset{I}{C}$ corresponding to the boundary circle vanishes on $\mathrm{SLF}(H)$. Thus applying $\rho_{\mathrm{SLF}}$ amounts to gluing back the disk~$D$. See \cite{F2} for the generalization of this fact to higher genus.

Integrals on $H$ will play an important role. Recall that a {\em left integral} (resp.\ {\em right integral}) is a non-zero linear form $\mu^l$ (resp.~$\mu^r$) on $H$ which satisfies
\begin{gather}\label{integrale}
\forall\, x \in H, \qquad \big(\mathrm{id} \otimes \mu^l\big) \circ \Delta(x) = \mu^l(x)1 \qquad \big(\text{resp.} \ \big(\mu^r \otimes \mathrm{id}\big) \circ \Delta(x) = \mu^r(x)1\big).
\end{gather}
Since $H$ is finite-dimensional, this is equivalent to:
\begin{gather}\label{integrale2}
\forall\, \psi \in \mathcal{O}(H), \qquad \psi\mu^l = \varepsilon(\psi)\mu^l \qquad \big(\text{resp.} \ \mu^r\psi = \varepsilon(\psi)\mu^r\big).
\end{gather}
It is well-known that left and right integrals always exist if~$H$ is finite-dimensional. Moreover, they are unique up to scalar. We fix $\mu^l$. Then $\mu^l \circ S^{-1}$ is a right integral, and we choose
\begin{gather}\label{muR}
\mu^r = \mu^l \circ S^{-1}.
\end{gather}

The following proposition is important for the sequel (points 2 and 3 are well-known and can be deduced from results of Radford, see~\cite{radford}).
\begin{Proposition}\label{propVIntegrale}
Let $\varphi_v, \varphi_{v^{-1}} \in H^*$ defined by
\[ \varphi_v = \mu^l\big(v^{-1}\big)^{-1}\mu^l\big( g^{-1}v^{-1} ? \big), \qquad \varphi_{v^{-1}} = \mu^l(v)^{-1}\mu^l\big(g^{-1}v ?\big). \]
Then
\[ \mathcal{D}(\varphi_v) = v, \qquad \mathcal{D}\big(\varphi_{v^{-1}}\big) = v^{-1}. \]
It follows that
\begin{enumerate}\itemsep=0pt
\item[$1)$] $\varphi_{v}, \varphi_{v^{-1}} \in \SLF(H)$,
\item[$2)$] $\mu^l\big(g^{-1} ?\big), \mu^r (g ? ) \in \SLF(H)$,
\item[$3)$] $ \forall\, x, y \in H$, $\mu^l(xy) = \mu^l\big(yS^{2}(x)\big)$, $\mu^r(xy) = \mu^r\big(S^2(y)x\big)$.
\end{enumerate}
\end{Proposition}
\begin{proof}Consider the following computation, where we use \eqref{ribbon} and \eqref{integrale}:
 \begin{align*}
\mathcal{D}\big(\mu^l\big( g^{-1}v^{-1} ? \big)\big) &= \big\langle \mu^l\big( g^{-1}v^{-1} ? \big) \otimes \mathrm{id}, (g \otimes 1) RR'\big\rangle \\
&= \big\langle \mu^l\big( g^{-1}v^{-1} ? \big) \otimes \mathrm{id}, gv\big(v^{-1}\big)'' \otimes v\big(v^{-1}\big)' \big\rangle\\
& = \mu^l\big(\big(v^{-1}\big)''\big) v\big(v^{-1}\big)' = \mu^l\big(v^{-1}\big)v.
\end{align*}
Since $H$ is factorizable, the map $\mathcal{D}$ is an isomorphism of vector spaces. The left integral $\mu^l$ is non-zero, so $\mu^l\big( g^{-1}v^{-1} ? \big)$ is non-zero either. Since~$\mathcal{D}$ is an isomorphism, it follows that $\mathcal{D}\big(\mu^l\big( g^{-1}v^{-1} ? \big)\big) = \mu^l\big(v^{-1}\big)v \neq 0$, and thus $\mu^l\big(v^{-1}\big) \neq 0$. Hence the formula for $\varphi_v$ is well defined. Moreover, we have the restriction $\mathcal{D}\colon \SLF(H) \overset{\sim}{\rightarrow} \mathcal{Z}(H)$, so since $v \in \mathcal{Z}(H)$, we get that $\varphi_v \in \SLF(H)$. This allows us to deduce the properties stated about~$\mu^l$. Using~\eqref{muR}, we obtain the properties~1),~2) and~3) for~$\mu^r$. We can now proceed with the computation for~$\varphi_{v^{-1}}$:
 \begin{align*}
\mathcal{D}\big( \mu^l\big(g^{-1}v ?\big) \big) &= \big\langle \mu^l\big(g^{-1}v ?\big) \otimes \mathrm{id}, (g \otimes 1) RR'\big\rangle = \mu^l ( va_ib_j ) b_i a_j\\
&= \mu^r ( v S(b_j)S(a_i) ) b_i a_j = \mu^r\big( vS(a_i)S^{-1}(b_j)\big)b_ia_j \\
&= \big\langle \mu^r \otimes \mathrm{id}, (v \otimes 1) (R'R)^{-1} \big\rangle = \mu^r ( v' )v''v^{-1} = \mu^r(v)v^{-1} = \mu^l(v)v^{-1},
\end{align*}
where we used \eqref{muR}, the property~3) previously shown and~\eqref{ribbon}. We conclude as before.
\end{proof}

\noindent Since $\mathcal{D}$ is an isomorphism of algebras, we have $\varphi_{v^{-1}}^{-1} = \varphi_v$, and thus
\begin{gather}\label{phiVcarre}
\varphi_{v^{-1}}\varphi_{v^{-1}}^{v^{-2}} = \frac{\mu^l\big(v^{-1}\big)}{\mu^l(v)}\varepsilon.
\end{gather}
By Proposition~\ref{actionAB}, the actions of $v_A^{-1}$ and $v_B^{-1}$ on $\SLF(H)$ are
\begin{gather}\label{actionsV}
\forall\, \psi \in \SLF(H), \qquad v_A^{-1} \triangleright \psi = \psi^{v^{-1}} = \psi\big(v^{-1}?\big) \qquad \text{and} \qquad v_B^{-1} \triangleright \psi = \big(\varphi_{v^{-1}}\psi^v\big)^{v^{-1}}.
\end{gather}

\begin{Lemma}\label{phiPhiV}
$\varphi_{v^{-1}} \varphi_{v^{-1}}^{v^{-1}} = \varphi_{v^{-1}}^{v^{-1}}$.
\end{Lemma}
\begin{proof}
For $x \in H$
 \begin{align*}
\big\langle \varphi_{v^{-1}} \varphi_{v^{-1}}^{v^{-1}}, x \big\rangle &= \mu^l (v )^{-2} \mu^l\big(vg^{-1}x'\big)\mu^l\big(g^{-1}x''\big) = \mu^l (v )^{-2} \big\langle \mu^l (v? )\mu^l, g^{-1}x \big\rangle\\
&= \mu^l(v)^{-1} \mu^l(g^{-1}x) = \varphi_{v^{-1}}^{v^{-1}}(x).
\end{align*}
We simply used \eqref{integrale2}.
\end{proof}

This lemma has an important consequence.

\begin{Proposition}\label{braidV}The following braid relation holds in $\mathcal{L}_{1,0}(H)$:
\[ v_A^{-1} v_B^{-1} v_A^{-1} = v_B^{-1} v_A^{-1} v_B^{-1}. \]
\end{Proposition}
\begin{proof}The morphisms $\alpha$ and $\beta$ satisfy the braid relation $\alpha\beta\alpha = \beta\alpha\beta$. Hence by Theorem~\ref{valueHatAlphaBeta} and since $\mathcal{Z}(\mathcal{L}_{1,0}(H)) \cong \mathbb{C}$, we have $\lambda v_A^{-1} v_B^{-1} v_A^{-1} = v_B^{-1} v_A^{-1} v_B^{-1}$ for some $\lambda \in \mathbb{C}$. We evaluate on the counit
 \begin{gather*}
\lambda v_A^{-1} v_B^{-1} v_A^{-1} \triangleright \varepsilon = \lambda v_A^{-1} v_B^{-1} \triangleright \varepsilon = \lambda v_A^{-1} \triangleright \varphi_{v^{-1}}^{v^{-1}} = \lambda \varphi_{v^{-1}}^{v^{-2}},\\
v_B^{-1} v_A^{-1} v_B^{-1} \triangleright \varepsilon = v_B^{-1} v_A^{-1} \triangleright \varphi_{v^{-1}}^{v^{-1}} = v_B^{-1} \triangleright \varphi_{v^{-1}}^{v^{-2}} = \big(\varphi_{v^{-1}}\varphi_{v^{-1}}^{v^{-1}}\big)^{v^{-1}} = \big(\varphi_{v^{-1}}^{v^{-1}}\big)^{v^{-1}} = \varphi_{v^{-1}}^{v^{-2}}.
\end{gather*}
We used $\varepsilon\big(v^{-1}?\big) = \varepsilon\big(v^{-1}\big)\varepsilon = \varepsilon$ and Lemma~\ref{phiPhiV}. It follows that $\lambda=1$.
\end{proof}

Observe that $(\alpha\beta)^6 \neq \mathrm{id}$, thus the other relation of $\mathrm{MCG}(\Sigma_{1,0})$ does not hold in $\mathcal{L}_{1,0}(H)$. In order to show it in the representation, we begin with a technical lemma, in which we use the notation of~\eqref{coinvParticuliers}.
\begin{Lemma}\label{lemmeTechniquePourSimplifierOmegaSurLesInvariants} For all $z \in \mathcal{L}_{0,1}^{\mathrm{inv}}(H) = \mathcal{Z}(H)$, we have $z_{v^2A^{-1}B^{-1}A} = z_{B^{-1}}$.
\end{Lemma}
\begin{proof}Write as usual $z = \sum_I \mathrm{tr}\big(\Lambda_I\overset{I}{g}\overset{I}{M}\big)$ (Lemma~\ref{writingInvariants}). We first show that $z_{v^{-1}AB^{-1}} = z_{vB^{-1}A}$ in the Heisenberg double. We have
\[ \Psi_{1,0}\big(z_{v^{-1}AB^{-1}}\big) = \sum_I \mathrm{tr}\big( \Lambda_I \overset{I}{g} \overset{I}{v} \!^{-1} \overset{I}{L} \!^{(+)} \overset{I}{T} \!^{-1} \overset{I}{L} \!^{(+)-1} \big)
= \sum_I \mathrm{tr}\big( \Lambda_I \overset{I}{g} \overset{I}{v} \!^{-1} \overset{I}{a_i} b_i S(\overset{I}{T}) \overset{I}{S(a_j)} b_j \big). \]
Using the defining relation of $\mathcal{H}(\mathcal{O}(H))$ together with~\eqref{deltaR} and~\eqref{elementDrinfeld}, we get
\[ \overset{I}{a_i} b_i S\big(\overset{I}{T}\big)
= \big(\overset{I}{a_i} \overset{I}{a_k} \overset{I}{S(b_k)}\big) S\big(\overset{I}{T}\big) b_i = \big(\overset{I}{a_i} \overset{I}{g}{^{-1}} \overset{I}{v}\big) S\big(\overset{I}{T}\big) b_i. \]
It follows that
\begin{align*}
\Psi_{1,0} (z_{v^{-1}AB^{-1}} ) &= \sum_I \mathrm{tr}\big( \Lambda_I \overset{I}{S^2(a_i)} S(\overset{I}{T}) \overset{I}{S(a_j)} b_i b_j \big) = \mathcal{D}^{-1}(z)\big( S^2(a_i) S(?) S(a_j) \big)b_ib_j\\
&= S\big(\mathcal{D}^{-1}(z)\big) (S(a_i) a_j ? )b_ib_j = S\big(\mathcal{D}^{-1}(z)\big) \in \mathcal{O}(H) \otimes 1.
\end{align*}
A similar computation shows that $\Psi_{1,0} z_{vB^{-1}A} ) = S\big(\mathcal{D}^{-1}(z)\big)$. Hence $z_{v^{-1}AB^{-1}} = z_{vB^{-1}A}$. Applying the morphism $\alpha$ to this equality, we find
\[ \alpha (z_{v^{-1}AB^{-1}} ) = z_{B^{-1}} = \alpha (z_{vB^{-1}A} ) = z_{v^2A^{-1}B^{-1}A} \]
as desired.
\end{proof}

Consider $\widehat{\omega} = (v_A v_B v_A)^{-1} \in \mathcal{L}_{1,0}(H)$, which implements the automorphism $\omega = \alpha\beta\alpha$: $\omega(x) = \widehat{\omega} x \widehat{\omega}^{-1}$. The key observation is the following lemma.

\begin{Lemma}\label{keyLemmeOmega}For all $\psi \in \SLF(H)$:
\[ \widehat{\omega}^2 \triangleright \psi = \frac{\mu^l\big(v^{-1}\big)}{\mu^l(v)}S(\psi). \]
\end{Lemma}
\begin{proof}First, we show the formula for $\psi = \varepsilon$:
\begin{align*} \widehat{\omega}^2 \triangleright \varepsilon & = \big(v_A^{-1} v_B^{-1} v_A^{-1}\big)^2 \triangleright \varepsilon = v_A^{-1} v_B^{-1} v_A^{-1} \triangleright \varphi_{v^{-1}}^{v^{-2}} = \bigl( \varphi_{v^{-1}} \varphi_{v^{-1}}^{v^{-2}} \bigr)^{v^{-2}} \\
& = \frac{\mu^l\big(v^{-1}\big)}{\mu^l(v)}\varepsilon^{v^{-2}} = \frac{\mu^l\big(v^{-1}\big)}{\mu^l(v)}\varepsilon,
\end{align*}
where we used \eqref{actionsV} and \eqref{phiVcarre}. Second, note that $\omega(\overset{I}{A}) = \overset{I}{v}{^2} \overset{I}{A}{^{-1}}\overset{I}{B}{^{-1}}\overset{I}{A}$ and $\omega(\overset{I}{B}) = \overset{I}{A}$. Then, thanks to Lemma~\ref{lemmeTechniquePourSimplifierOmegaSurLesInvariants}, it follows that for every $z \in \mathcal{Z}(H) = \mathcal{L}^{\mathrm{inv}}_{0,1}(H)$,
\[ \omega^2(z_B) = \omega(z_A) = z_{v^2A^{-1}B^{-1}A} = z_{B^{-1}}. \]
Hence $\widehat{\omega}^2 z_B =z_{B^{-1}}\widehat{\omega}^2$.
Finally, observe that by Proposition~\ref{actionAB}
\[ \forall\, \psi \in \SLF(H), \qquad \mathcal{D} \big(\psi^v \big)_B \triangleright \varepsilon = \psi. \]
These three facts together with Lemma~\ref{lemmeBMoinsUn} yield
\begin{align*}
\widehat{\omega}^2 \triangleright \psi &= \widehat{\omega}^2 \mathcal{D}\big(\psi^v\big)_B \triangleright \varepsilon = \mathcal{D}\big(\psi^v\big)_{B^{-1}} \widehat{\omega}^2 \triangleright \varepsilon = \frac{\mu^l\big(v^{-1}\big)}{\mu^l(v)} \mathcal{D}\big(\psi^v\big)_{B^{-1}} \triangleright \varepsilon\\
& = \frac{\mu^l\big(v^{-1}\big)}{\mu^l(v)} S\big(\psi^v\big)^{v^{-1}} = \frac{\mu^l\big(v^{-1}\big)}{\mu^l(v)}S(\psi)
\end{align*}
as desired.
\end{proof}

Recall that $\mathrm{PSL}_2(\mathbb{Z}) = \mathrm{SL}_2(\mathbb{Z})/\{\pm \mathbb{I}_2\}$ admits the following presentations
\[ \mathrm{PSL}_2(\mathbb{Z}) = \big\langle \tau_a, \tau_b \,|\, \tau_a\tau_b\tau_a = \tau_b\tau_a\tau_b, \, (\tau_a\tau_b)^3=1 \big\rangle = \big\langle s, t \,|\, (st)^3=1, \, s^2=1 \big\rangle. \]

\begin{Theorem}\label{mainResult} Recall that we assume that $H$ is a finite-dimensional factorizable ribbon Hopf algebra. The following assignment defines a projective representation of $\mathrm{MCG}(\Sigma_{1,0}) = \mathrm{SL}_2(\mathbb{Z})$ on $\SLF(H)$:
\[ \tau_a \mapsto \rho_{\mathrm{SLF}}\big(v_A^{-1}\big), \qquad \tau_b \mapsto \rho_{\mathrm{SLF}}\big(v_B^{-1}\big). \]
If moreover $S(\psi) = \psi$ for all $\psi \in \SLF(H)$, then this defines actually a projective representation of $\mathrm{PSL}_2(\mathbb{Z})$.
\end{Theorem}

We denote this projective representation by $\theta_{\mathrm{SLF}}$.

\begin{proof}By Proposition~\ref{braidV}, we know that the braid relation is satisfied in $\mathcal{L}_{1,0}(H)$. By Lem\-ma~\ref{keyLemmeOmega}, we have
\[ (v_A^{-1}v_B^{-1})^3 \triangleright \psi = \widehat{\omega}^2 \triangleright \psi = \frac{\mu^l\big(v^{-1}\big)}{\mu^l(v)}S(\psi). \]
If $S_{\vert \mathrm{SLF}(H)} = \mathrm{id}$, then
\[ \rho_{\mathrm{SLF}}\big(v_A^{-1} v_B^{-1}\big)^3 = \frac{\mu^l\big(v^{-1}\big)}{\mu^l(v)}\mathrm{id}. \]
Otherwise,
\begin{align*} \big(v_A^{-1}v_B^{-1}\big)^6 \triangleright \psi &= \frac{\mu^l\big(v^{-1}\big)}{\mu^l(v)}\widehat{\omega}^2 \triangleright S(\psi) = \frac{\mu^l\big(v^{-1}\big)^2}{\mu^l(v)^2}S^2(\psi) \\
& = \frac{\mu^l\big(v^{-1}\big)^2}{\mu^l(v)^2}\psi\big(g ? g^{-1}\big) = \frac{\mu^l\big(v^{-1}\big)^2}{\mu^l(v)^2}\psi. \tag*{\qed} \end{align*}\renewcommand{\qed}{}
\end{proof}

Observe that the quantity $\frac{\mu^l(v^{-1})}{\mu^l(v)}$ does not depend on the choice up to scalar of $\mu^l$.

\subsection{Equivalence with the Lyubashenko--Majid representation}\label{equivalenceAvecLM}
Recall that $H$ is a finite-dimensional factorizable ribbon Hopf algebra. Under this assumption, two operators $\mathcal{S}, \mathcal{T}\colon H \to H$ are defined in~\cite{LM}:
\[ \mathcal{S}(x) = \big(\mathrm{id} \otimes \mu^l\big)\big(R^{-1}(1 \otimes x)R'^{-1}\big), \qquad \mathcal{T}(x) = v^{-1}x. \]
It is shown that they are invertible and satisfy $(\mathcal{S}\mathcal{T})^3 = \lambda\mathcal{S}^2$, $\mathcal{S}^2 = S^{-1}$, with $\lambda \in \mathbb{C} {\setminus} \{0\}$.
We warn the reader that in~\cite{LM}, they consider the {\em inverse} of the ribbon element (see the bottom of the third page of their paper). That is why there is $v^{-1}$ in the formula for~$\mathcal{T}$.

Now we introduce two maps. The first is
\[ \fonc{\boldsymbol{\chi}}{H^*}{H,}{\beta}{(\beta \otimes \mathrm{id})(R'R),} \]
while the second is
\[ \fonc{\boldsymbol{\gamma}}{H}{H^*,}{x}{\mu^r(S(x) ?).} \]
The map $\boldsymbol{\chi}$ is another variant of the map $\Psi$ of Section~\ref{braidedHopfAlgebras} and is called Drinfeld morphism in~\cite{FGST}. The map $\boldsymbol{\gamma}$ is denoted $\widehat{\boldsymbol{\phi}}{^{-1}}$ in~\cite{FGST}. Under our assumptions, the inverse of $\boldsymbol{\gamma}$ exists (this is a result of Radford~\cite{radford}), but we do not use it. Consider the space of left $q$-characters
\[ \mathrm{Ch}^l(H) = \big\{ \beta \in H^* \,|\,\forall\, x,y \in H, \, \beta(xy) = \beta\big(S^2(y)x\big)\big\}. \]
These maps satisfy the following restrictions
\[ \boldsymbol{\chi} \colon \ \mathrm{Ch}^l(H) \longrightarrow \mathcal{Z}(H), \qquad \boldsymbol{\gamma} \colon \ \mathcal{Z}(H) \longrightarrow \mathrm{Ch}^l(H). \]
This is due to the fact that they intertwine the adjoint and the coadjoint actions (for the first the computation is analogous to that of the proof of Proposition~\ref{reshetikhin}, while the second is immediate by Proposition~\ref{propVIntegrale}).

It is not too difficult to show (see, e.g., \cite[Remark IV.1.2]{ibanez}) that
\[ \forall\, z \in \mathcal{Z}(H), \qquad \mathcal{S}(z) = \boldsymbol{\chi} \circ \boldsymbol{\gamma}(z). \]
It follows that $\mathcal{Z}(H)$ is stable under $\mathcal{S}$ and~$\mathcal{T}$. But since $S^2$ is inner, we have $\mathcal{S}^4(z) = S^{-2}(z) = z$ for each $z \in \mathcal{Z}(H)$. Thus there exists a projective representation $\theta_{\mathrm{LM}}$ of $\mathrm{SL}_2(\mathbb{Z})$ on $\mathcal{Z}(H)$, defined by
\[ \theta_{\mathrm{LM}}(s) = \mathcal{S}_{\vert \mathcal{Z}(H)}, \qquad \theta_{\mathrm{LM}}(t) = \mathcal{T}_{\vert \mathcal{Z}(H)}. \]

We will need the following relation between left and right integrals.
\begin{Lemma}\label{lemmeUnimodUnibalance}
Under our assumptions $H$ is unibalanced, which means that $\mu^l = \mu^r\big(g^2 ?\big)$.
\end{Lemma}
\begin{proof}The terminology ``unibalanced'' is picked from \cite{BBG}, where some facts about integrals and cointegrals are recalled. Recall (see, e.g., \cite[Proposition~8.10.10]{EGNO}) that a finite-dimensional factorizable Hopf algebra is unimodular, which means that there exists $c \in H$, called two-sided cointegral, such that $xc = cx = \varepsilon(x)c$ for all $x \in H$. Let $a \in H$ be the comodulus of $\mu^r$: $\psi\mu^r = \psi(a)\mu^r$ for all $\psi \in \mathcal{O}(H)$ (see, e.g., \cite[equation~(4.9)]{BBG}). By a result of Drinfeld (see \cite[Proposition~10.1.14]{Mon}, but be aware that in this book the notations and conventions for $a$ and $g$ are different from those we use), we know that
\[ uS(u)^{-1} = a(\mathfrak{a} \otimes \mathrm{id})(R), \]
where $\mathfrak{a} \in H^*$ is the modulus of the left cointegral $c^l$ of $H$. Here, since $c = c^l$ is two-sided, we have $\mathfrak{a} = \varepsilon$. Thus $g^2 = u^2v^{-2} = uS(u)^{-1} = a$ thanks to~\eqref{ribbon} and~\eqref{pivotCan}. We deduce that
\[ \mu^l = \mu^r \circ S = \mu^r(a?) = \mu^r\big(g^2?\big), \]
where the second equality is \cite[Proposition~4.7]{BBG}.
\end{proof}

The left $q$-characters are just shifted symmetric linear forms. More precisely, we have an isomorphism of algebras:
\[ \fonc{\big(g^{-1}\big)^*}{\SLF(H)}{\mathrm{Ch}^l(H),}{\psi}{\psi\big(g^{-1}?\big).} \]
Let us define shifted versions of $\boldsymbol{\chi}$ and of $\boldsymbol{\gamma}$:
\[ \boldsymbol{\chi}_{g^{-1}} = \boldsymbol{\chi} \circ \left(g^{-1}\right)^* \colon \ \SLF(H) \overset{\cong}{\longrightarrow} \mathcal{Z}(H), \qquad \boldsymbol{\gamma}_g = g^* \circ \boldsymbol{\gamma} \colon \ \mathcal{Z}(H) \overset{\cong}{\longrightarrow} \SLF(H). \]
The equality $\mathcal{S} = \boldsymbol{\chi}_{g^{-1}} \circ \boldsymbol{\gamma}_g$ still holds, but we have now $\SLF(H)$ instead of~$\mathrm{Ch}^l(H)$.

In order to show the equivalence of $\theta_{\mathrm{SLF}}$ and $\theta_{\mathrm{LM}}$, we begin with two technical lemmas.
\begin{Lemma}\label{lemmeTechniquePhiMu}\quad
\begin{enumerate}\itemsep=0pt
\item[$1.$] $\big(v^{-1}\big)' \otimes S^{-1}\big(\big(v^{-1}\big)''\big) = S\big(\big(v^{-1}\big)''\big) \otimes \big(v^{-1}\big)'$.
\item[$2.$] $\forall\, \psi \in \mathcal{O}(H)$, $\forall\, h \in H$, $\mu^r(h ?)\psi = \mu^r(h' ?) \psi\big(S^{-1}(h'')\big)$.
\end{enumerate}
\end{Lemma}
\begin{proof}The first equality follows from $S\big(v^{-1}\big) = v^{-1}$. For the second one
\begin{align*}
\big\langle \mu^r(h ?)\psi, x \big\rangle &= \mu^r(hx')\psi(x'') = \mu^r(h'x')\psi\big(S^{-1}(h''')h''x''\big) = \psi\big(S^{-1}(h''')\mu^r(h'x')h''x''\big)\\
&= \psi\big(S^{-1}(h'') \mu^r ((h'x)' )(h'x)''\big) = \mu^r(h'x) \psi\big(S^{-1}(h'')\big),
\end{align*}
where we simply used the defining property \eqref{integrale} of~$\mu^r$.
\end{proof}

We will employ an immediate consequence of Lemma~\ref{lemmeUnimodUnibalance}, namely (see Proposition~\ref{propVIntegrale})
\begin{gather}\label{nouveauPhiV}
\varphi_v = \mu^l\big(v^{-1}\big)^{-1}\mu^r\big(gv^{-1} ?\big).
\end{gather}

\begin{Lemma}\label{lemmePourEquivalence}
It holds
\[ \rho_{\mathrm{SLF}}\big(v_A^2 v_B\big) = \mu^l\big(v^{-1}\big)^{-1} \boldsymbol{\gamma}_g \circ \boldsymbol{\chi}_{g^{-1}}. \]
\end{Lemma}
\begin{proof}
We compute each side of the equality. On the one hand:
\[ \boldsymbol{\gamma}_g \circ \boldsymbol{\chi}_{g^{-1}}(\psi) = \boldsymbol{\gamma}_g\big((\psi \otimes \mathrm{id})\big(g^{-1}\big(v^{-1}\big)'v \otimes \big(v^{-1}\big)''v\big)\big) = \psi\big(g^{-1}\big(v^{-1}\big)'v\big) \mu^r\big(gS\big(\big(v^{-1}\big)''\big)v ?\big), \]
whereas on the other hand
\begin{align*}
v_A^2 v_B \triangleright \psi &= \bigl(\varphi_v\psi^v\bigr)^v = \mu^l\big(v^{-1}\big)^{-1} \big( \mu^r\big(gv^{-1} ?\big)\psi^v \big)^v\\
&= \mu^l\big(v^{-1}\big)^{-1}\big[\mu^r\big(g\big(v^{-1}\big)' ?\big)\psi\big( vg^{-1}S^{-1}\big(\big(v^{-1}\big)''\big) \big)\big]^v\\
&= \mu^l\big(v^{-1}\big)^{-1}\mu^r\big(gvS\big(\big(v^{-1}\big)''\big) ?\big)\psi\big( vg^{-1}\big(v^{-1}\big)' \big)
\end{align*}
as desired. We used that $v_A \triangleright \psi = \psi^v$, $v_B \triangleright \psi = \big(\varphi_v \psi^v\big)^{v^{-1}}$, \eqref{nouveauPhiV} and Lemma~\ref{lemmeTechniquePhiMu}.
\end{proof}

The link between the two presentations of $\mathrm{SL}_2(\mathbb{Z})$ is $s = \tau_a^{-1}\tau_b^{-1}\tau_a^{-1}$, $t = \tau_a$. Hence we define two operators $\mathcal{S}', \mathcal{T}' \colon \SLF(H) \to \SLF(H)$ by
\[ \mathcal{S}' = \theta_{\mathrm{SLF}}\big(\tau_a^{-1}\tau_b^{-1}\tau_a^{-1}\big) = \rho_{\mathrm{SLF}}(v_Av_Bv_A), \qquad \mathcal{T}' = \theta_{\mathrm{SLF}}(\tau_a) = \rho_{\mathrm{SLF}}\big(v_A^{-1}\big). \]

\begin{Theorem}\label{EquivalenceLMandSLF}Recall that we assume that $H$ is a finite-dimensional factorizable ribbon Hopf algebra. Then the projective representation $\theta_{\mathrm{SLF}}$ of Theorem~{\rm \ref{mainResult}} is equivalent to $\theta_{\mathrm{LM}}$.
\end{Theorem}
\begin{proof}
Consider the following isomorphism of vector spaces
\[ \fonc{f = \rho_{\mathrm{SLF}}\big(v_A^{-1}\big) \circ \boldsymbol{\gamma}_g}{\mathcal{Z}(H)}{\SLF(H),}{z}{\boldsymbol{\gamma}_g(z)^{v^{-1}} = \mu^r\big(gv^{-1}S(z) ?\big).} \]
By Lemma~\ref{lemmePourEquivalence},
\[ \mathcal{S}' = \mu^l\big(v^{-1}\big)^{-1}\rho_{\mathrm{SLF}}\big(v_A^{-1}\big) \circ \boldsymbol{\gamma}_g \circ \boldsymbol{\chi}_{g^{-1}} \circ \rho_{\mathrm{SLF}}(v_A). \]
Thus
\[ f \circ \mathcal{S} = \rho_{\mathrm{SLF}}\big(v_A^{-1}\big) \circ \boldsymbol{\gamma}_g \circ \boldsymbol{\chi}_{g^{-1}} \circ \boldsymbol{\gamma}_g = \mu^l\big(v^{-1}\big)\mathcal{S}' \circ f. \]
Next,
\[ f \circ \mathcal{T}(z) = f\big(v^{-1}z\big) = \boldsymbol{\gamma}_g(z)^{v^{-2}} = \rho_{\mathrm{SLF}}(v_A^{-1})\big(\boldsymbol{\gamma}_g(z)^{v^{-1}}\big) = \mathcal{T}' \circ f(z). \]
Then $f$ is an intertwiner of projective representations.
\end{proof}

\section[The example of $H = \overline{U}_q(\mathfrak{sl}(2))$]{The example of $\boldsymbol{H = \overline{U}_q(\mathfrak{sl}(2))}$}\label{exempleUq}

Let $q$ be a primitive root of unity of order $2p$, with $p \geq 2$. We now work in some detail the case of $H = \overline{U}_q(\mathfrak{sl}(2))$, the restricted quantum group associated to $\mathfrak{sl}(2, \mathbb{C})$, which will be denoted~$\overline{U}_q$ in the sequel. For the definitions and main properties about~$\overline{U}_q$, $\mathcal{Z}\big(\overline{U}_q\big)$, $\SLF\big(\overline{U}_q\big)$, their canonical bases and the $\overline{U}_q$-modules, see \cite{arike,F, FGST, GT, KS, suter}. Here we take back the notations and conventions of~\cite{F}.

To explicitly describe the representation of $\mathrm{SL}_2(\mathbb{Z})$, we use the GTA basis of $\SLF\big(\overline{U}_q\big)$ which is studied in detail in \cite{F}, and which has been introduced in~\cite{GT} and~\cite{arike}.

Even though $\overline{U}_q$ is not braided, we consider this example because this Hopf algebra is well-studied and is related to certain models in logarithmic conformal field theory~\cite{FGST}. As we shall recall below, it is very close from being braided (it suffices to add a square root of the generator $K$) and its important properties for our purposes are that it is factorizable (by abuse of terminology since it is not braided) and that it contains a ribbon element. It will follow that the defining relations of $\mathcal{L}_{0,1}(H)$ and $\mathcal{L}_{1,0}(H)$ still make sense in this context and that all the previous results remain true.

\subsection[The braided extension of $\overline{U}_q$]{The braided extension of $\boldsymbol{\overline{U}_q}$}\label{braidedExtension}
Recall that $\overline{U}_q$ is not braided \cite[Proposition~3.7.3]{KS}. But its extension by a square root of $K$ is braided, as shown in \cite{FGST}. Let $\overline{U}_q{^{\!1/2}}$ be this extension and $R \in \overline{U}_q{^{\!1/2}} \otimes \overline{U}_q{^{\!1/2}}$ be the universal $R$-matrix, given by
\[ R = q^{H \otimes H/2} \sum_{m=0}^{p-1}\frac{\hat q^m}{[m]!}q^{m(m-1)/2} E^m \otimes F^m,
\]
with
\[ q^{H \otimes H/2} = \frac{1}{4p}\sum_{n,j=0}^{4p-1}q^{-nj/2} K^{n/2} \otimes K^{j/2}, \]
where $q^{1/2}$ is a fixed square root of $q$. We use the notation $q^{H \otimes H/2}$ because $q^{H \otimes H/2} v \otimes w = q^{ab/2}$ if $K^{1/2}v = q^{a/2} v$ and $K^{1/2}w = q^{b/2}w$; also recall the notation $\hat q = q - q^{-1}$. There is a ribbon element $v$ associated to this $R$-matrix (we choose $g = K^{p+1}$ as pivotal element, and by \eqref{pivotCan} this fixes the choice of $v$). The formulas for $RR'$ and $v$ are in~\cite{FGST}; it is important to note that~$K^{1/2}$ does not appear in the expression of these elements
\[ RR' \in \overline{U}_q \otimes \overline{U}_q \qquad \text{and} \qquad v \in \overline{U}_q. \]
Moreover, thanks to the formula for $RR'$, it is obvious that the map
\[ \fonc{\Psi}{\overline{U}_q^*}{\overline{U}_q,}{\beta}{(\beta \otimes \mathrm{id})(RR')} \]
is an isomorphism of vector spaces. Hence, even if $\overline{U}_q$ is not braided, it can be thought of as a~factorizable Hopf algebra.

Let $I$ be a $\overline{U}_q{^{\!1/2}}$-module. Since $\overline{U}_q \subset \overline{U}_q{^{\!1/2}}$, $I$ determines a $\overline{U}_q$-module, which we denote $I_{\vert \overline{U}_q}$. We say that a $\overline{U}_q$-module $J$ is {\em liftable} if there exists a $\overline{U}_q{^{\!1/2}}$-module $\widetilde{J}$ such that $\widetilde{J}_{\vert \overline{U}_q} = J$. Not every $\overline{U}_q$-module is liftable, see~\cite{KS}. But the simple modules and the PIMs are liftable, which is enough for us. Indeed, it suffices to define the action of $K^{1/2}$ on these modules. Take back the notations of~\cite{F} for the canonical basis of modules. For the simple module $\mathcal{X}^{\epsilon}(s)$ ($\epsilon \in \{\pm\})$, there are two choices for $\epsilon^{1/2}$, and so the two possible liftings are defined by
\[ K^{1/2}v_j = \epsilon^{1/2}q^{(s-1-2j)/2}v_j \]
and the action of $E$ and $F$ is unchanged. The two possible liftings of the PIM $\mathcal{P}^{\epsilon}(s)$ are defined by
\begin{alignat*}{3}
& K^{1/2}b_0 = \epsilon^{1/2}q^{(s-1)/2}b_0, \qquad & & K^{1/2}x_0 = \big(\epsilon^{1/2}q^{p/2}\big)q^{(p-s-1)/2}x_0, & \\
& K^{1/2}y_0 = \big({-}\epsilon^{1/2}q^{p/2}\big)q^{(p-s-1)/2}y_0, \qquad && K^{1/2}a_0 = \epsilon^{1/2}q^{(s-1)/2}a_0&
\end{alignat*}
and the action of $E$ and $F$ is unchanged.

Let $\widetilde{\mathbb{C}}^-$ be the $1$-dimensional $\overline{U}_q{^{\!1/2}}$-module with basis $v$ defined by $Ev = Fv = 0$, $K^{1/2}v=-v$ (which is a lifting of $\mathcal{X}^+(1) = \mathbb{C}$). If $\widetilde{I}$ is a lifting of a simple module or a PIM~$I$, then we have seen that the only possible liftings of $I$ are $\widetilde{I}^+ = \widetilde{I}$ and $\widetilde{I}^- = \widetilde{I} \otimes \widetilde{\mathbb{C}}^-$. Moreover, using~\eqref{deltaR}, we get equalities which will be used in the next section
\begin{alignat}{3}
& \bigl(\overset{\widetilde{I}^- \widetilde{J}}{R} \bigr)_{\! 12} = \bigl(\overset{\widetilde{I}^+ \widetilde{J}}{R}\bigr)_{\! 12} \bigl(\overset{\widetilde{J}}{K^p}\bigr)_2, \qquad && \bigl(\overset{\widetilde{I} \widetilde{J}^-}{\!R\:}\bigr)_{12} = \bigl(\overset{\widetilde{I} \widetilde{J}^+}{\!R\:}\bigr)_{\! 12} \bigl(\overset{\widetilde{I}}{K^p}\bigr)_1, &\nonumber\\
& \bigl(\overset{\widetilde{I}^- \widetilde{J}}{R'}\bigr)_{\! 12} = \bigl(\overset{\widetilde{I}^+ \widetilde{J}}{R'}\bigr)_{\! 12} \bigl(\overset{\widetilde{J}}{K^p}\bigr)_2, \qquad && \bigl(\overset{\widetilde{I} \widetilde{J}^-}{R'}\bigr)_{\! 12} = \bigl(\overset{\widetilde{I}\widetilde{J}^+}{R'}\bigr)_{\! 12} \bigl(\overset{\widetilde{I}}{K^p}\bigr)_1.& \label{LiftR}
\end{alignat}

\subsection[$\mathcal{L}_{0,1}\big(\overline{U}_q\big)$ and $\mathcal{L}_{1,0}\big(\overline{U}_q\big)$]{$\boldsymbol{\mathcal{L}_{0,1}\big(\overline{U}_q\big)}$ and $\boldsymbol{\mathcal{L}_{1,0}\big(\overline{U}_q\big)}$}
We define $\mathcal{L}_{0,1}\big(\overline{U}_q\big)$ as the quotient of $\mathrm{T}(\overline{U}_q{^{\! *}})$ by the fusion relation
\[ \overset{I \otimes J}{M}\!_{12} = \overset{I}{M}_1\overset{\widetilde{I}\widetilde{J}}{(R')}_{12}\overset{J}{M}_2\overset{\widetilde{I}\widetilde{J}}{\big(R'^{-1}\big)}_{12}, \]
where $I$, $J$ are simple modules or PIMs and $\widetilde{I}$, $\widetilde{J}$ are liftings of $I$ and $J$. From \eqref{LiftR} and the fact that $K^p$ is central, we see that this does not depend on the choice of $\widetilde{I}$ and $\widetilde{J}$. As we saw in Section~\ref{dualAlg}, the matrix coefficients of the PIMs linearly span $\mathcal{L}_{0,1}(H)$, thus we can restrict to them in the definition. However, the simple modules are added for convenience. All the results of Section~\ref{sectionLoopL01} remain true for $\mathcal{L}_{0,1}\big(\overline{U}_q\big)$. In particular, $\Psi_{0,1}$ is an isomorphism since~$\overline{U}_q$ is factorizable (in the generalized sense explained in Section~\ref{braidedExtension}).

We now describe $\mathcal{L}_{0,1}\big(\overline{U}_q\big)$ by generators and relations. Let
\[ M = \overset{\mathcal{X}^+(2)}{M} =
\begin{pmatrix}
a&b\\
c&d
\end{pmatrix}
\qquad \text{and} \qquad \widetilde{R} = \overset{\widetilde{\mathcal{X}}^+(2)\widetilde{\mathcal{X}}^+(2)}{R} = q^{-1/2}
\begin{pmatrix}
q & 0 & 0 & 0\\
0 & 1 & \hat q & 0\\
0 & 0 & 1 & 0\\
0 & 0 & 0 & q
\end{pmatrix}, \]
where $\widetilde{\mathcal{X}}^+(2)$ is the lifting of $\mathcal{X}^+(2)$ defined by $K^{1/2}v_0 = q^{1/2}v_0$. By the decomposition rules of tensor products (see \cite{suter} and also \cite{KS}), every PIM (and every simple module) is a direct summand of some tensor power $\mathcal{X}^+(2)^{\otimes n}$. Thus every matrix coefficient of a PIM is a matrix coefficient of some $\mathcal{X}^+(2)^{\otimes n}$ (with $n \geq p$). It follows from the fusion relation that $a$, $b$, $c$,~$d$ generate $\mathcal{L}_{0,1}\big(\overline{U}_q\big)$. Let us determine relations between these elements. First, we have the reflection equation:
\[ \widetilde{R}_{12} M_1 \widetilde{R}_{21} M_2 = M_2 \widetilde{R}_{12} M_1 \widetilde{R}_{21}. \]
This equation is equivalent to the following exchange relations
 \begin{gather*}
da=ad, \qquad db=q^2bd, \qquad dc=q^{-2}cd, \\
ba = ab+q^{-1}\hat qbd, \qquad cb = bc+q^{-1}\hat q\big(da-d^2\big), \qquad ca = ac-q^{-1}\hat q dc.
\end{gather*}
Second, since $\mathcal{X}^+(2)^{\otimes 2} \cong \mathcal{X}^+(1) \oplus \mathcal{X}^+(3)$,\footnote{This decomposition does not hold if $p=2$: in that case, we have $\mathcal{X}^+(2)^{\otimes 2} \cong \mathcal{P}^+(1)$. The morphism $\Phi$ remains valid and corresponds to sending $\mathbb{C} = \mathcal{X}^+(1)$ in $\mathrm{Soc}\bigl(\mathcal{P}^+(1)\bigr)$.} there exists a unique (up to scalar) morphism $\Phi\colon \mathcal{X}^+(1) \to \mathcal{X}^+(2)^{\otimes 2}$. It is easily computed
\[ \Phi(1) = q v_0 \otimes v_1 - v_1 \otimes v_0. \]
By fusion, we have
\[ \overset{\mathcal{X}^+(2)^{\otimes 2}}{M}\!\!\!\!\!\!_{12}\Phi = M_1\widetilde{R}_{21} M_2 \widetilde{R}_{21}^{-1}\Phi = \Phi. \]
This gives just one new relation, which is the analogue of the quantum determinant
\[ ad-q^2bc=1. \]
Finally, let us compute the RSD isomorphism on $M$
 \begin{align*}
 \Psi_{0,1}
\begin{pmatrix}
a & b\\
c & d
\end{pmatrix}
&= \overset{\widetilde{\mathcal{X}}^+(2)}{L^{(+)}}\overset{\widetilde{\mathcal{X}}^+(2)}{L^{(-)-1}} =
\begin{pmatrix}
K^{1/2} & \hat q K^{1/2}F\\
0 & K^{-1/2}
\end{pmatrix}
\begin{pmatrix}
K^{1/2} & 0\\
\hat q K^{-1/2}E & K^{-1/2}
\end{pmatrix}\\
&=
\begin{pmatrix}
K+q^{-1}\hat q^2FE & q^{-1}\hat q F\\
\hat q K^{-1}E & K^{-1}
\end{pmatrix}.
\end{align*}
We deduce the relations $b^p = c^p = 0$ and $d^{2p}=1$ from the defining relations of $\overline{U}_q$.

\begin{Theorem}$\mathcal{L}_{0,1}\big(\overline{U}_q\big)$ admits the following presentation
\[ \left\langle
a,b,c,d\,
\Bigg|\,
\begin{matrix}
da=ad, & db=q^2bd, & dc=q^{-2}cd \\
ba = ab+q^{-1}\hat qbd, & cb = bc+q^{-1}\hat q\big(da-d^2\big), & ca = ac-q^{-1}\hat q dc\\
ad-q^2bc = 1, & b^p=c^p=0, & d^{2p}=1
\end{matrix}
\right\rangle. \]
A basis is given by the monomials $b^ic^jd^k$ with $0 \leq i,j \leq p-1$, $0 \leq k \leq 2p-1$.
\end{Theorem}
\begin{proof}Let $A$ be the algebra defined by this presentation. It is readily seen that $a = d^{-1} + q^2bcd^{-1}$ and that the monomials $b^ic^jd^k$ with $0 \leq i,j \leq p-1$, $0 \leq k \leq 2p-1$ linearly span~$A$. Thus $\dim(A) \leq 2p^3$. But we know that $2p^3 = \dim\big(\overline{U}_q\big) = \dim \big(\mathcal{L}_{0,1}\big(\overline{U}_q\big) \big)$ since the monomials $E^iF^jK^{\ell}$ with $0 \leq i,j \leq p-1$, $0 \leq k \leq 2p-1$ form the PBW basis of $\overline{U}_q$. It follows that $\dim(A) \leq \dim \big(\mathcal{L}_{0,1}\big(\overline{U}_q\big) \big)$. Since these relations are satisfied in $\mathcal{L}_{0,1}\big(\overline{U}_q\big)$, there exists a surjection $p\colon A \to \mathcal{L}_{0,1}\big(\overline{U}_q\big)$. Thus $\dim(A) \geq \dim \big(\mathcal{L}_{0,1}\big(\overline{U}_q\big) \big)$, and the theorem is proved.
\end{proof}

\begin{Remark}A consequence of this theorem is that $\mathcal{L}_{0,1}\big(\overline{U}_q\big)$ is a restricted version (i.e., a~finite-dimensional quotient by monomial central elements) of $\mathcal{L}_{0,1}(U_q)^{\mathrm{spe}}$, the specialization at our root of unity $q$ of the algebra~$\mathcal{L}_{0,1}(U_q)$. A complete study of the algebra $\mathcal{L}_{0,1}(U_q)^{\mathrm{spe}}$ will appear in~\cite{BaR}.
\end{Remark}

Applying the isomorphism of algebras $\mathcal{D}$ defined in \eqref{morphismeDrinfeld} to the GTA basis of $\SLF\big(\overline{U}_q\big)$, we get a basis of $\mathcal{Z}\big(\overline{U}_q\big)$. We introduce notations for these basis elements
\begin{gather}\label{baseCentreGTA}
\overset{\mathcal{X}^{\epsilon}(s)}{W} = \mathcal{D} ( \chi^{\epsilon}_s ), \qquad H^{s'} = \mathcal{D} ( G_{s'} )
\end{gather}
with $1 \leq s \leq p$, $\epsilon \in \{\pm\}$ and $1 \leq s' \leq p-1$. They satisfy the same multiplication rules than the elements of the GTA basis, see \cite[Section~5]{F} or~\cite{GT} (the elements $\boldsymbol{\chi}(s)$ defined in~\cite{GT} correspond to $[s]H^s$ here). Let us mention that under the identification $\mathcal{L}_{0,1}\big(\overline{U}_q\big) = \overline{U}_q$ \textit{via}~$\Psi_{0,1}$, it holds by definition $\overset{\mathcal{X}^{\epsilon}(s)}{W} = \mathrm{tr}(\overset{\mathcal{X}^{\epsilon}(s)}{K^{p+1}}\overset{\mathcal{X}^{\epsilon}(s)}{M})$, since we choose $K^{p+1}$ as pivotal element. In particular,
\begin{gather}\label{WCasimir}
\overset{\mathcal{X}^+(2)}{W} = -qa - q^{-1}d = -\hat q^2FE - qK - q^{-1}K^{-1} = -\hat q^2 C,
\end{gather}
where $C$ is the standard Casimir element of $\overline{U}_q$.

Similarly, we define $\mathcal{L}_{1,0}\big(\overline{U}_q\big)$ as the quotient of $\mathcal{L}_{0,1}\big(\overline{U}_q\big) \ast \mathcal{L}_{0,1}\big(\overline{U}_q\big)$ by the exchange relations
\begin{gather*} \overset{\widetilde{I}\widetilde{J}}{R}_{12}\overset{I}{B}_1\overset{\widetilde{I}\widetilde{J}}{(R')}_{12}\overset{J}{A}_2 = \overset{J}{A}_2\overset{\widetilde{I}\widetilde{J}}{R}_{12}\overset{I}{B}_1\overset{\widetilde{I}\widetilde{J}}{(R^{-1})}_{12}, \end{gather*}
where $I$, $J$ are simple modules or PIMs and $\widetilde{I}$, $\widetilde{J}$ are liftings of $I$ and $J$. From~\eqref{LiftR}, we see again that this does not depend on the choice of $\widetilde{I}$ and $\widetilde{J}$. The coefficients of $\overset{\mathcal{X}^+(2)}{A}$ and of $\overset{\mathcal{X}^+(2)}{B}$ generate $\mathcal{L}_{1,0}\big(\overline{U}_q\big)$. The morphism $\Psi_{1,0} \colon \mathcal{L}_{1,0}\big(\overline{U}_q\big) \to \mathcal{H}\big(\mathcal{O}\big(\overline{U}_q\big)\big)$ is well-defined (the square root of $K$ does not appear). Indeed, the matrix $\Psi_{1,0}\bigl( \overset{\mathcal{X}^+(2)}{A} \bigr)$ belongs to $1 \otimes \overline{U}_q \subset \mathcal{H}\big(\mathcal{O}\big(\overline{U}_q\big)\big)$ and is the same as the image of the matrix $M$ under the morphism $\Psi_{0,1}$ above; moreover, thanks to the commutation relations of the Heisenberg double we get
\begin{gather*}
 \Psi_{1,0}\bigl( \overset{\mathcal{X}^+(2)}{B} \bigr) = \overset{\widetilde{\mathcal{X}}^+(2)}{L^{(+)}} \overset{\widetilde{\mathcal{X}}^+(2)}{T} \overset{\widetilde{\mathcal{X}}^+(2)}{L^{(-)-1}} =
\begin{pmatrix}
K^{1/2} & \hat q K^{1/2}F\\
0 & K^{-1/2}
\end{pmatrix}
\begin{pmatrix}
\alpha & \beta\\
\gamma & \delta
\end{pmatrix}
\begin{pmatrix}
K^{1/2} & 0\\
\hat q K^{-1/2}E & K^{-1/2}
\end{pmatrix}\\
=
\begin{pmatrix}
q^{1/2}\alpha K + q^{-1/2}\hat q \delta K + q^{1/2}\hat q \gamma KF + q^{-1/2}\hat q \beta E + q^{-1/2} \hat q^2\delta FE & q^{-1/2}\beta + q^{-1/2}\hat q \delta F\\
q^{-1/2}\gamma + q^{1/2}\hat q \delta K^{-1}E & q^{1/2}\delta K^{-1}
\end{pmatrix}.
\end{gather*}
From these formulas we see that $\Psi_{1,0}$ is surjective: indeed, thanks to the coefficients of $\Psi_{1,0}\bigl(\! \overset{\mathcal{X}^+(2)}{A} \!\bigr)$ we obtain $E$, $F$, $K$ and then thanks to the coefficients of $\Psi_{1,0}\bigl( \overset{\mathcal{X}^+(2)}{B} \bigr)$ we obtain $\alpha$, $\beta$, $\gamma$, $\delta$. Moreover, by the exchange relation in $\mathcal{L}_{1,0}\big(\overline{U}_q\big)$ and the fact that $\mathcal{L}_{0,1}\big(\overline{U}_q\big) \cong \overline{U}_q$, we know that $\dim\bigl( \mathcal{L}_{1,0}\big(\overline{U}_q\big) \bigr) \leq \dim\big(\overline{U}_q\big)^2= \dim\bigl( \mathcal{H}\big(\mathcal{O}\big(\overline{U}_q\big)\big) \bigr)$, and hence $\Psi_{1,0}$ is an isomorphism of algebras. In order to get a presentation of $\mathcal{L}_{1,0}\big(\overline{U}_q\big)$, one can again restrict to $I = J = \mathcal{X}^+(2)$ and write down the corresponding exchange relations. We do not give this presentation of $\mathcal{L}_{1,0}\big(\overline{U}_q\big)$ since we will not use it in this work.

With these definitions of $\mathcal{L}_{0,1}\big(\overline{U}_q\big)$ and $\mathcal{L}_{1,0}\big(\overline{U}_q\big)$, all the computations made under the general assumptions before remain valid with $\bar U_q$ and we have a projective $\mathrm{SL}_2(\mathbb{Z})$-representation on $\mathrm{SLF}\big(\overline{U}_q\big)$.

\subsection[Explicit description of the $\mathrm{SL}_2(\mathbb{Z})$-projective representation]{Explicit description of the $\boldsymbol{\mathrm{SL}_2(\mathbb{Z})}$-projective representation}\label{SL2ZUq}

Note that it can be shown directly that $\overline{U}_q$ is unimodular and unibalanced, see for instance \cite[Corollary~II.2.8]{ibanez} (also note that in~\cite{BBG} it is shown that all the simply laced restricted quantum groups at roots of unity are unibalanced).

\begin{Proposition}\label{antipodeSLF}
For all $z \in \mathcal{Z}\big(\overline{U}_q\big)$, $S(z) =z$ and for all $\psi \in \SLF\big(\overline{U}_q\big)$, $S(\psi) = \psi$. It follows that in the case of $\overline{U}_q$, $\theta_{\SLF}$ is in fact a projective representation of $\mathrm{PSL}_2(\mathbb{Z})$.
\end{Proposition}
\begin{proof}By \cite[Appendix~D]{FGST}, the canonical central elements are expressed as $e_s = P_s(C)$, $w^{\pm}_s = \pi^{\pm}_sQ_s(C)$ where $P_s$ and $Q_s$ are polynomials, $C$ is the Casimir element and $\pi_s^{\pm}$ are discrete Fourier transforms of $\big(K^j\big)_{0 \leq j \leq 2p-1}$. It is easy to check that $S(C) = C$ and that $S\big(\pi^{\pm}_s\big) = \pi^{\pm}_s$, thus $S(e_s) = e_s$ and $S\big(w^{\pm}_s\big) = w^{\pm}_s$. Next, let $\psi \in \SLF\big(\overline{U}_q\big)$. Since $\boldsymbol{\gamma}_g$ is an isomorphism, we can write $\psi = \boldsymbol{\gamma}_g(z) = \mu^r(gS(z) ?)$ with $z \in \mathcal{Z}\big(\overline{U}_q\big)$. Then{\samepage
\[ S(\psi) = S\big(\mu^r(gz ?)\big) = \mu^r \circ S\big(?zg^{-1}\big) = \mu^l\big(g^{-1}z?\big) = \mu^r(gz?) = \psi. \]
We used that $S(z)=z$, Proposition~\ref{propVIntegrale} and the fact that $\overline{U}_q$ is unibalanced.}
\end{proof}

Recall that the action of the elements $v_A^{-1}$, $v_B^{-1}$ on $\mathrm{SLF}(H)$ implements the $\mathrm{SL}_2(\mathbb{Z})$-repre\-sen\-ta\-tion. We now determine this action on the GTA basis. Some preliminaries are in order; first recall (see~\cite{FGST}) that the expression of the ribbon element $v$ in the canonical basis $\big(e_s, w_s^{\pm}\big)$ of~$\mathcal{Z}\big(\overline{U}_q\big)$ is
\begin{gather}\label{rubanCentre}
v = \sum_{s=0}^p v_{\mathcal{X}^+(s)} e_s + \hat q \sum_{s=1}^{p-1} v_{\mathcal{X}^+(s)}\left( \frac{p-s}{[s]}w^+_s - \frac{s}{[s]}w^-_s \right)
\end{gather}
with $\hat q = q-q^{-1}$ and
\[ v_{\mathcal{X}^+(s)} = v_{\mathcal{X}^-(p-s)} = (-1)^{s-1}q^{\frac{-(s^2-1)}{2}}. \]
Note that $\overset{\mathcal{X}^{\epsilon}(s)}{v} = v_{\mathcal{X}^{\epsilon}(s)}\mathrm{id}$ and $v_{\mathcal{X}^+(0)}$ is just a notation for $v_{\mathcal{X}^-(p)}$ used to unify the formula. Expressing $v^{-1}$ in this basis is obvious. Second, it is easy to see that the action of $\mathcal{Z}\big(\overline{U}_q\big)$ on $\SLF\big(\overline{U}_q\big)$ is
\begin{alignat}{5}
&\big(\chi^+_s\big)^{e_t} = \delta_{s,t}\chi^+_s, \qquad && \big(\chi^-_s\big)^{e_t} = \delta_{p-s,t}\chi^-_s, \qquad && G_{s}^{ e_t} = \delta_{s,t}G_{s} &&& \nonumber\\
& \big(\chi^+_s\big)^{w^{\pm}_{t}} = 0, \qquad && \big(\chi^-_s\big)^{w^{\pm}_{t}} = 0, \qquad && G_{s}^{ w^+_{t}} = \delta_{s,t}\chi^+_{s}, \qquad && G_{s}^{ w^-_{t}} = \delta_{s,t}\chi^-_{p-s},&\label{actionCentreSLF}
\end{alignat}
where $\varphi^z = \varphi(z?)$. Finally, we have the following lemma.
\begin{Lemma}\label{actionBmoinsUnA}Let $z \in \mathcal{L}_{0,1}^{\mathrm{inv}}(H) = \mathcal{Z}(H)$ and let $\psi \in \SLF(H)$. Then
\[ z_{vB^{-1}A} \triangleright \psi = S\big(\mathcal{D}^{-1}(z)\big) \psi. \]
\end{Lemma}
\begin{proof}The proof is analogous to those of the two similar results in Section~\ref{sectionRepInvariants} and is thus left to the reader. Note that this lemma is not specific to the case of $\overline{U}_q$.
\end{proof}

\begin{Theorem}\label{actionSL2ZArike}The actions of $v_A^{-1}$ and $v_B^{-1}$ on the GTA basis are given by
\[ v_A^{-1} \triangleright \chi^{\epsilon}_s = v^{-1}_{\mathcal{X}^{\epsilon}(s)}\chi^{\epsilon}_s, \qquad v_A^{-1} \triangleright G_{s'} = v^{-1}_{\mathcal{X}^+(s')}G_{s'} - v_{\mathcal{X}^+(s')}^{-1}\hat q\left( \frac{p-s'}{[s']}\chi^+_{s'} - \frac{s'}{[s']}\chi^-_{p-s'} \right) \]
and
\begin{gather*}
v_B^{-1} \triangleright \chi^{\epsilon}_s = \xi \epsilon(-\epsilon)^{p-1}s q^{-(s^2-1)} \left(\sum_{\ell=1}^{p-1}(-1)^s(-\epsilon)^{p-\ell}\big(q^{\ell s} + q^{-\ell s}\big)\big(\chi^+_{\ell} + \chi^-_{p-\ell}\big) + \chi^+_p \right.\\
 \left.\hphantom{v_B^{-1} \triangleright \chi^{\epsilon}_s =}{} + (-\epsilon)^p(-1)^s\chi^-_p \vphantom{\sum_{\ell=1}^{p-1}}\right) +\xi\epsilon (-1)^sq^{-(s^2-1)}\sum_{j=1}^{p-1} (-\epsilon)^{j+1}[j][js]G_j,\\
v_B^{-1} \triangleright G_{s'} = \xi (-1)^{s'}q^{-(s'^2-1)}\frac{\hat q p}{[s']}\sum_{j=1}^{p-1}(-1)^{j+1}[j][js']\left(2G_j - \hat q \frac{p-j}{[j]}\chi^+_j + \hat q \frac{j}{[j]}\chi^-_{p-j}\right),
\end{gather*}
with $\epsilon \in \{\pm\}$, $0 \leq s \leq p$, $1 \leq s' \leq p-1$ and $\xi^{-1} = \frac{1-i}{2\sqrt{p}} \frac{\hat q^{p-1}}{[p-1]!} (-1)^p q^{-(p-3)/2}$.
\end{Theorem}
\begin{proof}The formulas for $v_A^{-1}$ are easily deduced from Proposition~\ref{actionAB}, \eqref{rubanCentre} and \eqref{actionCentreSLF}. Computing the action of $v_B^{-1}$ is more difficult. We will use the commutation relations of $v_B^{-1}$ with the $A,B$-matrices (which follow from Theorem~\ref{valueHatAlphaBeta}), namely
\begin{gather}\label{commBeta}
v_B^{-1} \overset{I}{A} = \overset{I}{v} \overset{I}{B}{^{-1}}\overset{I}{A}v_B^{-1}, \qquad v_B^{-1}\overset{I}{B} = \overset{I}{B}v_B^{-1}
\end{gather}
to compute the action of $v_B^{-1}$ by induction. The multiplication rules of the GTA basis (see \cite[Section 5]{F}) will be used several times. Let us denote
\[ v_B^{-1} \triangleright \chi^{\epsilon}_s = \sum_{\sigma \in \{\pm\}}\sum_{\ell=1}^p \lambda^{\sigma}_{\ell}(\epsilon, s) \chi^{\sigma}_{\ell} + \sum_{j=1}^{p-1}\delta_j(\epsilon, s)G_j. \]
Relation \eqref{commBeta} implies $ v_B^{-1} \overset{\mathcal{X}^+(2)}{W}\!\!\!\!\!_A = \overset{\mathcal{X}^+(2)}{W}\!\!\!\!_{vB{^{-1}}A} v_B^{-1}$ (recall~\eqref{baseCentreGTA}). On the one hand, we obtain by~\eqref{WCasimir}
\begin{align*} v_B^{-1} \overset{\mathcal{X}^+(2)}{W}\!\!\!\!\!_A \triangleright \chi^{\epsilon}_s & = v_B^{-1} \triangleright \chi^{\epsilon}_s\big({-}\hat q^2C ?\big)\\
& = \sum_{\substack{\ell=1\\ \sigma \in \{\pm\}}}^p -\epsilon \big(q^s+q^{-s}\big) \lambda^{\sigma}_{\ell}(\epsilon, s) \chi^{\sigma}_{\ell} + \sum_{j=1}^{p-1}-\epsilon\big(q^s + q^{-s}\big)\delta_j(\epsilon, s)G_j. \end{align*}
On the other hand, we use Lemma~\ref{actionBmoinsUnA} and the multiplication rules
\begin{gather*}
 \overset{\mathcal{X}^+(2)}{W}\!\!\!\!_{vB{^{-1}}A}v_B^{-1} \triangleright \chi^{\epsilon}_s = \sum_{\sigma \in \{\pm\}}\sum_{\ell=1}^p \lambda^{\sigma}_{\ell}(\epsilon, s) \chi^+_2\chi^{\sigma}_{\ell} + \sum_{j=1}^{p-1}\delta_j(\epsilon, s)\chi^+_2G_j\\
\hphantom{\overset{\mathcal{X}^+(2)}{W}\!\!\!\!_{vB{^{-1}}A}v_B^{-1} \triangleright \chi^{\epsilon}_s}{}
= \sum_{\sigma \in \{\pm\}} \big(\lambda^{\sigma}_2(\epsilon,s) + 2\lambda^{-\sigma}_p(\epsilon,s)\big) \chi^{\sigma}_1 + \sum_{\ell=2}^{p-2} \big(\lambda^{\sigma}_{s-1}(\epsilon,s) + \lambda^{\sigma}_{s+1}(\epsilon,s)\big) \chi^{\sigma}_{\ell}\\
\hphantom{\overset{\mathcal{X}^+(2)}{W}\!\!\!\!_{vB{^{-1}}A}v_B^{-1} \triangleright \chi^{\epsilon}_s =}{}
 + \big(\lambda^{\sigma}_{p-2}(\epsilon,s) + 2\lambda^{\sigma}_p(\epsilon,s)\big) \chi^{\sigma}_{p-1} + \lambda^{\sigma}_{p-1}(\epsilon,s)\chi^{\sigma}_p + \frac{\delta_2(\epsilon, s)}{[2]}G_1\\
\hphantom{\overset{\mathcal{X}^+(2)}{W}\!\!\!\!_{vB{^{-1}}A}v_B^{-1} \triangleright \chi^{\epsilon}_s =}{}
 + \sum_{j=2}^{p-2}[j]\left(\frac{\delta_{j-1}(\epsilon, s)}{[j-1]} + \frac{\delta_{j+1}(\epsilon, s)}{[j+1]}\right)\!G_j + \frac{\delta_{p-2}(\epsilon, s)}{[2]}G_{p-1}.
\end{gather*}
This gives recurrence equations between the coefficients which are easily solved
\begin{gather*}
v_B^{-1} \triangleright \chi^{\epsilon}_s = \lambda(\epsilon, s) \left( \sum_{\ell=1}^{p-1}(-1)^s(-\epsilon)^{p-\ell}\big(q^{\ell s} + q^{-\ell s}\big)\big(\chi^+_{\ell} + \chi^-_{p-\ell}\big) + \chi^+_p + (-\epsilon)^p(-1)^s\chi^-_p \right)\\
\hphantom{v_B^{-1} \triangleright \chi^{\epsilon}_s =}{} +\delta(\epsilon,s)\sum_{j=1}^{p-1} (-\epsilon)^{j+1}\frac{[j][js]}{[s]}G_j.
\end{gather*}
The coefficients $\lambda(\epsilon, s) = \lambda^+_p(\epsilon, s)$ and $\delta(\epsilon,s) = \delta_1(\epsilon,s)$ are still unknown. In order to compute them by induction, we use the relation $v_B^{-1} \overset{\mathcal{X}^+(2)}{W}\!\!\!\!\!_B = \overset{\mathcal{X}^+(2)}{W}\!\!\!\!\!_B v_B^{-1}$, which is another consequence of~\eqref{commBeta}. Before, note that
\[ \overset{\mathcal{X}^+(2)}{W}\!\!\!\!\!_B \triangleright \chi^{\epsilon}_s = \big(\chi^+_2(\chi^{\epsilon}_s)^v\big)^{v^{-1}} = \frac{v_{\mathcal{X}^{\epsilon}(s)}}{v_{\mathcal{X}^{\epsilon}(s-1)}}\chi^{\epsilon}_{s-1} + \frac{v_{\mathcal{X}^{\epsilon}(s)}}{v_{\mathcal{X}^{\epsilon}(s+1)}}\chi^{\epsilon}_{s+1} = -\epsilon q^{-s+\frac{1}{2}}\chi^{\epsilon}_{s-1} - \epsilon q^{s+\frac{1}{2}}\chi^{\epsilon}_{s+1}, \]
with $1 \leq s \leq p-1$ and the convention that $\chi^{\pm}_0=0$. It follows that
\begin{gather}\label{eqInductionBeta}
v_B^{-1} \triangleright \chi^{\epsilon}_{s+1} = -\epsilon q^{-s-\frac{1}{2}}\overset{\mathcal{X}^+(2)}{W}\!\!\!\!\!_B \triangleright \big(v_B^{-1} \triangleright \chi^{\epsilon}_{s}\big) - q^{-2s}v_B^{-1} \triangleright \chi^{\epsilon}_{s-1}.
\end{gather}
Due to \eqref{rubanCentre}, \eqref{actionCentreSLF} and the multiplication rules, we have
 \begin{align*}
\overset{\mathcal{X}^+(2)}{W}\!\!\!\!\!_B \triangleright \big(v_B^{-1} \triangleright \chi^{\epsilon}_{s}\big) &= \big( \chi^+_2 \big(v_B^{-1} \triangleright \chi^{\epsilon}_{s}\big)^v\big)^{v^{-1}}\\
&= \frac{v_{\mathcal{X}^+(p-1)}}{v_{\mathcal{X}^+(p)}}\big( \lambda^+_{p-1}(\epsilon,s) + \hat q \delta_{p-1}(\epsilon, s)\big) \chi^+_{p} + \frac{v_{\mathcal{X}^+(2)}}{[2]}\delta_2(\epsilon, s)G_1 + \cdots,
\end{align*}
where the dots ($\cdots$) mean the remaining of the linear combination in the GTA basis. After replacement by the values found previously and insertion in relation \eqref{eqInductionBeta}, we obtain
\begin{gather*}
 \lambda(\epsilon,s+1)\chi^+_p + \delta(\epsilon, s+1)G_1 + \cdots \\
 \qquad {}= \big( q^{-(s+1)}\big(q^s+q^{-s}\big)\lambda(\epsilon,s) - q^{-2s}\lambda(\epsilon,s-1) + (-\epsilon)^{p-1}(-1)^{s-1}\hat qq^{-(s+1)}\delta(\epsilon,s) \big) \chi^+_p\\
\qquad \quad{} + \big( {-}q^{-(s+2)}\big(q^s + q^{-s}\big)\delta(\epsilon,s) - q^{-2s}\delta(\epsilon,s-1) \big) G_1 +\cdots.
\end{gather*}
These are recurrence equations. It just remains to determine the first values $\lambda(\epsilon, 1)$, $\delta(\epsilon,1)$. Observe that, since $\overline{U}_q$ is unibalanced:
\begin{gather}\label{relXi}
v_B^{-1} \triangleright \chi^+_1 = \big(\varphi_{v^{-1}} \big(\chi^+_1\big)^v\big)^{v^{-1}} = \mu^l(v)^{-1}\mu^l\big(K^{p-1}v ?\big)^{v^{-1}} = \mu^l(v)^{-1}\mu^r\big(K^{p+1}?\big).
\end{gather}
In \cite[Section 4.3]{F} the decomposition of $\mu^r(K^{p+1}?)$ in the GTA basis has been found (when $\mu^r$ is suitably normalized). Thanks to this, we obtain
\[ v_B^{-1} \triangleright \chi^+_1 = \lambda(+, 1)\chi^+_p + \delta(+,1)G_1 + \dots = \xi(-1)^{p-1}\chi^+_p - \xi G_1 + \cdots \]
and
 \begin{align*}
v_B^{-1} \triangleright \chi^-_1 &= v_{\mathcal{X}^-(1)} v_B^{-1}\overset{\mathcal{X}^-(1)}{W}\!\!\!\!\!_B \triangleright \chi^+_1 = v_{\mathcal{X}^-(1)} \overset{\mathcal{X}^-(1)}{W}\!\!\!\!\!_B\triangleright v_B^{-1} \triangleright \chi^+_1 = v_{\mathcal{X}^-(1)}\big( \chi^-_1 \big(v_B^{-1} \triangleright \chi^+_1\big)^v \big)^{v^{-1}}\\
&= -\xi\chi^+_p + \xi G_1 + \dots = \lambda(-,1)\chi^+_p + \delta(-,1)G_1 + \cdots.
\end{align*}
The scalar $\xi$ does not depend on the choice of $\mu^r$ thanks to the factor $\mu^l(v)^{-1} = \mu^l \circ S(v)^{-1} = \mu^r(v)^{-1}$ in~\eqref{relXi}. Thanks to the formulas \cite{FGST} for $\mu^r$ and $v$ in the PBW basis we compute $\mu^r(v)$ and this gives the value of $\xi$. We are now in position to solve the recurrence equations. It is easy to check that the solutions are
\[ \delta(\epsilon, s) = \xi\epsilon (-1)^sq^{-(s^2-1)}[s], \qquad \lambda(\epsilon, s) = \xi\epsilon(-\epsilon)^{p-1}s q^{-(s^2-1)}. \]
We now proceed with the proof of the formula for $G_{s'}$. Relation \eqref{commBeta} implies $v_B^{-1}H_B^1 = H_B^1v_B^{-1}$ (recall~\eqref{baseCentreGTA}). By~\eqref{rubanCentre}, \eqref{actionCentreSLF} and the multiplication rules, we have on the one hand
\[ v_B^{-1}H_B^1 \triangleright \chi^+_s = v_B^{-1} \triangleright \big(G_1\big(\chi^+_s\big)^v \big)^{v^{-1}} = [s] v_B^{-1} \triangleright G_s - \hat q(p-s) v_B^{-1} \triangleright \chi^+_s + \hat q s v_B^{-1} \triangleright \chi^-_{p-s}, \]
whereas on the other hand
\[ H_B^1v_B^{-1} \triangleright \chi^+_s = \big(G_1 \big(v_B^{-1} \triangleright \chi^+_s\big)^v\big)^{v^{-1}} = \hat q p \sum_{j=1}^{p-1} \delta_j(+,s) \left(G_j - \hat q \frac{p-j}{[j]} \chi^+_j + \hat q\frac{j}{[j]}\chi^-_{p-j}\right). \]
Equalizing \looseness=1 both sides and inserting the previously found values, we obtain the desired for\-mula.
\end{proof}

\begin{Remark}The guiding principle of the previous computations was that the mutiplication of two symmetric linear forms in the GTA basis is easy when one of them is $\chi^+_2$, $\chi^-_1$ or $G_1$ (see~\cite[Section~5]{F}), and that all the formulas can be derived from $v_B^{-1} \triangleright \chi^+_1$ using only such products.
\end{Remark}

Recall that the standard representation $\mathbb{C}^2$ of $\mathrm{SL}_2(\mathbb{Z}) = \mathrm{MCG}(\Sigma_{1,0})$ is defined by
\[ \tau_a \mapsto
\begin{pmatrix}
1 & 0\\
-1 & 1
\end{pmatrix}, \qquad
\tau_b \mapsto
\begin{pmatrix}
1 & 1\\
0 & 1
\end{pmatrix}.\]

\begin{Lemma}\label{lemmePropRepSL2Z}Let $V$ be a (projective) representation of $\mathrm{SL}_2(\mathbb{Z})$ which admits a basis $\left(x_s, y_s\right)$ such that
 \begin{alignat*}{3}
& \tau_a x_s = \sum_{\ell} a_{\ell}(s)x_{\ell}, \qquad && \tau_b x_s = \sum_{\ell} b_{\ell}(s)(x_{\ell} + y_{\ell}), & \\
& \tau_a y_s = \sum_{\ell} a_{\ell}(s)(y_{\ell} - x_{\ell}), \qquad && \tau_b y_s = \sum_{\ell} b_{\ell}(s)y_{\ell}.&
\end{alignat*}
Then there exists a $($projective$)$ representation $W$ of $\mathrm{SL}_2(\mathbb{Z})$ such that $V \cong \mathbb{C}^2 \otimes W$. More precisely, $W$ admits a basis $(w_s)$ such that
\[
\tau_a w_s = \sum_{\ell} a_{\ell}(s)w_{\ell}, \qquad \tau_b w_s = \sum_{\ell} b_{\ell}(s) w_{\ell}.
 \]
\end{Lemma}
\begin{proof}It is easy to check that the formulas for $\tau_a w_s$ and $\tau_b w_s$ indeed define a $\mathrm{SL}_2(\mathbb{Z})$-re\-pre\-sen\-ta\-tion on~$W$. Let $(e_1, e_2)$ be the canonical basis of $\mathbb{C}^2$. Then
\[ e_1 \otimes w_s \mapsto y_s, \qquad e_2 \otimes w_s \mapsto x_s \]
is an isomorphism which intertwines the $\mathrm{SL}_2(\mathbb{Z})$-action.
\end{proof}

\begin{Theorem}\label{thDecRep}The $(p+1)$-dimensional subspace $\mathcal{V} = \vect \big(\chi^+_s + \chi^-_{p-s}, \chi^{\pm}_p \big)_{1 \leq s \leq p-1}$ is stable under the $\mathrm{SL}_2(\mathbb{Z})$-action of Theorem~{\rm \ref{actionSL2ZArike}}. Moreover, there exists a $(p-1)$-dimensional projective representation~$\mathcal{W}$ of $\mathrm{SL}_2(\mathbb{Z})$ such that
\[ \SLF\big(\overline{U}_q\big) = \mathcal{V} \oplus \big(\mathbb{C}^2 \otimes \mathcal{W}\big). \]
\end{Theorem}
\begin{proof} By \cite[Corollary~5.1]{F}, $\mathcal{V}$ is an ideal of $\SLF\big(\overline{U}_q\big)$. It is easy to see that $\mathcal{V}$ is moreover stable under the action \eqref{actionCentreSLF} of $\mathcal{Z}\big(\overline{U}_q\big)$. Thus we deduce without any computation that $\mathcal{V}$ is $\mathrm{SL}_2(\mathbb{Z})$-stable. Next, in view of the formulas in Theorem~\ref{actionSL2ZArike}, it is natural to define
\[ x_s = \hat q \frac{p-s}{[s]}\chi^+_s - \hat q \frac{s}{[s]}\chi^-_{p-s}, \qquad y_s = G_s - x_s. \]
Then
\begin{alignat*}{3}
& v_A^{-1} \triangleright x_s = v_{\mathcal{X}^+(s)}^{-1} x_s, \qquad && v_B^{-1} \triangleright x_s = \xi (-1)^{s}q^{-(s^2-1)}\frac{\hat q p}{[s]}\sum_{j=1}^{p-1}(-1)^{j+1}[j][js]\!\left(x_j + y_j\right), &\\
& v_A^{-1} \triangleright y_s = v_{\mathcal{X}^+(s)}^{-1}(y_s - x_s) , \qquad && v_B^{-1} \triangleright y_s = \xi (-1)^{s}q^{-(s^2-1)}\frac{\hat q p}{[s]}\sum_{j=1}^{p-1}(-1)^{j+1}[j][js] y_j.&
\end{alignat*}
The result follows from Lemma~\ref{lemmePropRepSL2Z}.
\end{proof}

The structure of the Lyubashenko--Majid representation on $\mathcal{Z}\big(\overline{U}_q\big)$ has been described in~\cite{FGST}. By Theorem~\ref{EquivalenceLMandSLF}, this projective representation is equivalent to the one constructed here and Theorem~\ref{thDecRep} is in perfect agreement with the decomposition given in~\cite{FGST}.

Note that the subspace $\mathcal{V}$ is generated by the characters of the finite-dimensional projective $\overline{U}_q$-modules: $\mathcal{V} = \mathrm{vect}\big( \chi^P \big)_{P \in \mathrm{Proj}\big(\overline{U}_q\big)}$. We precise that, explicitly, the projective representa\-tion~$\mathcal{W}$ has a basis $(w_s)_{1 \leq s \leq p-1}$ such that the action of $\tau_a$, $\tau_b$ is
\[ \tau_a w_s = v_{\mathcal{X}^+(s)}^{-1} w_s, \qquad \tau_b w_s = \xi (-1)^{s}q^{-(s^2-1)}\frac{\hat q p}{[s]}\sum_{j=1}^{p-1}(-1)^{j+1}[j][js] w_j. \]

\subsection[A conjecture about the representation of $\mathcal{L}_{1,0}^{\mathrm{inv}}\big(\overline{U}_q\big)$ on $\SLF\big(\overline{U}_q\big)$]{A conjecture about the representation of $\boldsymbol{\mathcal{L}_{1,0}^{\mathrm{inv}}\big(\overline{U}_q\big)}$ on $\boldsymbol{\SLF\big(\overline{U}_q\big)}$}\label{sectionConjecture}

Another natural (but harder) question is to determine the structure of $\SLF\big(\overline{U}_q\big)$ under the action of $\mathcal{L}_{1,0}^{\mathrm{inv}}\big(\overline{U}_q\big)$. As mentioned in the proof of Theorem~\ref{thDecRep}, the subspace
\[ \mathcal{V}= \vect\big(\chi^+_s + \chi^-_{p-s}, \chi^{\pm}_p\big)_{1 \leq s \leq p-1}\] is quite ``stable''. We propose the following conjecture.
\begin{Conjecture}
$\mathcal{V}$ is a $\mathcal{L}_{1,0}^{\mathrm{inv}}\big(\overline{U}_q\big)$-submodule of $\SLF\big(\overline{U}_q\big)$.
\end{Conjecture}
In order to prove this conjecture one needs to find a basis or a generating set of $\mathcal{L}_{1,0}^{\mathrm{inv}}\big(\overline{U}_q\big)$, and then to show that $\mathcal{V}$ is stable under the action of the basis elements (or of the generating elements). Both tasks are difficult.

Let us mention that since $\mathcal{V}$ is an ideal of $\SLF\big(\overline{U}_q\big)$, it is stable under the action of $z_{A}$, $z_{B}$ and $z_{B^{-1}A}$ for all $z \in \mathcal{Z}\big(\overline{U}_q\big) = \mathcal{L}^{\mathrm{inv}}_{0,1}\big(\overline{U}_q\big)$ (see Proposition~\ref{actionAB} and Lemma~\ref{actionBmoinsUnA}). Also recall the wide family of invariants given in \eqref{FamilleInvL10}; we can try to test the conjecture with them. A long computation (which is not specific to $\overline{U}_q$) shows that
\[ \mathrm{tr}_{12}\Big(\overset{I \otimes J}{g}\!\!\!_{12}\Phi_{12}\overset{I}{A}_1\overset{IJ}{(R')}_{12}\overset{J}{B}_2\overset{IJ}{R}_{12}\Big) \triangleright \chi^K = v_J \mathrm{tr}_{13}\Big( \overset{I \otimes K}{T}\!\!\!\!_{13} \overset{I \otimes K}{v}{^{-1}}\!\!\!\!\!\!\!\!\!_{13} \:\:\: s_{IJ,K}(\Phi)_{13} \Big), \]
where $\chi^K$ is the character of $K$, $\overset{J}{v} = v_J\mathrm{id}$ (note that we may assume that $I$, $J$, $K$ are simple modules) and
\[ s_{IJ,K}(\Phi) = \mathrm{tr}_{2}\Big( \overset{J}{g}_2 \overset{JK}{R}_{23}\Phi_{12}\overset{JK}{(R')}_{23} \Big). \]
Proving that $\mathcal{V}$ is stable under the action of these invariants amounts to show symmetry pro\-perties between $s_{IJ, \mathcal{X}^{+}(s)}$ and $s_{IJ, \mathcal{X}^{-}(p-s)}$ for all simple $\overline{U}_q$-modules~$I$,~$J$. We have checked that it is true if $\Phi = \mathrm{id}_{I \otimes J}$ (in this case $ s_{IJ,K}(\mathrm{id}_{I \otimes J}) = s_{J,K} \mathrm{id}_{I \otimes K}$, where $s_{J,K}$ is the usual $S$-matrix) for all simple modules $I$,~$J$, and also that it holds for $I = J = \mathcal{X}^+(2)$ with every~$\Phi$.

\begin{Proposition}\label{propStructureSLFSousL10inv}\quad
\begin{enumerate}\itemsep=0pt
\item[$1)$] $\SLF\big(\overline{U}_q\big)$ is indecomposable as a $\mathcal{L}_{1,0}^{\mathrm{inv}}\big(\overline{U}_q\big)$-module.
\item[$2)$] Assume that the Conjecture holds. Then the $\mathcal{L}_{1,0}^{\mathrm{inv}}\big(\overline{U}_q\big)$-modules $\mathcal{V}$ and $\SLF\big(\overline{U}_q\big) / \mathcal{V}$ are simple. It follows that $\SLF\big(\overline{U}_q\big)$ has length $2$ as a $\mathcal{L}_{1,0}^{\mathrm{inv}}\big(\overline{U}_q\big)$-module.
\end{enumerate}
\end{Proposition}
\begin{proof}These are basically consequences of \eqref{actionCentreSLF} and of the multiplication rules in the GTA basis \cite[Section 5]{F}. To avoid particular cases, let $\chi^{\epsilon}_0 = 0$, $\chi^{\epsilon}_{p+1} = \chi^{-\epsilon}_1$, $\chi^{\epsilon}_{-1} = \chi^{-\epsilon}_{p-1}$ and $e_{-1} = e_{p+1} = 0$.

1) Observe that $\SLF\big(\overline{U}_q\big)$ is generated by $\chi^+_1 = \varepsilon$ as a $\mathcal{L}_{1,0}^{\mathrm{inv}}\big(\overline{U}_q\big)$-module: this is a general fact which follows immediately from Lemma~\ref{actionBmoinsUnA}. Explicitly (see \eqref{valueHatAlphaBeta})
\[ \overset{\mathcal{X}^{\epsilon}(s)}{W}\!\!\!\!_{vB^{-1}A} \triangleright \chi^+_1 = \chi^{\epsilon}_s\chi^+_1 = \chi^{\epsilon}_s, \qquad H^s_{vB^{-1}A} \triangleright \chi^+_1 = G_s \chi^+_1 = G_s. \]
Write $\SLF\big(\overline{U}_q\big) = U_1 \oplus U_2$. At least one of the two subspaces $U_1, U_2$ necessarily contains an element of the form $\varphi = G_1 + \sum_{i \neq 1} \lambda_i G_i + \sum_{j, \epsilon} \eta^{\epsilon}_j \chi^{\epsilon}_j$; assume that it is $U_1$. Then $\big(w_1^+\big)_A \triangleright \varphi = \varphi^{w_1^+} = \chi^+_1 \in U_1$ thanks to~\eqref{actionCentreSLF}. It follows that $U_1 = \SLF\big(\overline{U}_q\big)$ and $U_2 = \{0\}$, as desired.

2) Let $0 \neq U \subset \mathcal{V}$ be a submodule, and let $\psi = \sum_{j=0}^p \lambda_{j}\big(\chi^+_j + \chi^-_{p-j}\big) \in U$ with $\lambda_s \neq 0$ for some~$s$. Then using Proposition~\ref{actionAB} and~\eqref{actionCentreSLF}, we get $(e_s)_A \triangleright \psi = \psi^{e_s} = \lambda_{s}\big(\chi^+_s + \chi^-_{p-s}\big)$, and thus $\chi^+_s + \chi^-_{p-s} \in U$. Apply $\overset{\mathcal{X}^+(2)}{W}\!\!\!\!_{vB^{-1}A}$ (we use Lemma~\ref{actionBmoinsUnA} and Proposition~\ref{antipodeSLF})
\[ \overset{\mathcal{X}^+(2)}{W}\!\!\!\!_{vB^{-1}A} \triangleright \big(\chi^+_s + \chi^-_{p-s}\big) = \chi^+_2\big(\chi^+_s + \chi^-_{p-s}\big) = \big(\chi^+_{s-1} + \chi^-_{p-s+1}\big) + \big(\chi^+_{s+1} + \chi^-_{p-s-1}\big). \]
Hence
\begin{gather*}
 (e_{s-1})_A\overset{\mathcal{X}^+(2)}{W}\!\!\!\!_{vB^{-1}A} \triangleright \big(\chi^+_s + \chi^-_{p-s}\big) = \chi^+_{s-1} + \chi^-_{p-s+1},\\
 (e_{s+1})_A\overset{\mathcal{X}^+(2)}{W}\!\!\!\!_{vB^{-1}A} \triangleright \big(\chi^+_s + \chi^-_{p-s}\big) = \chi^+_{s+1} + \chi^-_{p-s-1}.
\end{gather*}
It follows that $\chi^+_{s-1} + \chi^-_{p-s+1}, \chi^+_{s+1} + \chi^-_{p-s-1} \in U$. Continuing like this, one gets step by step that all the basis vectors belong to~$U$, hence $U = \mathcal{V}$.

Next, let $\overline{G}_s$ and $\overline{\chi}^+_s$ be the classes of $G_s$ and $\chi^+_s$ modulo $\mathcal{V}$ (with $\overline{\chi}^+_0 = \overline{\chi}^+_p = 0$). Let $0 \neq U \subset \SLF\big(\overline{U}_q\big) / \mathcal{V}$ be a submodule and $\omega = \sum_{j = 1}^{p-1} \nu_{j}\overline{G}_{j} + \sigma_{j}\overline{\chi}^+_{j} \in U$ be non-zero. If all the $\nu_j$ are $0$, then there exists $\sigma_s \neq 0$ and $(e_s)_A \triangleright \omega = \sigma_s\overline{\chi}^+_s \in U$. If one of the $\nu_j$, say $\nu_s$, is non-zero, then $(w^+_s)_A \triangleright \omega = \nu_s \overline{\chi}^+_s \in U$. In both cases we get $\overline{\chi}^+_s \in U$. Now we proceed as previously
\[ (e_{s-1})_A\overset{\mathcal{X}^+(2)}{W}\!\!\!\!_{vB^{-1}A} \triangleright \overline{\chi}^+_s = \overline{\chi}^+_{s-1}, \qquad (e_{s+1})_A\overset{\mathcal{X}^+(2)}{W}\!\!\!\!_{vB^{-1}A} \triangleright \overline{\chi}^+_s = \overline{\chi}^+_{s+1}. \]
Thus we get step by step that $\overline{\chi}^+_{j} \in U$ for all $j$. Apply $H^1_{vB^{-1}A}$
\[ H_{vB^{-1}A}^1 \triangleright \overline{\chi}^+_{j} = G_1\chi^+_{j} + \mathcal{V} = [j]G_j + \mathcal{P}. \]
It follows that $\overline{G}_j \in U$ for all~$j$, and thus $U = \SLF\big(\overline{U}_q\big) / \mathcal{V}$ as desired.
\end{proof}

In order to determine the structure of $\SLF\big(\overline{U}_q\big)$ if the Conjecture is true, it will remain to determine whether $\mathcal{V}$ is a direct summand of $\SLF\big(\overline{U}_q\big)$ or not.

\subsection*{Acknowledgements}
I am grateful to my advisors, St\'ephane Baseilhac and Philippe Roche, for their regular support and their useful remarks. I thank the referees for carefully reading the manuscript and for many valuable comments which improved the paper.

\pdfbookmark[1]{References}{ref}
\LastPageEnding

\end{document}